\documentclass[a4paper,centertags,oneside,12pt]{amsart}
\usepackage{mathrsfs}
\usepackage{appendix}
\usepackage{amssymb}
\usepackage{fancyhdr}
\usepackage{charter}
\usepackage{typearea}
\usepackage{pdfsync}
\usepackage{mathrsfs}
\usepackage[a4paper,top=3cm,bottom=3cm,left=2cm,right=2cm]{geometry}

\usepackage{enumitem}
\setlist[1]{itemsep=5pt}
\newcommand{\comment}[1]{}

\makeatletter
      \def\@setcopyright{}
      \def\serieslogo@{}
      \makeatother

\newcommand{\Complex}{\mathbb C}
\newcommand{\Real}{\mathbb R}
\newcommand{\N}{\mathbb N}
\newcommand{\ddbar}{\overline\partial}
\newcommand{\pr}{\partial}
\newcommand{\ol}{\overline}
\newcommand{\Td}{\widetilde}
\newcommand{\norm}[1]{\left\Vert#1\right\Vert}
\newcommand{\abs}[1]{\left\vert#1\right\vert}
\newcommand{\rabs}[1]{\left.#1\right\vert}
\newcommand{\set}[1]{\left\{#1\right\}}
\newcommand{\To}{\rightarrow}

 \def\cC{\mathscr{C}}

\theoremstyle{plain}
\newtheorem{thm}{Theorem}[section]
\newtheorem{cor}[thm]{Corollary}

\newtheorem{lem}[thm]{Lemma}

\theoremstyle{definition}
\newtheorem{defn}[thm]{Definition}

\theoremstyle{remark}
\newtheorem{rem}[thm]{Remark}
\newtheorem{ex}[thm]{Example}
\newtheorem{que}[thm]{Question}

\numberwithin{equation}{section}

\renewcommand{\labelenumi}{(\alph{enumi})}

\begin{document}
\title[]{On the singularities of the Szeg\H{o} projections on lower energy forms
}
\author[]{Chin-Yu Hsiao}
\address{Institute of Mathematics, Academia Sinica, 6F, Astronomy-Mathematics Building,
No.1, Sec.4, Roosevelt Road, Taipei 10617, Taiwan}
\thanks{The first author was partially supported by Taiwan Ministry of Science of Technology project 
103-2115-M-001-001, the DFG project MA 2469/2-2 and
funded through the Institutional Strategy of
the University of Cologne within the German Excellence Initiative}
\email{chsiao@math.sinica.edu.tw or chinyu.hsiao@gmail.com}
\author[]{George Marinescu}
\address{Universit{\"a}t zu K{\"o}ln,  Mathematisches Institut,
    Weyertal 86-90,   50931 K{\"o}ln, Germany\\
    \& Institute of Mathematics `Simion Stoilow', Romanian Academy,
Bucharest, Romania}
\thanks{The second author partially supported by the DFG projects SFB/TR 12, MA 2469/2-1 and Universit\'e Paris 7}
\email{gmarines@math.uni-koeln.de}
\dedicatory{Dedicated to our teachers Professors Louis Boutet de Monvel and Johannes Sj\"ostrand}

\setlength{\headheight}{14pt}
\pagestyle{fancy}
\lhead{\itshape{Chin-Yu Hsiao \& George Marinescu}}
\rhead{\itshape{On the singularities of the Szeg\H{o} projections}}
\cfoot{\thepage}

\begin{abstract}
Let $X$ be an abstract not necessarily compact orientable CR manifold of dimension $2n-1$, $n\geqslant2$. Let
$\Box^{(q)}_{b}$ be the Gaffney extension of Kohn Laplacian for $(0,q)$--forms. We show that the spectral 
function of $\Box^{(q)}_b$ admits a full asymptotic expansion on the non-degenerate part of the Levi form. 
As a corollary, we deduce that if $X$ is compact and the Levi form is non-degenerate of constant signature on $X$, 
then the spectrum of $\Box^{(q)}_b$ in $]0,\infty[$ consists of point eigenvalues of finite multiplicity. 
Moreover, we show that a certain microlocal conjugation of the associated Szeg\H{o} kernel admits an asymptotic expansion 
under a local closed range condition. As applications, we establish the Szeg\H{o} kernel asymptotic expansions 
on some weakly pseudoconvex CR manifolds and on CR manifolds with transversal CR $S^1$ actions. 
By using these asymptotics, we establish some local embedding theorems on CR manifolds and we give 
an analytic proof  of a theorem of Lempert asserting that a compact strictly pseudoconvex 
CR manifold of dimension three with a transversal CR $S^1$ action can be CR embedded 
into $\Complex^N$, for some $N\in\mathbb N$.
\end{abstract}

\maketitle \tableofcontents

\section{Introduction and statement of the main results} \label{s-intro}

Let $(X, T^{1,0}X)$ be a CR manifold of hypersurface type and dimension $2n-1$, $n\geq2$.
Let $\Box^{(q)}_b$ be the Gaffney extension of the Kohn Lalpacian acting on $(0,q)$ forms.
The orthogonal projection $\Pi^{(q)}:L^2_{(0,q)}(X)\To {\rm Ker\,}\Box^{(q)}_b$ onto ${\rm Ker\,}\Box^{(q)}_b$
is called the Szeg\H{o} projection, while its distribution kernel $\Pi^{(q)}(x,y)$ is called the Szeg\H{o} kernel.
The study of the Szeg\H{o} projection and kernel is a classical subject in several complex variables and CR geometry.

When the Levi form satisfies condition $Y(q)$ on $X$ (see Definition~\ref{d-suIIa}), then Kohn's subelliptic
estimates with loss of one dervative for the solutions of $\Box^{(q)}_b u=f$ hold, cf.~\cite{CS01,FK:72,Koh64},
and hence $\Pi^{(q)}$ is a smoothing operator.
When condition $Y(q)$ fails, one is interested in the singularities of the Szeg\H{o} kernel $\Pi^{(q)}(x,y)$.

A very important case is when $X$ is a compact strictly pseudoconvex
CR manifold (in this case $Y(0)$ fails). Assume first that $X$ is the boundary of a strictly pseudoconvex domain.
Boutet de Monvel-Sj\"ostrand~\cite{BouSj76} showed that $\Pi^{(0)}(x,y)$
is a Fourier integral operator with complex phase. In particular,
$\Pi^{(0)}(x,y)$ is smooth outside the diagonal of $X\times X$
and there is a precise description of the singularity on the diagonal $x=y$,
where $\Pi^{(0)}(x,x)$ has a certain asymptotic expansion. 

The Boutet de Monvel-Sj\"ostrand description of the Szeg\H{o} kernel had a profound impact
in many research areas, especially through \cite{BG81}: several complex variables, 
symplectic and contact geometry, geometric quantization, K\"ahler geometry, semiclassical analysis,
quantum chaos, cf.\ \cite{BMS,BU,BPU,BdM:86,BdM:88,Ca99,Charles03,Do01,Engl:02,Gu89,Hir06,Ka77,
LuTian:04,Ma10,ShZ99,Zelditch98},
to quote just a few. These ideas also partly motivated the introduction of alternative approaches, 
see \cite{Ma10,MM06,MM08a,MM07}.


From the works of Boutet de Monvel \cite{BdM1:74}, Boutet de Monvel-Sj\"ostrand \cite{BouSj76},
Harvey-Lawson \cite{HaLa75}, Burns \cite{Bur:79} and Kohn \cite{Koh85,Koh86}
follows that the conditions below are equivalent for a compact strictly pseudoconvex
CR manifold $X$, $\dim_{\mathbb R}X\geqslant3$:
\begin{enumerate}
\renewcommand{\labelenumi}{\rm(\alph{enumi})}
\item $X$ is embeddable in the Euclidean space $\mathbb{C}^N$, for $N$ sufficiently large;
\item $X$ bounds a strictly pseudoconvex complex manifold;
\item The Kohn Laplacian $\Box^{(0)}_b$ on functions of $X$
has closed range in $L^2$.
\end{enumerate}
Therefore the description of the Szeg\H{o} kernel given by \cite{BouSj76} holds for CR manifolds satisfying
the equivalent conditions (a)--(c).
Moreover, if $X$ an abstract compact strictly pseudoconvex of dimension $\geq5$, then $X$ satisfies condition (a),
by a theorem of Boutet de Monvel \cite{BdM1:74}. 
Among embeddable strictly pseudoconvex CR manifolds of dimension three there are those
carrying interesting geometric structures, such as transverse $S^1$ actions, 
 cf.\ \cite{BlDu91,Eps92,Lem92},
conformal structures, cf.\ \cite{Biq02}, or Sasakian structures, cf.\ \cite{MY07}.

The first author ~\cite{Hsiao08} showed that if the Levi form is non-degenerate and
$\Box^{(q)}_b$ has
$L^2$ closed range for some $q\in\set{0,1,\ldots,n-1}$, then $\Pi^{(q)}(x,y)$ is a complex Fourier integral operator.
Therefore the study of the singularities of the Szeg\H{o} kernel is closely related to the 
closed range property of the Kohn Laplacian.

Kohn \cite{Koh86} proved that if $Y(q)$ fails but $Y(q-1)$ and $Y(q+1)$ hold on a compact CR manifold $X$,
then $\Box^{(q)}_b$ has $L^2$ closed
range. In this case the result of \cite{Hsiao08} applies and we can describe the Szeg\H{o} kernel 
$\Pi^{(q)}(x,y)$.
On the negative side, Burns \cite{Bur:79} showed that the closed range property fails in the case of the
non-embeddable exemples of Grauert, Andreotti-Siu and Rossi \cite{Gra:94,AS:70,Ros:65}. 

Beside Kohn's criterion, it is very difficult to determine when $\Box^{(q)}_b$ has $L^2$ closed range. Let's see a simple
example. Let $[z]=[z_1,\ldots,z_N]$ be the homogeneous coordinates of $\Complex\mathbb P^{N-1}$. Consider
$1\leq m\leq N-1$. Put
\begin{equation}\label{e-gue140303}
X:=\Big\{[z_1,\ldots,z_N]\in\Complex\mathbb P^{N-1};\,\sum_{j=1}^N\lambda_j\abs{z_j}=0\Big\},
\end{equation}
where $\lambda_j<0$ for $1\leq j\leq m$ and  $\lambda_j>0$ for $m+1\leq j\leq N$.
Then, $X$ is a compact CR manifold of dimension $2(N-1)-1$ with CR structure
$T^{1,0}X:=T^{1,0}\Complex\mathbb
P^{N-1}\cap\Complex TX$. It is straightforward to see that the Levi form has
exactly $m-1$ negative eigenvalues
and $N-m-1$ positive eigenvalues at every point of $X$.
Thus, when $q=m-1$, $N-m-1=q+1$, $Y(q)$ and $Y(q+1)$ fail and Kohn's
criterion does not work in this case. Even in this simple example,
it doesn't follow from Kohn's criterion that $\Box^{(q)}_b$ has $L^2$ closed range.
We are lead to ask the following questions:
\begin{que}\label{q-gue140303I}
Let $X$ be compact CR manifold whose Levi form is non-degenerate of signature $(n_-,n_+)$. 
Assume $n_+\in\{n_--1,n_-+1\}$. When does $\Box^{(q)}_b$ have $L^2$ closed range for $q=n_-$? 
Note that in this case, $Y(q)$ and $Y(q+1)$ fail.
\end{que}
This question was asked by Boutet de Monvel. We will introduce another condition, called $W(q)$ (see Definition \ref{d-gue140223})
which applied for CR manifolds with $S^1$ action shows that $\Box^{(q)}_b$ has $L^2$ closed range
in the situation of Question \ref{q-gue140303I}, see Theorem \ref{t-gue140223III}. In particular,
the Kohn Laplacian on the manifold $X$ in the examples \eqref{e-gue140303} has closed range.
\begin{que}\label{q-gue140303II}
Let $X$ be a not-necessarily compact CR manifold and the Levi form can be degenerate 
(for example $X$ is weakly pseudoconvex). Assume that $\Box^{(q)}_b$ has $L^2$ closed range.
Does the Szeg\H{o} kernel $\Pi^{(q)}(x,y)$ admit an asymptotic expansion on the set where the Levi form is non-degenerate? 
\end{que}

\begin{que}\label{q-gue140303III}
Find a natural local analytic condition (weaker than $L^2$ closed range condition) which implies
 that the Szeg\H{o} kernel admits a local
asymptotic expansion. 
\end{que}

Without any regularity
assumption, ${\rm Ker\,}\Box^{(q)}_b$ could be trivial and therefore we consider the spectral projections
$\Pi^{(q)}_{\leq\lambda}:=E([0,\lambda])$, for $\lambda>0$, where $E$ denotes the
spectral measure of $\Box^{(q)}_{b}$. 
\begin{que}\label{q-gue140303IV}
Is $\Pi^{(q)}_{\leq\lambda}$ a Fourier integral operator, for every $\lambda>0$?  
\end{que}\
The purpose of this work is to answer 
Questions~\ref{q-gue140303I}-\ref{q-gue140303IV}. Our first main results
tell us that on the
non-degenerate part of the Levi form, $\Pi^{(q)}_{\leq\lambda}$ is a Fourier integral operator with complex
phase, for every $\lambda>0$, and $\Pi^{(q)}_{(\lambda_1,\lambda_2]}:=E((\lambda_1,\lambda_2])
$ is a smoothing operator, for every $0<\lambda_1<\lambda_2$. 
\begin{thm}\label{t-gue140305_a}
Let $X$ be a CR manifold whose Levi form is non-degenerate of constant signature $(n_-,n_+)$
at each point of an open set $D\Subset X$.  Then for every $\lambda>0$ the restriction 
of the spectral projector $\Pi^{(q)}_{\leq\lambda}$ to $D$
is a smoothing operator for $q\notin\set{n_-,n_+}$ and
is a Fourier integral operator with complex phase for $q\in\set{n_-,n_+}$.
Moreover, in the latter case, the singularity
of $\Pi^{(q)}_{\leq\lambda}$ does not depend on $\lambda$,
in the sense that the difference $\Pi^{(q)}_{\leq\lambda_1}-\Pi^{(q)}_{\leq\lambda_2}$
is smoothing on $D$ for any $\lambda_1,\lambda_2>0$.
\end{thm}
The Fourier integral operators $A$ considered here (and in this paper) have kernels of the form
\begin{equation}\label{e:fio}
A(x,y)=\int^{\infty}_{0}\!\! e^{i\varphi_-(x, y)t}s_-(x, y, t)dt+\int^{\infty}_{0}\!\! e^{i\varphi_+(x, y)t}s_+(x, y, t)dt+R(x,y)\,,
\end{equation}
where the integrals are oscillatory integrals, $\varphi_-$, $\varphi_+$ are complex phase functions, $s_-$, $s_+$ classical symbols of type $(1,0)$ and order $n-1$, $s_-=0$ if $q\neq n_-$,  $s_+=0$ if $q\neq n_+$ and $R$ a smooth function,
see Section \ref{s:prelim} for a precise definition.

A detailed version of Theorem \ref{t-gue140305_a} will be given in Theorems \ref{t-gue140305_b}, 
\ref{t-gue140207} and \ref{t-gue140211}.
As a corollary of Theorem~\ref{t-gue140305_a}, we deduce:
\begin{cor}\label{c-gue140211I}
Let $X$ be a CR manifold of dimension $2n-1$, whose Levi form is non-degenerate of constant signature 
at each point of an open set $D\Subset X$. Let $0\leq q\leq n-1$ and $0<\lambda_1<\lambda_2$. 
Then, the projector $\Pi^{(q)}_{(\lambda_1,\lambda_2]}$ is
a smoothing operator on $D$.
In particular, if $X$ is compact and the Levi form is non-degenerate of constant signature 
on $X$, then the projector $\Pi^{(q)}_{(\lambda_1,\lambda_2]}$ is
a smoothing operator on $X$.
\end{cor}
As a consequence, we deduce that if $X$ is compact and the Levi form is non-degenerate of constant signature on 
$X$, then the spectrum of $\Box^{(q)}_b$ in
$]0,\infty[$ consists of point eigenvalues of finite multiplicity. Burns-Epstein~\cite[Theorem\,1.3]{BE90} proved
that if $X$ is compact, strictly pseudoconvex of dimension three, then the spectrum of $\Box^{(0)}_b$ in $]0,\infty[$ consists of
point eigenvalues of finite multiplicity. We generalize their result to any $q\in\set{0,1,\ldots,n-1}$ and any dimension.
\begin{thm}\label{t-gue140211Im}
We assume that $X$ is compact and the Levi form is non-degenerate of constant signature $(n_-,n_+)$ on $X$. 
Fix $q\in\set{0,1,\ldots,n-1}$. Then, for any $\mu>0$, ${\rm Spec\,}\Box^{(q)}_b\cap\,[\mu,\infty[$ 
is a discrete subset of $\Real$, any $\nu\in{\rm Spec\,}\Box^{(q)}_b$ with $\nu>0$
is an eigenvalue of $\Box^{(q)}_b$ and the eigenspace
$H^q_{b,\nu}(X):=\big\{u\in{\rm Dom\,}\Box^{(q)}_b;\, \Box^{(q)}_bu=\nu u\big\}$
is finite dimensional with $H^q_{b,\nu}(X)\subset\Omega^{0,q}(X)$.
\end{thm}
If $\Box^{(q)}_b$ has closed range, then $\Pi^{(q)}=\Pi^{(q)}_{\leq\lambda}$ for some $\lambda>0$,
so we can deduce the asymptotic of the Szeg\H{o} kernel from Theorem \ref{t-gue140305_a}. 
We introduce now the following local and more flexible version of the closed range property.
\begin{defn}\label{d-gue140212}
Fix $q\in\set{0,1,2,\ldots,n-1}$. Let $Q:L^2_{(0,q)}(X)\To L^2_{(0,q)}(X)$
be a continuous operator. We say that $\Box^{(q)}_b$ has local $L^2$ closed range
on an open set $D\subset X$ with respect to $Q$ if for every $D'\Subset D$, 
there exist constants $C_{D'}>0$ and $p\in\mathbb N$, such that
\[\norm{Q(I-\Pi^{(q)})u}^2\leq C_{D'}\big(\,(\Box^{(q)}_b)^pu\,|\,u\big),\ \ \forall u\in\Omega^{0,q}_0(D').\]
\end{defn}
\noindent
When $D=X$, $Q$ is the identity map  and $p=2$, this property is just the $L^2$ closed range property for $\Box^{(q)}_b$. 
When $D=X$, $Q$ is the identity map, $p=1$ and $q=0$, this property is the $L^2$ closed range property for $\ddbar_b$.
\begin{thm}\label{t-gue140305VIa}
Let $X$ be a CR manifold of dimension $2n-1$, whose Levi form is non-degenerate of constant signature 
$(n_-,n_+)$ at
each point of an open set $D\Subset X$. Let $q\in\set{0,1,\ldots,n-1}$ and let
$Q\in L^0_{{\rm cl\,}}(X,T^{*0,q}X\boxtimes T^{*0,q}X)$
be a classical pseudodifferential operator on $X$ and let 
$Q^*\in L^0_{{\rm cl\,}}(X,T^{*0,q}X\boxtimes T^{*0,q}X)$ be
the $L^2$
adjoint of $Q$ with respect to $(\,\cdot\,|\,\cdot\,)$. Suppose that $\Box^{(q)}_b$ has local $L^2$
closed range on $D$ with respect
to $Q$ and $Q\Pi^{(q)}=\Pi^{(q)}Q$ on $L^2_{(0,q)}(X)$. Then, $Q^*\Pi^{(q)}Q$ is smoothing on $D$ if
$q\notin\set{n_-,n_+}$ and is a Fourier integral operator with complex phase if $q\in\set{n_-,n_+}$.
\end{thm}
This result will be proved in Section \ref{s-gue140212}, see Theorem \ref{t-gue140305VIab} 
for a detailed version of Theorem~\ref{t-gue140305VIa}.

For $Q\in L^m_{{\rm cl\,}}(X)$, 
let $\sigma_Q(x,\xi)\in\cC^\infty(T^*M)$ denote the principal symbol of $Q$. 
Let $\Sigma=\Sigma^{-}\cup\Sigma^{+}$ be the characteristic manifold of $\Box^{(q)}_b$ 
(see \eqref{e-gue140205I}). 
By using Theorem~\ref{t-gue140305VIa}, we establish the following local embedding theorem 
(see Section~\ref{s-gue140212} for a proof).

\begin{thm}\label{t-emb}
Let $X$ be a CR manifold of dimension $2n-1$, whose Levi form is positive at each point of an open set 
$D\Subset X$. Let $Q\in L^0_{{\rm cl\,}}(X)$ with $Q\Pi^{(0)}=\Pi^{(0)}Q$ and $\sigma_Q(x,\xi)\neq 0$ 
at each point of $\Sigma^{-}$. Suppose that $\Box^{(0)}_b$ 
has local $L^2$ closed range on $D$ with respect to $Q$. Then, for any point $x_0\in D$, 
there is an open neighborhood $\hat D\Subset D$ of $x_0$ such that $\hat D$ can be embedded into 
$\Complex^n$ by a global CR map.
\end{thm}
We notice that in Theorem~\ref{t-emb},
$\Box^{(0)}_b$ might not have $L^2$ closed range, however, with the help of the operator $Q$, 
we can still understand the Szeg\H{o} projection and produce many global CR functions.

We will apply Theorem~\ref{t-gue140305VIa} to establish Szeg\H{o} kernel asymptotic
expansions on compact CR manifolds with transversal CR $S^1$ actions under certain Levi curvature
assumptions.
\begin{thm}\label{t-bis}
Let $(X,T^{1,0}X)$ be a compact CR manifold of dimension $2n-1$, $n\geq2$, with a transversal CR $S^1$ action and let 
$T\in \cC^\infty(X,TX)$ be the real vector field induced by this $S^1$ action. For $m\in\mathbb Z$, let 
$\mathcal{B}^{0,q}_m(X)\subset L^2_{(0,q)}(X)$ be the completion of 
\[B^{0,q}_m(X):=\set{u\in\Omega^{0,q}(X);\, Tu=-\sqrt{-1}mu}\] and let 
$Q^{(q)}_{\leq0}:L^2_{(0,q)}(X)\To\oplus_{m\in\mathbb Z, m\leq0}\mathcal{B}^{0,q}_m(X)$ 
be the orthogonal projection. Assume that $Z(q)$ fails but $Z(q-1)$ and $Z(q+1)$ hold at every point of $X$. Then, $\Box^{(q)}_b$ has local $L^2$ closed range on $X$ with respect to $Q^{(q)}_{\leq0}$.
Suppose further that the Levi form is non-degenerate of constant signature $(n_-,n_+)$ on an open 
canonical coordinate patch $D\Subset X$. Then, $Q^{(q)}_{\leq0}\Pi^{(q)}Q^{(q)}_{\leq0}$ 
is smoothing on $D$ if $q\neq n_-$ and $Q^{(q)}_{\leq0}\Pi^{(q)}Q^{(q)}_{\leq0}$ is a Fourier integral 
operator with complex phase if $q=n_-$\,.
\end{thm}
This result will be proved in \S \ref{s-gue140211}, see Theorem \ref{t-gue140223I} for the details and see 
Definition~\ref{d-gue140211I} and Definition~\ref{d-gue140218} for the meanings of transversal CR $S^1$ 
action and condition $Z(q)$. 
As a consequence we obtain (cf.\ Theorem~\ref{t-gue140223III} and Corollary \ref{c-gue140306}):
\begin{thm}\label{c-gue140306a}
Let $(X,T^{1,0}X)$ be a compact CR manifold of dimension $2n-1$, $n\geq2$, with a transversal CR $S^1$ action. 
Assume $W(q)$ holds on $X$ for some $q\in\set{0,1,\ldots,n-1}$. Then $\Box^{(q)}_b$ has $L^2$ closed range. 
In particular, for any CR submanifold in $\Complex\mathbb P^N$ of the form \eqref{e-gue140303},
the associated Szeg\H{o} kernel $\Pi^{(q)}(x,y)$ admits a full asymptotic expansion.
\end{thm}
We notice that if the Levi form is non-degenerate of constant signature on $X$ then $W(q)$ holds on $X$ 
(see Definition~\ref{d-gue140223}). In particular, for a $3$-dimensional compact strictly pseudoconvex CR manifold, 
$W(0)$ holds on $X$. Hence, $\Box^{(0)}_b$ has $L^2$ closed range if $X$ admits a transversal CR $S^1$ action. 
From this, we deduce the following global embeddablity of
Lempert \cite[Theorem\,2.1]{Lem92}, cf.\ also \cite[Theorem\,A16]{Eps92} (see Section~\ref{s-gue140211}).
\begin{thm}\label{t-gue140306f}
Let $(X,T^{1,0}X)$ be a compact strictly pseudoconvex CR manifold of dimension three with a transversal
CR $S^1$ action. Then $X$ can be CR embedded into $\Complex^N$, for some $N\in\mathbb N$\,.
\end{thm}
Note that Baouendi-Rothschild-Treves \cite{BRT85} proved that the existence
of a local transverse CR action implies local embeddability. Let us point out that transversality
in Theorem \ref{t-gue140306f} cannot be dispensed with: the non-embeddable example of 
Grauert, Andreotti-Siu and Rossi \cite{Gra:94,AS:70,Ros:65} admits a nontransversal
circle action; see also the example of
of Barrett \cite{Ba88} which admits a transverse CR torus action, but no one-dimensional sub-action
exists which itself is transverse. The embeddable small deformations of of $S^1$ invariant strictly
pseudoconvex CR structures
on circle bundles over Riemann surfaces were described by Epstein \cite{Eps92}.

Theorem \ref{t-gue140305VIa} yields immediately the following.
\begin{thm}\label{t-gue140305VIc}
Suppose that $X$ is a CR manifold such that $\Box^{(0)}_b$ has closed range in $L^2$. 
Then the Szeg\H{o} projector $\Pi^{(0)}$ is a Fourier integral operator on the subset where the Levi form is positive definite.
\end{thm}
\begin{cor}\label{c-140621}
Let $X$ be a compact pseudoconvex CR manifold satisfying one of the following conditions:
\begin{itemize}
\item[(i)] $X=\partial M$, where $M$ is a relatively compact pseudoconvex domain
in a complex manifold, such that there exists a strictly psh function in a neighborhood of $X$.  
\item[(ii)] $X$ admits a CR embedding into some Euclidean space $\mathbb{C}^N$.
\end{itemize}
Then the Szeg\H{o} projector $\Pi^{(0)}$ is a Fourier integral operator on the subset where the Levi form is
positive definite.
\end{cor}
Indeed, it was shown that $\overline\partial_b$ has closed range in $L^2$ under condition (i) in \cite[p.\,543]{Koh86}
and under condition (ii)  in \cite{Bar12,Ni06}. For boundaries of pseudoconvex domains in $\mathbb{C}^n$
the closed range property was shown in \cite{BoSh:86,Koh86,Shaw85}. Note also that any three-dimensional 
pseudoconvex and of finite type CR manifold $X$ admits a CR embedding into some 
$\mathbb{C}^N$ if $\ddbar_b$ has $L^2$ closed range, cf.\ \cite{Ch89}. 

We can give a very concrete description of the Szeg\H{o} kernel in case (i) of Corollary \ref{c-140621}.
Let $M$ be a relatively compact domain with smooth boundary in a complex manifold $M'$ 
and $M=\{\rho<0\}$ where $\rho\in\cC^\infty(M')$ is a defining function of $M$. 
We assume that the Levi form $\mathcal{L}(\rho)$ is everywhere positive semi-definite on the complex 
tangent space to $X=\partial M$ and is positive definite of a subset $D\subset X$. Fix $D_0\Subset D$ and let $U$ be a small neighbourhood of $D_0$ in $M'$. As in \cite{BouSj76}, 
one can construct an almost-analytic extension 
$\varphi=\varphi(x,y):M'\times M'\to\mathbb{C}$ of $\rho$ with the following properties:
\begin{equation}\label{e:aae}
\begin{split}
&\text{$\varphi(x,x)=\frac{1}{\sqrt{-1}}\varrho(x)$ and $\partial_y\varphi$, $\overline\partial_x\varphi$
vanish to infinite order on the diagonal $x=y$.}\\
&\text{$\varphi(y,x)=\overline{\varphi(x,y)}$.}\\
&\text{${\rm Im\,}\varphi(x,y)\geq c\abs{x-y}^2$ on $U\times U$, where $c>0$ is a constant}.
\end{split}
\end{equation}
Then on $D_0$, the Szeg\H{o} kernel $\Pi^{(0)}(x,y)$ of $X$ has the form
\begin{equation}\label{e:szego}
\begin{split}
\Pi^{(0)}(x,y)&=\int^{\infty}_{0}\!\! e^{i\varphi(x, y)t}s(x, y, t)dt+R(x,y)\\
&=F(x,y)\big(\!-i\varphi(x, y)+0\big)^{-n}+G(x,y)\log\!\big(\!-i\varphi(x, y)+0\big)
\end{split}
\end{equation}
for some smooth functions $F$, $G$ and $R$. Here we denote by $\big(\!-i\varphi(x, y)+0\big)^{-n}$,
$\log\!\big(\!-i\varphi(x, y)+0\big)$ the distributions
limit of $\big(\!-i\varphi(x, y)+\varepsilon\big)^{-n}$ and $\log\big(\!-i\varphi(x, y)+\varepsilon\big)$
as $\varepsilon\to0+$.
See Section \ref {s-gue140224} for more details.

Our method can be extended to non-compact weakly pseudoconvex tube domains in $\mathbb{C}^{n}$ with basis a strictly
pseudoconvex domain in $\mathbb{C}^{n-1}$ cf.\ Theorems \ref{t-gue140306} and \ref{t-gue140306I}.

Let us finally mention that the analysis of the Szeg\H{o} kernels was also used to study embeddings given by CR sections of a 
positive CR bundle, introduced
in \cite{HM12a} (see also \cite{Hsiao12,Hsiao14}).

The Szeg\H{o} projector plays an important role in embedding problems
also through the framework of relative index for Szeg\H{o} projectors introduced by Epstein cf.\ \cite{Eps06}. 
One outcome of this analysis is the solution of the relative index conjecture \cite{Eps12}, which implies that the set of 
embeddable deformations of a strictly pseudoconvex CR structure on a compact three-dimensional
manifold is closed in the $\mathscr{C}^\infty$-topology.
  
The layout of this paper is as follows. In Section~\ref{s:prelim}, we collect some notations, 
definitions and statements we use throughout.

In Section~\ref{s-gue140205}, we review some results in~\cite{Hsiao08}
about the existence of a microlocal Hodge decomposition of the Kohn Laplacian on an open set
of a CR manifold where the Levi form is non-degenerate.

In Section~\ref{s-gue140205I},
we first study the microlocal behahaviour of the spectral function and by using the
microlocal Hodge decomposition of the Kohn Laplacian established in~\cite{Hsiao08},
we prove Theorem~\ref{t-gue140305_a}. Furthermore,  by using Theorem~\ref{t-gue140305_a}
and some standard technique in functional analysis, we prove Theorem~\ref{t-gue140211Im}.

Section~\ref{s-gue140212} is devoted to proving Theorem~\ref{t-gue140305VIa} and \ref{t-emb}.

In Section~\ref{s-gue140211}, we study CR manifolds with transversal CR $S^1$ actions.
We introduce the microlocal cut-off functions $Q^{(q)}_{\leq0}$ and $Q^{(q)}_{\geq0}$
and study the closed range property with respect to these operators. Finally we establish 
Theorems~\ref{t-bis},~\ref{c-gue140306a} and \ref{t-gue140306f}.

In Section~\ref{s-gue140224}, by using H\"ormander's $L^2$ estimates,
we establish the local $L^2$ closed range property for $\Box^{(0)}_{b}$ with respect to
 $Q^{(0)}$ for some weakly pseudocnovex tube domains in $\mathbb{C}^n$, hence establish the asymptotics of the Szeg\H{o} kernel
(see Theorem~\ref{t-gue140306} and Theorem~\ref{t-gue140306I}).

Finally, in Section~\ref{s-gue140215}, we prove the technical
Theorem~\ref{t-gue140215} by using semi-classical analysis and global theory of complex
Fourier integral operators of Melin-Sj\"ostrand \cite{MS74}.
Theorem~\ref{t-gue140215} will be used in the proof of Theorem~\ref{t-gue140305VIa}.

\section{Preliminaries}\label{s:prelim}
\subsection{Standard notations} \label{s-ssna}
We shall use the following notations: $\mathbb N=\set{1,2,\ldots}$, $\mathbb N_0=\mathbb N\cup\set{0}$, $\Real$ 
is the set of real numbers, $\ol\Real_+:=\set{x\in\Real;\, x\geq0}$.
For a multiindex $\alpha=(\alpha_1,\ldots,\alpha_n)\in\mathbb N_0^n$
we denote by $\abs{\alpha}=\alpha_1+\ldots+\alpha_n$ its norm and by $l(\alpha)=n$ its length.
For $m\in\mathbb N$, write $\alpha\in\set{1,\ldots,m}^n$ if $\alpha_j\in\set{1,\ldots,m}$, 
$j=1,\ldots,n$. $\alpha$ is strictly increasing if $\alpha_1<\alpha_2<\ldots<\alpha_n$. For $x=(x_1,\ldots,x_n)$ we write
\[
\begin{split}
&x^\alpha=x_1^{\alpha_1}\ldots x^{\alpha_n}_n,\\
& \pr_{x_j}=\frac{\pr}{\pr x_j}\,,\quad
\pr^\alpha_x=\pr^{\alpha_1}_{x_1}\ldots\pr^{\alpha_n}_{x_n}=\frac{\pr^{\abs{\alpha}}}{\pr x^\alpha}\,,\\
&D_{x_j}=\frac{1}{i}\pr_{x_j}\,,\quad D^\alpha_x=D^{\alpha_1}_{x_1}\ldots D^{\alpha_n}_{x_n}\,,
\quad D_x=\frac{1}{i}\pr_x\,.
\end{split}
\]
Let $z=(z_1,\ldots,z_n)$, $z_j=x_{2j-1}+ix_{2j}$, $j=1,\ldots,n$, be coordinates of $\Complex^n$.
We write
\[
\begin{split}
&z^\alpha=z_1^{\alpha_1}\ldots z^{\alpha_n}_n\,,\quad\ol z^\alpha=\ol z_1^{\alpha_1}\ldots\ol z^{\alpha_n}_n\,,\\
&\pr_{z_j}=\frac{\pr}{\pr z_j}=
\frac{1}{2}\Big(\frac{\pr}{\pr x_{2j-1}}-i\frac{\pr}{\pr x_{2j}}\Big)\,,\quad\pr_{\ol z_j}=
\frac{\pr}{\pr\ol z_j}=\frac{1}{2}\Big(\frac{\pr}{\pr x_{2j-1}}+i\frac{\pr}{\pr x_{2j}}\Big),\\
&\pr^\alpha_z=\pr^{\alpha_1}_{z_1}\ldots\pr^{\alpha_n}_{z_n}=\frac{\pr^{\abs{\alpha}}}{\pr z^\alpha}\,,\quad
\pr^\alpha_{\ol z}=\pr^{\alpha_1}_{\ol z_1}\ldots\pr^{\alpha_n}_{\ol z_n}=
\frac{\pr^{\abs{\alpha}}}{\pr\ol z^\alpha}\,.
\end{split}
\]
For $j, s\in\mathbb Z$, set $\delta_{j,s}=1$ if $j=s$, $\delta_{j,s}=0$ if $j\neq s$.

Let $M$ be a $\cC^\infty$ paracompact manifold.
We let $TM$ and $T^*M$ denote the tangent bundle of $M$
and the cotangent bundle of $M$ respectively.
The complexified tangent bundle of $M$ and the complexified cotangent bundle of $M$ are be denoted by $\Complex TM$
and $\Complex T^*M$, respectively. Write $\langle\,\cdot\,,\cdot\,\rangle$ to denote the pointwise
duality between $TM$ and $T^*M$.
We extend $\langle\,\cdot\,,\cdot\,\rangle$ bilinearly to $\Complex TM\times\Complex T^*M$.
Let $G$ be a $\cC^\infty$ vector bundle over $M$. The fiber of $G$ at $x\in M$ will be denoted by $G_x$.
Let $E$ be another vector bundle over $M$. We write
$G\boxtimes E$ to denote the vector bundle over $M\times M$ with fiber over $(x, y)\in M\times M$
consisting of the linear maps from $G_x$ to $E_y$.  Let $Y\subset M$ be an open set. 
From now on, the spaces of distribution sections of $G$ over $Y$ and
smooth sections of $G$ over $Y$ will be denoted by $\mathscr D'(Y, G)$ and $\cC^\infty(Y, G)$ respectively.
Let $\mathscr E'(Y, G)$ be the subspace of $\mathscr D'(Y, G)$ whose elements have compact support in $Y$.
For $m\in\Real$, let $H^m(Y, G)$ denote the Sobolev space
of order $m$ of sections of $G$ over $Y$. Put
\begin{gather*}
H^m_{\rm loc\,}(Y, G)=\big\{u\in\mathscr D'(Y, G);\, \varphi u\in H^m(Y, G),
      \, \forall\varphi\in \cC^\infty_0(Y)\big\}\,,\\
      H^m_{\rm comp\,}(Y, G)=H^m_{\rm loc}(Y, G)\cap\mathscr E'(Y, G)\,.
\end{gather*}
We recall the Schwartz kernel theorem \cite[Theorems\,5.2.1, 5.2.6]{Hor03}, \cite[Thorem\,B.2.7]{MM07},
\cite[p.\,296]{Tay1:96}.
Let $G$ and $E$ be $\cC^\infty$ vector
bundles over a paracompact orientable $\cC^\infty$ manifold $M$ equipped with a smooth density of integration. If
$A: \cC^\infty_0(M,G)\To \mathscr D'(M,E)$
is continuous, we write $K_A(x, y)$ or $A(x, y)$ to denote the distribution kernel of $A$.
The following two statements are equivalent
\begin{enumerate}
\item $A$ is continuous: $\mathscr E'(M,G)\To \cC^\infty(M,E)$,
\item $K_A\in \cC^\infty(M\times M,G_y\boxtimes E_x)$.
\end{enumerate}
If $A$ satisfies (a) or (b), we say that $A$ is smoothing. Let
$A,B: \cC^\infty_0(M,G)\to \mathscr D'(M,E)$ be continuous operators.
We write $A\equiv B$ (on $M$) if $A-B$ is a smoothing operator. 

We say that $A$ is properly supported if the restrictions of the two projections 
$(x,y)\mapsto x$, $(x,y)\mapsto y$ to ${\rm Supp\,}K_A$
are proper.

Let $H(x,y)\in\mathscr D'(M\times M,G_y\boxtimes E_x)$. We write $H$ to denote the unique 
continuous operator $\cC^\infty_0(M,G)\To\mathscr D'(M,E)$ with distribution kernel $H(x,y)$. 
In this work, we identify $H$ with $H(x,y)$.

\subsection{Set up and Terminology} \label{s-su}
Let $(X,T^{1,0}X)$ be an orientable not necessarily compact, paracompact CR manifold of dimension $2n-1$, $n\geqslant2$,
where $T^{1,0}X$ is a
CR structure of $X$. Recall that $T^{1,0}X$ is a complex $n-1$ dimensional subbundle of $\Complex TX$, 
satisfying $T^{1,0}X\cap T^{0,1}X=\set{0}$, where $T^{0,1}X=\ol{T^{1,0}X}$,
and $[\mathcal{V},\mathcal{V}]\subset\mathcal{V}$, where $\mathcal{V}=\cC^\infty(X,T^{1,0}X)$.

Fix a smooth Hermitian metric $\langle\,\cdot\,|\,\cdot\,\rangle$ on $\Complex TX$ 
so that $\langle\,u\,|\,v\,\rangle$ is real if $u$, $v$ are real tangent vectors and $T^{1,0}X$
is orthogonal to $T^{0,1}X:=\ol{T^{1,0}X}$.
Then locally there is a real vector field $T$ of length one which is pointwise orthogonal to
$T^{1, 0}X\oplus T^{0, 1}X$. $T$ is unique up to the choice of sign. For $v\in\Complex TX$, 
we write $\abs{v}^2:=\langle\,v\,|\,v\,\rangle$.
We denote by $T^{*1,0}X$ and $T^{*0,1}X$ the dual bundles of $T^{1,0}X$ and $T^{0,1}X$, respectively.
Define the vector bundle of $(0, q)$ forms by $T^{*0,q}X:=\Lambda^{q}T^{*0,1}X$.
The Hermitian metric $\langle\,\cdot\,|\,\cdot\,\rangle$ on $\Complex TX$ induces,
by duality, a Hermitian metric on $\Complex T^*X$ and also on the bundles of $(0,q)$ 
forms $T^{*0,q}X$, $q=0,1,\ldots,n-1$. We shall also denote all these induced metrics 
by $\langle\,\cdot\,|\,\cdot\,\rangle$. For $u\in T^{*0,q}X$, we write $\abs{u}^2:=\langle\,u\,|\,u\,\rangle$.  
Let $D\subset X$ be an open set. Let $\Omega^{0,q}(D)$ denote the space of smooth sections 
of $T^{*0,q}X$ over $D$ and let $\Omega^{0,q}_0(D)$ be the subspace of
$\Omega^{0,q}(D)$ whose elements have compact support in $D$.

Locally there exists an orthonormal frame $\omega_1,\ldots,\omega_{n-1}$
of the bundle $T^{*1,0}X$. The real $(2n-2)$ form
$\omega=i^{n-1}\omega_1\wedge\ol\omega_1\wedge\ldots\wedge\omega_{n-1}\wedge\ol\omega_{n-1}$
is independent of the choice of the orthonormal frame. Thus $\omega$ is globally
defined. Locally there exists a real $1$-form $\omega_0$ of length one which is orthogonal to
$T^{*1,0}X\oplus T^{*0,1}X$. The form $\omega_0$ is unique up to the choice of sign.
Since $X$ is orientable, there is a nowhere vanishing $(2n-1)$ form $\Theta$ on $X$.
Thus, $\omega_0$ can be specified uniquely by requiring that $\omega\wedge\omega_0=f\Theta$,
where $f$ is a positive function. Therefore $\omega_0$, so chosen, is globally defined.
We call $\omega_0$
the uniquely determined global real $1$-form.
We take a vector field $T$ so that
\begin{equation}\label{e-suIa}
\abs{T}=1\,,\quad \langle\,T\,,\,\omega_0\,\rangle=-1\,.
\end{equation}
Therefore $T$ is uniquely determined. We call $T$ the uniquely determined global real vector field. We have the
pointwise orthogonal decompositions:
\begin{equation} \label{e-suIIa}
\Complex T^*X=T^{*1,0}X\oplus T^{*0,1}X\oplus\set{\lambda\omega_0;\,
\lambda\in\Complex},  \:\:
\Complex TX=T^{1,0}X\oplus T^{0,1}X\oplus\set{\lambda T;\,\lambda\in\Complex}.
\end{equation}
\begin{defn} \label{d-suIa}
For $p\in X$, the Levi form $\mathcal{L}_p$ is the Hermitian quadratic form on $T^{1,0}_pX$ defined as follows. 
For any $Z,\ W\in T^{1,0}_pX$, pick $\mathcal{Z},\mathcal{W}\in
\cC^\infty(X,T^{1,0}X)$ such that
$\mathcal{Z}(p)=Z$, $\mathcal{W}(p)=W$. Set
\begin{equation} \label{e-suIII}
\mathcal{L}_p(Z,\ol W)=\frac{1}{2i}\big\langle\big[\mathcal{Z}\ ,\ol{\mathcal{W}}\,\big](p)\ ,\omega_0(p)\big\rangle\,,
\end{equation}
where $\big[\mathcal{Z}\ ,\ol{\mathcal{W}}\,\big]=\mathcal{Z}\ \ol{\mathcal{W}}-\ol{\mathcal{W}}\ \mathcal{Z}$ 
denotes the commutator of $\mathcal{Z}$ and $\ol{\mathcal{W}}$.
Note that $\mathcal{L}_p$ does not depend of the choices of $\mathcal{Z}$ and $\mathcal{W}$.
\end{defn}

Locally there exists an orthonormal basis $\{\mathcal{Z}_1,\ldots,\mathcal{Z}_{n-1}\}$
of $T^{1,0}X$ with respect to the Hermitian metric $\langle\,\cdot\,|\,\cdot\,\rangle$ such that $\mathcal{L}_p$ is 
diagonal in this basis, $\mathcal{L}_p(\mathcal{Z}_j,\ol{\mathcal{Z}}_l)=\delta_{j,l}\lambda_j(p)$.
The entries $\lambda_1(p)$, \ldots, $\lambda_{n-1}(p)$ are called the eigenvalues of the Levi form
at $p\in X$ with respect to $\langle\,\cdot\,|\,\cdot\,\rangle$.

\begin{defn} \label{d-suIIa}
Given $q\in\{0,\ldots,n-1\}$, the Levi form is said to satisfy condition $Y(q)$ at $p\in X$, if $\mathcal{L}_p$ 
has at least either $\min{(q+1,n-q)}$ pairs of eigenvalues 
with opposite signs or $\max{(q+1, n-q)}$ eigenvalues of the same sign. 
Notice that the sign of the eigenvalues does not depend on the choice of the metric 
$\langle\,\cdot\,|\,\cdot\,\rangle$.
\end{defn}

Let 
\begin{equation} \label{e-suIV}
\ddbar_b:\Omega^{0,q}(X)\To\Omega^{0,q+1}(X)
\end{equation}
be the tangential Cauchy-Riemann operator. 
We will work with two volume forms on $X$:
\begin{itemize}
\item A given smooth positive $(2n-1)$-form $m(x)$ on $X$.
\item The volume form $v(x)$ induced by the Hermitian metric $\langle\,\cdot\,|\,\cdot\,\rangle$.
\end{itemize}  
The natural global $L^2$ inner product $(\,\cdot\,|\,\cdot\,)$ on $\Omega^{0,q}_0(X)$ 
induced by $m(x)$ and $\langle\,\cdot\,|\,\cdot\,\rangle$ is given by
\begin{equation}\label{e:l2}
(u|v):=\int_X\langle u(x)|v(x)\rangle\, m(x)\,,\quad u,v\in\Omega^{0,q}_0(X)\,.
\end{equation}
We denote by $L^2_{(0,q)}(X)$ 
the completion of $\Omega^{0,q}_0(X)$ with respect to $(\,\cdot\,|\,\cdot\,)$. 
We write $L^2(X):=L^2_{(0,0)}(X)$. We extend $(\,\cdot\,|\,\cdot\,)$ to $L^2_{(0,q)}(X)$ 
in the standard way. For $f\in L^2_{(0,q)}(X)$, we denote $\norm{f}^2:=(\,f\,|\,f\,)$.
We extend
$\ddbar_{b}$ to $L^2_{(0,r)}(X)$, $r=0,1,\ldots,n-1$, by
\begin{equation}\label{e-suVII}
\ddbar_{b}:{\rm Dom\,}\ddbar_{b}\subset L^2_{(0,r)}(X)\To L^2_{(0,r+1)}(X)\,,
\end{equation}
where ${\rm Dom\,}\ddbar_{b}:=\{u\in L^2_{(0,r)}(X);\, \ddbar_{b}u\in L^2_{(0,r+1)}(X)\}$, 
where for any $u\in L^2_{(0,r)}(X)$, $\ddbar_{b} u$ is defined in the sense of distributions.
We also write
\begin{equation}\label{e-suVIII}
\ol{\pr}^{*}_{b}:{\rm Dom\,}\ol{\pr}^{*}_{b}\subset L^2_{(0,r+1)}(X)\To L^2_{(0,r)}(X)
\end{equation}
to denote the Hilbert space adjoint of $\ddbar_{b}$ in the $L^2$ space with respect to $(\,\cdot\,|\,\cdot\, )$.
Let $\Box^{(q)}_{b}$ denote the (Gaffney extension) of the Kohn Laplacian given by
\begin{equation}\label{e-suIX}
\begin{split}
{\rm Dom\,}\Box^{(q)}_{b}=\Big\{s\in L^2_{(0,q)}(X);&\, 
s\in{\rm Dom\,}\ddbar_{b}\cap{\rm Dom\,}\ol{\pr}^{*}_{b},\,
\ddbar_{b}s\in{\rm Dom\,}\ol{\pr}^{*}_{b},\ \ol{\pr}^{*}_{b}s\in{\rm Dom\,}\ddbar_{b}\Big\}\,,\\
\Box^{(q)}_{b}s&=\ddbar_{b}\ol{\pr}^{*}_{b}s+\ol{\pr}^{*}_{b}\ddbar_{b}s
\:\:\text{for $s\in {\rm Dom\,}\Box^{(q)}_{b}$}\,.
 \end{split}
\end{equation}
By a result of Gaffney, for every $q=0,1,\ldots,n-1$, $\Box^{(q)}_{b}$ is a positive self-adjoint operator 
(see \cite[Proposition\,3.1.2]{MM07}). That is, $\Box^{(q)}_{b}$ is self-adjoint and 
the spectrum of $\Box^{(q)}_{b}$ is contained in $\ol\Real_+$, $q=0,1,\ldots,n-1$. 
We shall write ${\rm Spec\,}\Box^{(q)}_b$ to denote the spectrum of $\Box^{(q)}_b$. 
For a Borel set $B\subset\Real$ we denote by $E(B)$ the spectral projection of $\Box^{(q)}_{b}$ 
corresponding to the set $B$, where $E$ is the spectral measure of $\Box^{(q)}_{b}$ 
(see Davies~\cite[\S\,2]{Dav95} for the precise meanings of spectral projection and spectral measure). 
For $\lambda_1>\lambda\geq0$, we set
\begin{equation} \label{e-suX}
\begin{split}
&H^q_{b,\leq\lambda}(X):={\rm Ran\,}E\bigr((-\infty,\lambda]\bigr)\subset L^2_{(0,q)}(X)\,,\\
&H^q_{b,>\lambda}(X):={\rm Ran\,}E\bigr((\lambda,\infty)\bigr)\subset L^2_{(0,q)}(X),\\
&H^q_{b,(\lambda,\lambda_1]}(X):={\rm Ran\,}E\bigr((\lambda,\lambda_1]\bigr)\subset L^2_{(0,q)}(X).
\end{split}
\end{equation}
For $\lambda=0$, we denote
\begin{equation} \label{e-suXI}
H_b^q(X):=H^q_{b,\leq0}(X)={\rm Ker\,}\Box^{(q)}_{b}.
\end{equation}
For $\lambda_1>\lambda\geq0$, let
\begin{equation}\label{e-suXI-I}
\begin{split}
&\Pi^{(q)}_{\leq\lambda}:L^2_{(0,q)}(X)\To H^q_{b,\leq\lambda}(X),\\
&\Pi^{(q)}_{>\lambda}:L^2_{(0,q)}(X)\To H^q_{b,>\lambda}(X),\\
&\Pi^{(q)}_{(\lambda,\lambda_1]}:L^2_{(0,q)}(X)\To H^q_{b,(\lambda,\lambda_1]}(X)
\end{split}
\end{equation}
be the orthogonal projections with respect to the product $(\,\cdot\,|\,\cdot\,)$ defined in \eqref{e:l2} and let
\begin{equation}\label{e-suXI-II}
\Pi^{(q)}_{\leq\lambda}(x,y),\:\Pi^{(q)}_{>\lambda}(x,y),\: \Pi^{(q)}_{(\lambda,\lambda_1]}(x,y)
\in\mathscr D'(X\times X,T^{*0,q}_yX\boxtimes T^{*0,q}_xX),
\end{equation}
%
denote the distribution kernels of $\Pi^{(q)}_{\leq\lambda}$, $\Pi^{(q)}_{>\lambda}$ and 
$\Pi^{(q)}_{(\lambda,\lambda_1]}$, respectively. For $\lambda=0$, we denote 
$\Pi^{(q)}:=\Pi^{(q)}_{\leq0}$, $\Pi^{(q)}(x,y):=\Pi^{(q)}_{\leq0}(x,y)$.

We recall now some notions of microlocal analysis.
The characteristic manifold of $\Box^{(q)}_b$ is given by $\Sigma=\Sigma^+\cup\Sigma^-$, where
\begin{equation}\label{e-gue140205I}
\Sigma^+=\set{(x, \lambda\omega_0(x))\in T^*X;\, \lambda>0},\:\:
\Sigma^-=\set{(x, \lambda\omega_0(x))\in T^*X;\, \lambda<0},
\end{equation}
where $\omega_0\in\cC^\infty(X,T^*X)$ is the uniquely determined global
$1$-form (see the discussion before \eqref{e-suIa}).

Let $\Gamma$ be a conic open set of $\Real^M$, $M\in\mathbb N$, and let $E$ be a smooth
vector bundle over $\Gamma$. Let $m\in\Real$, $0\leq\rho,\delta\leq1$. Let $S^m_{\rho,\delta}(\Gamma,E)$
denote the H\"{o}rmander symbol space on $\Gamma$ with values in $E$ of order $m$ type $(\rho,\delta)$
and let $S^m_{{\rm cl\,}}(\Gamma,E)$
denote the space of classical symbols on $\Gamma$ with values in $E$ of order $m$, 
see  Grigis-Sj\"{o}strand~\cite[Definition 1.1 and p.\,35]{GS94} and Definition~\ref{d-gue140221a} below.

Let $D$ be an open set of $X$. Let
$L^m_{\frac{1}{2},\frac{1}{2}}(D,T^{*0,q}X\boxtimes T^{*0,q}X)$ and
$L^m_{{\rm cl\,}}(D,T^{*0,q}D\boxtimes T^{*0,q}D)$
denote the space of pseudodifferential operators on $D$ of order $m$ type $(\frac{1}{2},\frac{1}{2})$
from sections of $T^{*0,q}X$ to sections of $T^{*0,q}X$ and the space of classical
pseudodifferential operators on $D$ of order $m$ from sections of
$T^{*0,q}X$ to sections of $T^{*0,q}X$ respectively. The classical result of
Calderon and Vaillancourt tells us that for any $A\in L^m_{\frac{1}{2},\frac{1}{2}}(D,T^{*0,q}X\boxtimes T^{*0,q}X)$,
\begin{equation}\label{e-gue140205}
\mbox{$A:H^s_{\rm comp}(D,T^{*0,q}X)\To H^{s-m}_{\rm loc}(D,T^{*0,q}X)$ is continuous, for every $s\in\Real$}.
\end{equation}
We refer to H\"{o}rmander~\cite{Hor85} for a proof.

\begin{defn}\label{d-gue140221a}
For $m\in\Real$, $S^m_{1,0}\big(D\times D\times\mathbb{R}_+,T^{*0,q}_yX\boxtimes T^{*0,q}_xX\big)$ 
is the space of all $a(x,y,t)\in\cC^\infty\big(D\times D\times\mathbb{R}_+,T^{*0,q}_yX\boxtimes T^{*0,q}_xX\big)$ 
such that for all compact $K\Subset D\times D$ and all $\alpha, \beta\in\mathbb N^{2n-1}_0$, $\gamma\in\mathbb N_0$, 
there is a constant $C_{\alpha,\beta,\gamma}>0$ such that 
\[\abs{\pr^\alpha_x\pr^\beta_y\pr^\gamma_t a(x,y,t)}\leq C_{\alpha,\beta,\gamma}(1+\abs{t})^{m-\abs{\gamma}},\ \ 
\forall (x,y,t)\in K\times\Real_+,\ \ t\geq1.\]
Put 
\[S^{-\infty}\big(D\times D\times\mathbb{R}_+,T^{*0,q}_yX\boxtimes T^{*0,q}_xX\big):=
\bigcap_{m\in\Real}S^m_{1,0}\big(D\times D\times\mathbb{R}_+,T^{*0,q}_yX\boxtimes T^{*0,q}_xX\big).\]
Let $a_j\in S^{m_j}_{1,0}\big(D\times D\times\mathbb{R}_+,T^{*0,q}_yX\boxtimes T^{*0,q}_xX\big)$, 
$j=0,1,2,\ldots$ with $m_j\To-\infty$, $j\To\infty$. 
Then there exists $a\in S^{m_0}_{1,0}\big(D\times D\times\mathbb{R}_+,T^{*0,q}_yX\boxtimes T^{*0,q}_xX\big)$ 
unique modulo $S^{-\infty}$, such that 
$a-\sum^{k-1}_{j=0}a_j\in S^{m_k}_{1,0}\big(D\times D\times\mathbb{R}_+,T^{*0,q}_yX\boxtimes T^{*0,q}_xX\big)$ 
for $k=0,1,2,\ldots$. 

If $a$ and $a_j$ have the properties above, we write $a\sim\sum^{\infty}_{j=0}a_j$ in 
$S^{m_0}_{1,0}\big(D\times D\times\mathbb{R}_+,T^{*0,q}_yX\boxtimes T^{*0,q}_xX\big)$. We write
\begin{equation}  \label{e-gue140205III}
s(x, y, t)\in S^{n-1}_{{\rm cl\,}}\big(D\times D\times\mathbb{R}_+,T^{*0,q}_yX\boxtimes T^{*0,q}_xX\big)
\end{equation}
if $s(x, y, t)\in S^{n-1}_{1,0}\big(D\times D\times\mathbb{R}_+,T^{*0,q}_yX\boxtimes T^{*0,q}_xX\big)$ and 
\begin{equation}\label{e-fal}\begin{split}
&s(x, y, t)\sim\sum^\infty_{j=0}s^j(x, y)t^{n-1-j}\quad\text{ in }S^{n-1}_{1, 0}
\big(D\times D\times\mathbb{R}_+\,,T^{*0,q}_yX\boxtimes T^{*0,q}_xX\big)\,,\\
&s^j(x, y)\in \cC^\infty\big(D\times D,T^{*0,q}_yX\boxtimes T^{*0,q}_xX\big),\ \ j\in\N_0.\end{split}\end{equation}
\end{defn}

\begin{defn}\label{d-gue140221}
Let $Q\in L^0_{{\rm cl\,}}(X,T^{*0,q}X\boxtimes T^{*0,q}X)$ be a classical pseudodifferential operator on $X$. 
Let $D\Subset X$ be an open local coordinate patch of $X$ with local coordinates 
$x=(x_1,\ldots,x_{2n-1})$ and let $\eta=(\eta_1,\ldots,\eta_{2n-1})$
be the dual variables of $x$.
We write
\[\mbox{$Q\equiv0$ at $\Sigma^-\cap T^*D$}\,,\]
if for every $D'\Subset D$, 
\[
Q(x,y)\equiv\int e^{i\langle x-y,\eta\rangle}q(x,\eta)d\eta\:\:\text{on $D'$},
\]
where $q(x,\eta)\in S^0_{{\rm cl\,}}(T^*D',T^{*0,q}X\boxtimes T^{*0,q}X)$
and there exist $M>0$ and a conic open neighbourhood $\Lambda_-$ of $\Sigma^-$
such that for every $(x,\eta)\in T^*D'\cap\Lambda_-$ with $\abs{\eta}\geq M$, we have $q(x,\eta)=0$.
We define similarly $\mbox{$Q\equiv0$ at $\Sigma^+\cap T^*D$}$\,.
\end{defn}

\section{Microlocal Hodge decomposition Theorems for $\Box^{(q)}_b$}\label{s-gue140205}

In this section we review some results in~\cite{Hsiao08} about the existence of a microlocal 
Hodge decomposition of the Kohn Laplacian on an open set of a CR manifold where the Levi form is non-degenerate.

Theorems \ref{t-gue140205}, Theorem~\ref{t-gue140205I} and Theorem~\ref{t-gue140205II} 
are proved in chapter 6, chapter 7 and chapter 8 of part I in \cite{Hsiao08}. In
\cite{Hsiao08} the existence of the microlocal Hodge decomposition is stated for compact CR manifolds, 
but the construction and arguments used are essentially local.

\begin{thm} \label{t-gue140205}
We assume that the Levi form is non-degenerate of constant signature $(n_-,n_+)$ at each point of an open set $D\Subset X$.
Let $q\neq n_-, n_+$. Then, there is a properly supported operator 
$A\in L^{-1}_{\frac{1}{2},\frac{1}{2}}(D,T^{*0,q}D\boxtimes T^{*0,q}D)$ such that $\Box^{(q)}_bA\equiv I$ on $D$.
\end{thm}

Let $p_0(x, \xi)\in \cC^\infty(T^*X)$ be the principal symbol of $\Box^{(q)}_b$. 
Note that $p_0(x,\xi)$ is a polynomial of degree $2$ in $\xi$. Recall that the characteristic manifold 
of $\Box^{(q)}_b$ is given by $\Sigma=\Sigma^+\cup\Sigma^-$, where $\Sigma^+$ and $\Sigma^-$ 
are given by \eqref{e-gue140205I}.

\begin{thm} \label{t-gue140205I}
We assume that the Levi form is non-degenerate of constant signature $(n_-,n_+)$ 
at each point of an open set $D\Subset X$.
Let $q=n_-$ or $n_+$. Then there exist properly supported continuous operators 
$A\in  L^{-1}_{\frac{1}{2},\frac{1}{2}}(D,T^{*0,q}D\boxtimes T^{*0,q}D)$ , 
$S_-, S_+\in L^{0}_{\frac{1}{2},\frac{1}{2}}(D,T^{*0,q}D\boxtimes T^{*0,q}D)$, such that
\begin{equation}\label{e-gue140205II}
\begin{split}
&\mbox{$\Box^{(q)}_bA+S_-+S_+=I$ on $D$},\\
&\Box^{(q)}_bS_-\equiv0\ \ \mbox{on $D$},\ \ \Box^{(q)}_bS_+\equiv0\ \ \mbox{on $D$},\\
&A\equiv A^*\ \ \mbox{on $D$},\ \ S_-A\equiv0\ \ \mbox{on $D$}, \ \ S_+A\equiv0\ \ \mbox{on $D$},\\
&S_-\equiv S_-^*\equiv S_-^2\ \ \mbox{on $D$},\\
&S_+\equiv S_+^*\equiv S_+^2\ \ \mbox{on $D$},\\
&S_-S_+\equiv S_+S_-\equiv0\ \ \mbox{on $D$},
\end{split}
\end{equation}
where $A^*$, $S_-^*$ and $S_+^*$ are the formal adjoints of $A$, $S_-$ and $S_+$
with respect to $(\,\cdot\,|\,\cdot\,)$ respectively and $S_-(x,y)$ satisfies
\[S_-(x, y)\equiv\int^{\infty}_{0}e^{i\varphi_-(x, y)t}s_-(x, y, t)dt\ \ \mbox{on $D$}\]
with a symbol $s_-(x, y, t)\in S^{n-1}_{{\rm cl\,}}\big(D\times D\times\mathbb{R}_+,T^{*0,q}_yX\boxtimes T^{*0,q}_xX\big)$
as in \eqref{e-gue140205III}, \eqref{e-fal}
and phase function $\varphi_-$ such that $\varphi=\varphi_-$ satisfies 
\begin{equation}\label{e-gue140205IV}
\begin{split}
&\varphi\in \cC^\infty(D\times D),\ \ {\rm Im\,}\varphi(x, y)\geq0,\\
&\varphi(x, x)=0,\ \ \varphi(x, y)\neq0\ \ \mbox{if}\ \ x\neq y,\\
&d_x\varphi(x, y)\big|_{x=y}=-\omega_0(x), \ \ d_y\varphi(x, y)\big|_{x=y}=\omega_0(x), \\
&\varphi(x, y)=-\ol\varphi(y, x).
\end{split}
\end{equation}
Moreover, there is a function $f\in \cC^\infty(D\times D)$ such that
\begin{equation} \label{e-gue140205V}
p_0(x, \varphi'_x(x,y))-f(x,y)\varphi(x,y)
\end{equation}
vanishes to infinite order at $x=y$.
Similarly,
\[S_+(x, y)\equiv\int^{\infty}_{0}\!\! e^{i\varphi_+(x, y)t}s_+(x, y, t)dt\ \ \mbox{on $D$}\]
with $s_+(x, y, t)\in S^{n-1}_{{\rm cl\,}}\big(D\times D\times\mathbb{R}_+,T^{*0,q}_yX\boxtimes T^{*0,q}_xX\big)$
as in \eqref{e-gue140205III}, \eqref{e-fal}
and $-\ol\varphi_+(x, y)$ satisfies \eqref{e-gue140205IV} and \eqref{e-gue140205V}. 
Moreover, if $q\neq n_+$, then $s_+(x,y,t)$ vanishes to infinite order at $x=y$. 
If $q\neq n_-$, then $s_-(x,y,t)$ vanishes to infinite order at $x=y$.
\end{thm}

The operators $S_{+}$, $S_{-}$ are called \emph{approximate Szeg\H{o} kernels}.

\begin{rem}\label{r-gue140211}
With the notations and assumptions used in Theorem~\ref{t-gue140205I}, assume that $q=n_-\neq n_+$. Since
$s_+(x,y,t)$ vanishes to infinite order at $x=y$, we have $S_+\equiv0$ on $D$. 
Similarly, if $q=n_+\neq n_-$. then $S_-\equiv0$ on $D$.
\end{rem}

The following result describes the phase function in local coordinates.
\begin{thm} \label{t-gue140205II}
We assume that the Levi form is non-degenerate of constant signature $(n_-,n_+)$
at each point of an open set $D\Subset X$.
For a given point $x_0\in D$, let $\{W_j\}_{j=1}^{n-1}$
be an orthonormal frame of $T^{1, 0}X$ in a neighbourhood of $x_0$
such that
the Levi form is diagonal at $x_0$, i.e.\ $\mathcal{L}_{x_{0}}(W_{j},\overline{W}_{s})=\delta_{j,s}\mu_{j}$, $j,s=1,\ldots,n-1$.
We take local coordinates
$x=(x_1,\ldots,x_{2n-1})$, $z_j=x_{2j-1}+ix_{2j}$, $j=1,\ldots,n-1$,
defined on some neighbourhood of $x_0$ such that $\omega_0(x_0)=dx_{2n-1}$, $x(x_0)=0$, 
and for some $c_j\in\Complex$, $j=1,\ldots,n-1$\,,
\[W_j=\frac{\pr}{\pr z_j}-i\mu_j\ol z_j\frac{\pr}{\pr x_{2n-1}}-
c_jx_{2n-1}\frac{\pr}{\pr x_{2n-1}}+O(\abs{x}^2),\ j=1,\ldots,n-1\,.\]
Set
$y=(y_1,\ldots,y_{2n-1})$, $w_j=y_{2j-1}+iy_{2j}$, $j=1,\ldots,n-1$.
Then, for $\varphi_-$ in Theorem~\ref{t-gue140205I}, we have
\begin{equation} \label{e-gue140205VI}
{\rm Im\,}\varphi_-(x,y)\geq c\sum^{2n-2}_{j=1}\abs{x_j-y_j}^2,\ \ c>0,
\end{equation}
in some neighbourhood of $(0,0)$ and
\begin{equation} \label{e-gue140205VII}
\begin{split}
&\varphi_-(x, y)=-x_{2n-1}+y_{2n-1}+i\sum^{n-1}_{j=1}\abs{\mu_j}\abs{z_j-w_j}^2 \\
&\quad+\sum^{n-1}_{j=1}\Bigr(i\mu_j(\ol z_jw_j-z_j\ol w_j)+c_j(-z_jx_{2n-1}+w_jy_{2n-1})\\
&\quad+\ol c_j(-\ol z_jx_{2n-1}+\ol w_jy_{2n-1})\Bigr)+(x_{2n-1}-y_{2n-1})f(x, y) +O(\abs{(x, y)}^3),
\end{split}
\end{equation}
where $f$ is smooth and satisfies $f(0,0)=0$, $f(x, y)=\ol f(y, x)$.
\end{thm}

The following formula for the leading term $s^0_-$ on the diagonal follows from \cite[\S 8]{Hsiao08}, 
its calculation being local in nature.
For a given point $x_0\in D$, let $\{W_j\}_{j=1}^{n-1}$ be an
orthonormal frame of $(T^{1,0}X,\langle\,\cdot\,|\,\cdot\,\rangle)$ near $x_0$, for which the Levi form
is diagonal at $x_0$. Put
\begin{equation}\label{levi140530}
\mathcal{L}_{x_0}(W_j,\ol W_\ell)=\mu_j(x_0)\delta_{j\ell}\,,\;\; j,\ell=1,\ldots,n-1\,.
\end{equation}
We will denote by
\begin{equation}\label{det140530}
\det\mathcal{L}_{x_0}=\prod_{j=1}^{n-1}\mu_j(x_0)\,.
\end{equation}
Let $\{T_j\}_{j=1}^{n-1}$ denote the basis of $T^{*0,1}X$, dual to $\{\ol W_j\}^{n-1}_{j=1}$. We assume that
$\mu_j(x_0)<0$ if\, $1\leq j\leq n_-$ and $\mu_j(x_0)>0$ if\, $n_-+1\leq j\leq n-1$. Put
\begin{equation}\label{n140530}
\begin{split}
&\mathcal{N}(x_0,n_-):=\set{cT_1(x_0)\wedge\ldots\wedge T_{n_-}(x_0);\, c\in\Complex},\\
&\mathcal{N}(x_0,n_+):=\set{cT_{n_-+1}(x_0)\wedge\ldots\wedge T_{n-1}(x_0);\, c\in\Complex}\end{split}
\end{equation}
and let
\begin{equation}\label{tau140530}
\begin{split}
\tau_{x_0,n_-}:T^{*0,q}_{x_0}X\To\mathcal{N}(x_0,n_-)\,,\quad
\tau_{x_0,n_+}:T^{*0,q}_{x_0}X\To\mathcal{N}(x_0,n_+)\,,\end{split}
\end{equation}
be the orthogonal projections onto $\mathcal{N}(x_0,n_-)$ and $\mathcal{N}(x_0,n_+)$
with respect to $\langle\,\cdot\,|\,\cdot\,\rangle$ respectively.
We recall that $m(x)$ is the given smooth $2n-1$ form on $X$ and $v(x)$ is the volume form induced by $\langle\,\cdot\,|\,\cdot\,\rangle$, see the discussion after~\eqref{e-suIV}.
\begin{thm} \label{t-gue140205III}
We assume that the Levi form is non-degenerate of constant signature
$(n_-,n_+)$ at each point of an open set $D\Subset X$.
If $q=n_-$, then for the leading term $s^0_-(x,y)$ of the expansion \eqref{e-fal} of $s_-(x,y,t)$, we have
\begin{equation}\label{e-gue140205VIII}
s^0_-(x_0, x_0)=\frac{1}{2}\pi^{-n}\abs{\det\mathcal{L}_{x_0}}\frac{v(x_0)}{m(x_0)}\tau_{x_0,n_-}\,,\:\:x_0\in D.
\end{equation}
Similarly, if $q=n_+$, then for the leading term $s^0_+(x,y)$ of the expansion \eqref{e-fal} of $s_+(x,y,t)$, we have
\begin{equation}\label{e-gue140205VIIIb}
s^0_+(x_0, x_0)=\frac{1}{2}\pi^{-n}\abs{\det\mathcal{L}_{x_0}}\frac{v(x_0)}{m(x_0)}\tau_{x_0,n_+}\,,\:\:x_0\in D.
\end{equation}
\end{thm}

\section{Microlocal spectral theory for $\Box^{(q)}_b$} \label{s-gue140205I}

In this section, we will apply the microlocal Hodge decomposition theorems for $\Box^{(q)}_b$ from
Section~\ref{s-gue140205} in order to study the singularities for the kernel $\Pi^{(q)}_{\leq\lambda}(x,y)$
on the non-degenerate part of the Levi form. The section ends with the proof of Theorem \ref{t-gue140211Im}.

For any $\lambda>0$, it is clearly that there is a continuous operator
\[N^{(q)}_\lambda:L^2_{(0,q)}(X)\To{\rm Dom\,}\Box^{(q)}_b\]
such that
\begin{equation}\label{e-gue140206m}
\begin{split}
&\mbox{$\Box^{(q)}_bN^{(q)}_\lambda+\Pi^{(q)}_{\leq\lambda}=I$ on $L^2_{(0,q)}(X)$},\\
&\mbox{$N^{(q)}_\lambda\Box^{(q)}_b+\Pi^{(q)}_{\leq\lambda}=I$ on ${\rm Dom\,}\Box^{(q)}_b$}.
\end{split}
\end{equation}
Let us formulate a detailed version of of Theorem \ref{t-gue140305_a}.
\begin{thm}\label{t-gue140305_b} 
With the notations and assumptions used above, assume that the Levi form is non-degenerate
of constant signature $(n_-,n_+)$ at each point of an open set $D\Subset X$.
If $q\notin\set{n_-,n_+}$, then there is a $A\in L^{-1}_{\frac{1}{2},\frac{1}{2}}(D,T^{*0,q}X\boxtimes T^{*0,q}X)$,
such that for any $\lambda>0$, we have
\[
\Pi^{(q)}_{\leq\lambda}\equiv 0\,\quad\text{and}\quad N^{(q)}_\lambda\equiv A\,\quad \text{on $D$}.
\]
%
If $q\in\set{n_-,n_+}$, then
for any $\lambda>0$, we have
\[
\Pi^{(q)}_{\leq\lambda}\equiv S_-+S_+\,\quad\text{and}\quad
N^{(q)}_\lambda\equiv G\quad \text{on $D$},
\]
where $G\in L^{-1}_{\frac{1}{2},\frac{1}{2}}(D,T^{*0,q}X\boxtimes T^{*0,q}X)$,
$S_-, S_+\in L^{0}_{\frac{1}{2},\frac{1}{2}}(D,T^{*0,q}X\boxtimes T^{*0,q}X)$
are independent of $\lambda$ and the kernels of $S_-$ and $S_+$ satisfy
\[
S_\pm(x, y)\equiv\int^{\infty}_{0}e^{i\varphi_\pm(x, y)t}s_\pm(x, y, t)dt\ \ \mbox{on $D$}
\]
with symbols 
$s_\pm (x, y, t)\in S^{n-1}_{{\rm cl\,}}\big(D\times D\times\mathbb{R}_+\,,T^{*0,q}_yX\boxtimes T^{*0,q}_xX\big)$
as in \eqref{e-gue140205III}, \eqref{e-fal}, $s_-=0$ if $q\neq n_-$, $s_+=0$ if $q\neq n_+$, where 
$s^0_-(x,x)$ and $s^0_+(x,x)$ are given by \eqref{e-gue140205VIIIm}, 
and phase functions $\varphi_\pm$ such that $\varphi=\varphi_-$ and $\varphi=-\ol\varphi_+$ satisfy 
\eqref{e-gue140205IV}, \eqref{e-gue140205V}
{\rm(}see Theorem~\ref{t-gue140305I} and Theorem~\ref{t-gue140305II}, for more properties of 
the phases $\varphi_\pm${\rm)}.
\end{thm}
Since $s_-(x,y,t)=0$ if $q\neq n_-$, $S_-\equiv0$ on $D$ if $q\neq n_-$. Similarly, $S_+\equiv0$ on $D$ if $q\neq n_+$.
The following result describes the phase function in local coordinates.

\begin{thm} \label{t-gue140305I}
The function $\varphi_-$ from Theorem \ref{t-gue140305_b} fulfills the estimates
\eqref{e-gue140205VI} and \eqref{e-gue140205VII} in local coordinates near a point of $D$, 
chosen as in Theorem \ref{t-gue140205II}.
\end{thm}
\begin{defn}\label{d-gue140305}
With the assumptions and notations used in Theorem~\ref{t-gue140305_b},
let $\varphi_1, \varphi_2\in \cC^\infty(D\times D)$. We assume that $\varphi_1$ and $\varphi_2$ satisfy
\eqref{e-gue140205IV} and \eqref{e-gue140205VI}.
We say that $\varphi_1$ and $\varphi_2$ are equivalent on $D$ if for any
$b_1(x,y,t)\in  S^{n-1}_{{\rm cl\,}}\big(D\times D\times\mathbb{R}_+,T^{*0,q}_yX\boxtimes T^{*0,q}_xX\big)$
we can find
$b_2(x,y,t)\in  S^{n-1}_{{\rm cl\,}}\big(D\times D\times\mathbb{R}_+,T^{*0,q}_yX\boxtimes T^{*0,q}_xX\big)$
such that
\[\int^\infty_0e^{i\varphi_1(x,y)t}b_1(x,y,t)dt\equiv e^{i\varphi_2(x,y)t}b_2(x,y,t)dt\ \ \mbox{on $D$}\]
and vise versa.
\end{defn}
We characterize now the phase $\varphi_-$ (see Section~\ref{s-gue140215}).
\begin{thm} \label{t-gue140305II}
With the assumptions and notations used in Theorem~\ref{t-gue140305_b}, let $\varphi_1\in \cC^\infty(D\times D)$.
We assume that $\varphi_1$ satisfies \eqref{e-gue140205IV} 
and \eqref{e-gue140205VI}. The functions
$\varphi_1$ and $\varphi_-$ are equivalent on $D$ in the sense of Definition~\ref{d-gue140305}
if and only if there is a function $h\in \cC^\infty(D\times D)$ such that $\varphi_1(x,y)-h(x,y)\varphi_-(x,y)$
vanishes to infinite order at $x=y$.
\end{thm}
The proof of Theorem~\ref{t-gue140305II} is essentially the same as the proof
of Theorem~\ref{t-gue140215} and therefore will be omitted.

We give the formulas of the leading terms of the asymptotic expansions of the symbols
$s_\pm(x,y)$ from Theorem~\ref{t-gue140305_b}.
\begin{thm} \label{t-gue140305III}
With the assumptions and notations used in Theorem~\ref{t-gue140305_b}, and the notations
\eqref{levi140530}, \eqref{n140530}, \eqref{tau140530}, we have
for a given point $x_0\in D$,
\begin{equation}\label{e-gue140205VIIIm}
\begin{split}
s^0_-(x_0, x_0)=\frac{1}{2}\pi^{-n}\abs{\det\mathcal{L}_{x_0}}\frac{v(x_0)}{m(x_0)}\tau_{x_0,n_-}
\,,\:\:\text{for $q=n_-$}\,,\\
s^0_+(x_0, x_0)=\frac{1}{2}\pi^{-n}\abs{\det\mathcal{L}_{x_0}}\frac{v(x_0)}{m(x_0)}\tau_{x_0,n_+}
\,,\:\:\text{for $q=n_+$}\,.
\end{split}
\end{equation}

\end{thm}
Recall that $\Pi^{(q)}_{\leq\lambda}$ is given by \eqref{e-suXI-I}.
Let $\lambda\geq0$. From the spectral theory for self-adjoint operators (see Davies~\cite{Dav95}), it is well-known that
\[\Pi^{(q)}_{\leq\lambda}:L^2_{(0,q)}(X)\To{\rm Dom\,}\Box^{(q)}_b\,,\quad
\Pi^{(q)}_{\leq\lambda}\Box^{(q)}_b=\Box^{(q)}_b\Pi^{(q)}_{\leq\lambda}\:\: 
\text{on ${\rm Dom\,}\Box^{(q)}_b$}\,\]
and $\Pi^{(q)}_{\leq\lambda}\Box^{(q)}_b:{\rm Dom\,}\Box^{(q)}_b\To L^2_{(0,q)}(X)$ is continuous.
Since ${\rm Dom\,}\Box^{(q)}_b$ is dense in $L^2_{(0,q)}(X)$, we can extend $\Pi^{(q)}_{\leq\lambda}\Box^{(q)}_b$
continuously to $L^2_{(0,q)}(X)$ in the standard way.
Similarly, for every $m\in\mathbb N$, we can extend $\Pi^{(q)}_{\leq\lambda}(\Box^{(q)}_b)^m$
continuously to $L^2_{(0,q)}(X)$
and we have
\begin{equation}\label{e-gue140205Ia}
\begin{split}
&\mbox{$(\Box^{(q)}_b)^m\Pi^{(q)}_{\leq\lambda}=\Pi^{(q)}_{\leq\lambda}(\Box^{(q)}_b)^m$ on $L^2_{(0,q)}(X)$},\\
&\mbox{$(\Box^{(q)}_b)^m\Pi^{(q)}_{\leq\lambda}=\Pi^{(q)}_{\leq\lambda}(\Box^{(q)}_b)^m:L^2_{(0,q)}(X)\To{\rm Dom\,}
\Box^{(q)}_b$ is continuous}.
\end{split}
\end{equation}
Now, we fix $\lambda>0$. It is clearly that there is a continuous operator
\[N^{(q)}_\lambda:L^2_{(0,q)}(X)\To{\rm Dom\,}\Box^{(q)}_b\]
such that
\begin{equation}\label{e-gue140206}
\begin{split}
&\mbox{$\Box^{(q)}_bN^{(q)}_\lambda+\Pi^{(q)}_{\leq\lambda}=I$ on $L^2_{(0,q)}(X)$},\\
&\mbox{$N^{(q)}_\lambda\Box^{(q)}_b+\Pi^{(q)}_{\leq\lambda}=I$ on ${\rm Dom\,}\Box^{(q)}_b$}.
\end{split}
\end{equation}
Until further notice, we assume that the Levi form is non-degenerate of constant signature $(n_-,n_+)$ 
at each point of an open set $D\Subset X$ and we work on $D$. We need

\begin{thm}\label{t-gue140206}
With the assumptions and notations used above, let $q=n_-$ or $n_+$. We have
\[\mbox{$\Box^{(q)}_b\Pi^{(q)}_{\leq\lambda}\equiv0$ on $D$}.\]
\end{thm}

\begin{proof}
In view of \eqref{e-gue140205II}, we see that
\begin{equation}\label{e-gue140206I}
\mbox{$A^*\Box^{(q)}_b+S^*_-+S^*_+=I$ on $D$}.
\end{equation}
Note that $A^*$, $S^*_-$, $S^*_+$, $A$, $S_-$ and $S_+$ are properly supported. We recall that
\begin{equation}\label{e-gue140206II}
\begin{split}
&A^*, A:H^s_{{\rm comp\,}}(D,T^{*0,q}X)\To H^{s+1}_{{\rm comp\,}}(D,T^{*0,q}X),\ \ \forall s\in\mathbb Z,\\
&A^*, A:H^s_{{\rm loc\,}}(D,T^{*0,q}X)\To H^{s+1}_{{\rm loc\,}}(D,T^{*0,q}X),\ \ \forall s\in\mathbb Z,\\
&S_-^*, S_-, S_+^*,S_+:H^s_{{\rm comp\,}}(D,T^{*0,q}X)\To H^{s}_{{\rm comp\,}}(D,T^{*0,q}X),\ \ \forall s\in\mathbb Z,\\
&S_-^*, S_-, S_+^*,S_+:H^s_{{\rm loc\,}}(D,T^{*0,q}X)\To H^{s}_{{\rm loc\,}}(D,T^{*0,q}X),\ \ \forall s\in\mathbb Z.
\end{split}
\end{equation}
From \eqref{e-gue140206I}, we have
\begin{equation}\label{e-gue140206II-I}
A^*(\Box^{(q)}_b)^2\Pi^{(q)}_{\leq\lambda}+(S^*_-+S^*_+)\Box^{(q)}_b
\Pi^{(q)}_{\leq\lambda}=\Box^{(q)}_b\Pi^{(q)}_{\leq\lambda}.
\end{equation}
Since $(S^*_-+S^*_+)\Box^{(q)}_b\equiv0$ on $D$, we have
\begin{equation}\label{e-gue140206III}
(S^*_-+S^*_+)\Box^{(q)}_b\Pi^{(q)}_{\leq\lambda}:
H^0_{{\rm comp\,}}(D,T^{*0,q}X)\To H^{s}_{{\rm loc\,}}(D,T^{*0,q}X),\ \ \forall s\in\mathbb N_0.
\end{equation}
From \eqref{e-gue140205Ia} and \eqref{e-gue140206II}, we see that
\begin{equation}\label{e-gue140206IV}
A^*(\Box^{(q)}_b)^2\Pi^{(q)}_{\leq\lambda}:H^0_{{\rm comp\,}}(D,T^{*0,q}X)\To H^{1}_{{\rm loc\,}}(D,T^{*0,q}X).
\end{equation}
From \eqref{e-gue140206IV}, \eqref{e-gue140206III} and \eqref{e-gue140206II-I}, we conclude that
\begin{equation}\label{e-gue140206V}
\Box^{(q)}_b\Pi^{(q)}_{\leq\lambda}:H^0_{{\rm comp\,}}(D,T^{*0,q}X)\To H^{1}_{{\rm loc\,}}(D,T^{*0,q}X).
\end{equation}
Similarly, we can repeat the procedure above and deduce that
\begin{equation}\label{e-gue140206VI}
(\Box^{(q)}_b)^2\Pi^{(q)}_{\leq\lambda}:H^0_{{\rm comp\,}}(D,T^{*0,q}X)\To H^{1}_{{\rm loc\,}}(D,T^{*0,q}X).
\end{equation}
From \eqref{e-gue140206VI} and \eqref{e-gue140206II}, we get
\begin{equation}\label{e-gue140206VII}
A^*(\Box^{(q)}_b)^2\Pi^{(q)}_{\leq\lambda}:H^0_{{\rm comp\,}}(D,T^{*0,q}X)\To H^{2}_{{\rm loc\,}}(D,T^{*0,q}X).
\end{equation}
Combining \eqref{e-gue140206VII}, \eqref{e-gue140206III} with \eqref{e-gue140206II-I}, we obtain
\begin{equation}\label{e-gue140206VIII}
\Box^{(q)}_b\Pi^{(q)}_{\leq\lambda}:H^0_{{\rm comp\,}}(D,T^{*0,q}X)\To H^{2}_{{\rm loc\,}}(D,T^{*0,q}X).
\end{equation}
Continuing in this way, we deduce that
\begin{equation}\label{e-gue140206a}
\Box^{(q)}_b\Pi^{(q)}_{\leq\lambda}:H^0_{{\rm comp\,}}(D,T^{*0,q}X)\To 
H^{s}_{{\rm loc\,}}(D,T^{*0,q}X),\ \ \forall s\in\mathbb N_0.
\end{equation}
Since $\Box^{(q)}_b\Pi^{(q)}_{\leq\lambda}=\Pi^{(q)}_{\leq\lambda}\Box^{(q)}_b$,
\begin{equation}\label{e-gue140206aI}
\Pi^{(q)}_{\leq\lambda}\Box^{(q)}_{b}:H^0_{{\rm comp\,}}(D,T^{*0,q}X)\To H^{s}_{{\rm loc\,}}(D,T^{*0,q}X),\ \ 
\forall s\in\mathbb N_0.
\end{equation}
By taking adjoint in \eqref{e-gue140206aI}, we conclude that
\begin{equation}\label{e-gue140206aII}
\Box^{(q)}_b\Pi^{(q)}_{\leq\lambda}:H^{-s}_{{\rm comp\,}}(D,T^{*0,q}X)
\To H^{0}_{{\rm loc\,}}(D,T^{*0,q}X),\ \ \forall s\in\mathbb N_0.
\end{equation}
Similarly, we can repeat the procedure above and deduce that for every $m\in\mathbb N$,
\begin{equation}\label{e-gue140206aIII}
\begin{split}
&(\Box^{(q)}_b)^m\Pi^{(q)}_{\leq\lambda}:H^{-s}_{{\rm comp\,}}(D,T^{*0,q}X)\To 
H^{0}_{{\rm loc\,}}(D,T^{*0,q}X),\ \ \forall s\in\mathbb N_0,\\
&(\Box^{(q)}_b)^m\Pi^{(q)}_{\leq\lambda}:
H^{0}_{{\rm comp\,}}(D,T^{*0,q}X)\To H^{s}_{{\rm loc\,}}(D,T^{*0,q}X),\ \ \forall s\in\mathbb N_0.
\end{split}
\end{equation}
Now, from \eqref{e-gue140206}, we have
\begin{equation}\label{e-gue140206aIV}
(S^*_-+S^*_+)\Box^{(q)}_bN^{(q)}_\lambda+(S^*_-+S^*_+)\Pi^{(q)}_{\leq\lambda}=S^*_-+S^*_+.
\end{equation}
Since $(S^*_-+S^*_+)\Box^{(q)}_b\equiv0$ on $D$, from \eqref{e-gue140206aIV}, it is easy to see that
\begin{equation}\label{e-gue140206aV}
(S^*_-+S^*_+)-(S^*_-+S^*_+)\Pi^{(q)}_{\leq\lambda}:
H^0_{{\rm comp\,}}(D,T^{*0,q}X)\To H^s_{{\rm loc\,}}(D,T^{*0,q}X),\ \ \forall s\in\mathbb N_0.
\end{equation}
From \eqref{e-gue140206I}, we have
\begin{equation}\label{e-gue140206aVII}
A^*\Box^{(q)}_b\Pi^{(q)}_{\leq\lambda}+(S^*_-+S^*_+)\Pi^{(q)}_{\leq\lambda}=\Pi^{(q)}_{\leq\lambda}.
\end{equation}
From \eqref{e-gue140206II}, \eqref{e-gue140206aIII}, \eqref{e-gue140206aVII} and 
\eqref{e-gue140206aV}, it is not difficult to see that
\begin{equation}\label{e-gue140206aVI}
(S^*_-+S^*_+)-\Pi^{(q)}_{\leq\lambda}:H^{0}_{{\rm comp\,}}(D,T^{*0,q}X)\To 
H^s_{{\rm loc\,}}(D,T^{*0,q}X),\ \ \forall s\in\mathbb N_0
\end{equation}
and hence
\begin{equation}\label{e-gue140206aVIII}
(S_-+S_+)-\Pi^{(q)}_{\leq\lambda}:H^{-s}_{{\rm comp\,}}(D,T^{*0,q}X)\To 
H^0_{{\rm loc\,}}(D,T^{*0,q}X),\ \ \forall s\in\mathbb N_0.
\end{equation}
Combining \eqref{e-gue140206aVIII} with \eqref{e-gue140206II}, we deduce that for any $s\in\mathbb N_0$
we can extend 
$\Pi^{(q)}_{\leq\lambda}$ to the space 
$H^{-s}_{{\rm comp\,}}(D,T^{*0,q}X)$, and we have
\begin{equation}\label{e-gue140206b}
\Pi^{(q)}_{\leq\lambda}:H^{-s}_{{\rm comp\,}}(D,T^{*0,q}X)\To 
H^{-s}_{{\rm loc\,}}(D,T^{*0,q}X),\ \ \forall s\in\mathbb N_0.
\end{equation}
From \eqref{e-gue140206b} and note that $(S^*_-+S^*_+)\Box^{(q)}_b\equiv0$ on $D$, we have
\begin{equation}\label{e-gue140206bI}
(S^*_-+S^*_+)\Box^{(q)}_b\Pi^{(q)}_{\leq\lambda}:H^{-s}_{{\rm comp\,}}(D,T^{*0,q}X)
\To H^{s}_{{\rm loc\,}}(D,T^{*0,q}X),\ \ \forall s\in\mathbb N_0.
\end{equation}
From \eqref{e-gue140206bI}, \eqref{e-gue140206aIII}, \eqref{e-gue140206II-I} and \eqref{e-gue140206II}, we obtain
\begin{equation}\label{e-gue14020bII}
\Box^{(q)}_b\Pi^{(q)}_{\leq\lambda}:H^{-s}_{{\rm comp\,}}(D,T^{*0,q}X)
\To H^{1}_{{\rm loc\,}}(D,T^{*0,q}X),\ \ \forall s\in\mathbb N_0.
\end{equation}
Similarly, we can repeat the procedure above and deduce that
\begin{equation}\label{e-gue140206bIII}
(\Box^{(q)}_b)^2\Pi^{(q)}_{\leq\lambda}:
H^{-s}_{{\rm comp\,}}(D,T^{*0,q}X)\To H^{1}_{{\rm loc\,}}(D,T^{*0,q}X),\ \ \forall s\in\mathbb N_0.
\end{equation}
From \eqref{e-gue140206bIII} and \eqref{e-gue140206II}, we get
\begin{equation}\label{e-gue140206bIV}
A^*(\Box^{(q)}_b)^2\Pi^{(q)}_{\leq\lambda}:H^{-s}_{{\rm comp\,}}(D,T^{*0,q}X)
\To H^{2}_{{\rm loc\,}}(D,T^{*0,q}X),\ \ \forall s\in\mathbb N_0.
\end{equation}
Combining \eqref{e-gue140206bIV}, \eqref{e-gue140206bI} with \eqref{e-gue140206II-I}, we obtain
\[\Box^{(q)}_b\Pi^{(q)}_{\leq\lambda}:H^{-s}_{{\rm comp\,}}(D,T^{*0,q}X)
\To H^{2}_{{\rm loc\,}}(D,T^{*0,q}X),\ \ \forall s\in\mathbb N_0.\]
Continuing in this way, we deduce that
\[\Box^{(q)}_b\Pi^{(q)}_{\leq\lambda}:H^{-s}_{{\rm comp\,}}(D,T^{*0,q}X)
\To H^{\ell}_{{\rm loc\,}}(D,T^{*0,q}X),\ \ \forall s, \ell\in\mathbb N_0. \]
Hence, $\Box^{(q)}_b\Pi^{(q)}_{\leq\lambda}\equiv0$ on $D$. The theorem follows.
\end{proof}
Now, we can prove one of the main results of this work.
\begin{thm}\label{t-gue140207}
We assume that the Levi form is non-degenerate of constant signature $(n_-,n_+)$
at each point of an open set $D\Subset X$. Let $q=n_-$ or $n_+$. Then, for any $\lambda>0$, we have
\begin{equation}\label{e-gue140207}
\Pi^{(q)}_{\leq\lambda}\equiv S_-+S_+\quad\text{and}\quad
N^{(q)}_{\lambda}\equiv A\quad\text{on $D$},
\end{equation}
where $N^{(q)}_\lambda$ is given by \eqref{e-gue140206}, $S_-$, $S_+$ and $A$ are as in Theorem~\ref{t-gue140205I}.
\end{thm}
\begin{proof}
Fix $\lambda>0$. From \eqref{e-gue140206I}, we have
\[\mbox{$A^*\Box^{(q)}_b\Pi^{(q)}_{\leq\lambda}+(S^*_-+S^*_+)\Pi^{(q)}_{\leq\lambda}=
\Pi^{(q)}_{\leq\lambda}$ on $D$}.\]
In view of Theorem~\ref{t-gue140206}, we see that
\begin{equation}\label{e-gue140210}
\mbox{$(S^*_-+S^*_+)\Pi^{(q)}_{\leq\lambda}=\Pi^{(q)}_{\leq\lambda}-F_1$ on $D$,}
\end{equation}
where
\begin{equation}\label{e-gue140211}
\begin{split}
&F_1=A^*\Box^{(q)}_b\Pi^{(q)}_{\leq\lambda},\\
&\mbox{$F_1\equiv0$ on $D$}.
\end{split}
\end{equation}
On the other hand, from \eqref{e-gue140206}, we have
\[N^{(q)}_\lambda\Box^{(q)}_b(S_-+S_+)+\Pi^{(q)}_{\leq\lambda}(S_-+S_+)=S_-+S_+.\]
Since $\Box^{(q)}_b(S_-+S_+)\equiv0$ on $D$, we conclude that
\begin{equation}\label{e-gue140210I}
\begin{split}
&S_-+S_+=\Pi^{(q)}_{\leq\lambda}(S_-+S_+)+N^{(q)}_\lambda F,\\
&S^*_-+S^*_+=(S^*_-+S^*_+)\Pi^{(q)}_{\leq\lambda}+F^*N^{(q)}_\lambda,
\end{split}
\end{equation}
where $F\equiv0$ on $D$ and $F^*$ is the adjoint of $F$. Note that
$F$ and $F^*$ are properly supported on $D$. From \eqref{e-gue140210} and \eqref{e-gue140210I}, we deduce that
\begin{equation}\label{e-gue140210II}
\begin{split}
&S_-+S_++F^*_1=\Pi^{(q)}_{\leq\lambda}+N^{(q)}_\lambda F,\\
&S^*_-+S^*_++F_1=\Pi^{(q)}_{\leq\lambda}+F^*N^{(q)}_\lambda,
\end{split}
\end{equation}
where $F^*_1$ is the adjoint of $F_1$.
From \eqref{e-gue140210II}, we have
\begin{equation}\label{e-gue140210III}
\mbox{$\Bigr(S^*_-+S^*_++F_1-\Pi^{(q)}_{\leq\lambda}\Bigr)
\Bigr(S_-+S_++F^*_1-\Pi^{(q)}_{\leq\lambda}\Bigr)=F^*(N^{(q)}_\lambda)^2F$ on $H^0_{{\rm comp\,}}(D,T^{*0,q}X)$}.
\end{equation}
Since
\[F^*(N^{(q)}_\lambda)^2F:
\mathscr E'(D,T^{*0,q}X)\To\Omega^{0,q}_0(X)\subset L^2_{(0,q)}(X)\To\Omega^{0,q}(X),\]
we have $F^*(N^{(q)}_\lambda)^2F\equiv0$ on $D$. From this observation and \eqref{e-gue140210III}, we obtain
\begin{equation}\label{e-gue140211I}
\mbox{$\Bigr(S^*_-+S^*_++F_1-\Pi^{(q)}_{\leq\lambda}\Bigr)
\Bigr(S_-+S_++F^*_1-\Pi^{(q)}_{\leq\lambda}\Bigr)\equiv0$ on $D$}.
\end{equation}
Now,
\begin{equation}\label{e-gue140211II}
\begin{split}
&\Bigr(S^*_-+S^*_++F_1-\Pi^{(q)}_{\leq\lambda}\Bigr)\Bigr(S_-+S_++F^*_1-\Pi^{(q)}_{\leq\lambda}\Bigr)\\
&=(S^*_-+S^*_+)(S_-+S_+)+(S^*_-+S^*_+)F^*_1-(S^*_-+S^*_+)\Pi^{(q)}_{\leq\lambda}+F_1(S_-+S_+)\\
&\quad+F_1F^*_1-F_1\Pi^{(q)}_{\leq\lambda}-\Pi^{(q)}_{\leq\lambda}(S_-+S_+)-
\Pi^{(q)}_{\leq\lambda}F^*_1+\Pi^{(q)}_{\leq\lambda}.
\end{split}
\end{equation}
Since $F_1\equiv0$ on $D$ and $S_-$, $S_+$ are properly supported on $D$, it is clearly that
$F_1(S_-+S_+)$ and $(S^*_-+S^*_+)F^*_1$ are well-defined and
\begin{equation}\label{e-gue140211III}
F_1(S_-+S_+)\equiv0\,,\quad
(S^*_-+S^*_+)F^*_1\equiv0\quad\text{on $D$}.
\end{equation}
From \eqref{e-gue140211} and Theorem~\ref{t-gue140206}, we see that
\begin{equation}\label{e-gue140211IV}
\mbox{$F_1\Pi^{(q)}_{\leq\lambda}=A^*\Box^{(q)}_b(\Pi^{(q)}_{\leq\lambda})^2=
A^*\Box^{(q)}_b\Pi^{(q)}_{\leq\lambda}\equiv0$ on $D$}
\end{equation}
and hence
\begin{equation}\label{e-gue140211V}
\mbox{$\Pi^{(q)}_{\leq\lambda}F^*_1\equiv0$ on $D$}.
\end{equation}
From \eqref{e-gue140211}, we see that $F_1F^*_1=A^*(\Box^{(q)}_b)^2\Pi^{(q)}_{\leq\lambda}A$. 
Since $A$ is properly supported, $F_1F^*_1$ is well-defined as a continuous operator
\[F_1F^*_1:\Omega^{0,q}_0(D)\To\mathscr D'(D,T^{*0,q}X).\]
Moreover, from the proof of Theorem~\ref{t-gue140206}, we see that 
$(\Box^{(q)}_b)^2\Pi^{(q)}_{\leq\lambda}\equiv0$ on $D$. Thus,
\begin{equation}\label{e-gue140211VI}
\mbox{$F_1F^*_1=A^*(\Box^{(q)}_b)^2\Pi^{(q)}_{\leq\lambda}A\equiv0$ on $D$}.
\end{equation}
From \eqref{e-gue140210}, \eqref{e-gue140211II}, \eqref{e-gue140211III}, 
\eqref{e-gue140211IV}, \eqref{e-gue140211V} and \eqref{e-gue140211VI}, it is straightforward to see that
\begin{equation}\label{e-gue140211VII}
\begin{split}
\Bigr(&S^*_-+S^*_++F_1-\Pi^{(q)}_{\leq\lambda}\Bigr)\Bigr(S_-+S_++F^*_1-
\Pi^{(q)}_{\leq\lambda}\Bigr)\\&\equiv(S^*_-+S^*_+)(S_-+S_+)-\Pi^{(q)}_{\leq\lambda}\:\:\text{on $D$}.
\end{split}
\end{equation}
From \eqref{e-gue140211VII} and \eqref{e-gue140211I}, we conclude that
\begin{equation}\label{e-gue140211VIII}
\mbox{$(S^*_-+S^*_+)(S_-+S_+)\equiv\Pi^{(q)}_{\leq\lambda}$ on $D$}.
\end{equation}
From \eqref{e-gue140205II}, it is not difficult to see that $(S^*_-+S^*_+)(S_-+S_+)\equiv S_-+S_+$ on $D$. 
Combining this observation with \eqref{e-gue140211VIII}, we get
\begin{equation}\label{e-gue140211aI}
\mbox{$S_-+S_+\equiv\Pi^{(q)}_{\leq\lambda}$ on $D$}.
\end{equation}
The first formula in \eqref{e-gue140207} follows.
We now prove the second formula in \eqref{e-gue140207}. We first claim that
\begin{equation}\label{e-gue140211aII}
\mbox{$(S^*_-+S^*_+)N^{(q)}_\lambda(S_-+S_+)\equiv0$ on $D$}.
\end{equation}
From \eqref{e-gue140210II}, \eqref{e-gue140211} and notice that 
$N^{(q)}_\lambda\Pi^{(q)}_{\leq\lambda}=\Pi^{(q)}_{\leq\lambda}N^{(q)}_\lambda=0$, we have
\[N^{(q)}_\lambda(S_-+S_+)=(N^{(q)}_\lambda)^2F\]
and hence
\begin{equation}\label{e-gue140211aIII}
(N^{(q)}_\lambda)^2(S_-+S_+)=(N^{(q)}_\lambda)^3F.
\end{equation}
From \eqref{e-gue140210II}, \eqref{e-gue140211} and \eqref{e-gue140211aIII}, we have
\begin{equation}\label{e-gue140211aIV}
\begin{split}
(S^*_-+S^*_+)N^{(q)}_\lambda(S_-+S_+)&=(\Pi^{(q)}_{\leq\lambda}+F^*N^{(q)}_\lambda-F_1)N^{(q)}_\lambda(S_-+S_+)\\
&=F^*(N^{(q)}_\lambda)^2(S_-+S_+)\\
&\mbox{$=F^*(N^{(q)}_\lambda)^3F\equiv0$ on $D$}.
\end{split}
\end{equation}
The claim \eqref{e-gue140211aII} follows.
On $D$, we have
\begin{equation}\label{e-gue140211aV}
N^{(q)}_\lambda=(A^*\Box^{(q)}_b+S^*_-+S^*_+)N^{(q)}_\lambda=
A^*(I-\Pi^{(q)}_{\leq\lambda})+(S^*_-+S^*_+)N^{(q)}_\lambda
\end{equation}
and
\begin{equation}\label{e-gue140211aVI}
N^{(q)}_\lambda=N^{(q)}_\lambda(\Box^{(q)}_bA+S_-+S_+)=(I-\Pi^{(q)}_{\leq\lambda})A+N^{(q)}_\lambda(S_-+S_+).
\end{equation}
From \eqref{e-gue140211aV} and \eqref{e-gue140211aVI}, we have
\begin{equation}\label{e-gue140211aVII}
\begin{split}
N^{(q)}_\lambda&=(I-\Pi^{(q)}_{\leq\lambda})A+N^{(q)}_\lambda(S_-+S_+)\\
&=A-\Pi^{(q)}_{\leq\lambda}A+\Bigr(A^*(I-\Pi^{(q)}_{\leq\lambda})+
(S^*_-+S^*_+)N^{(q)}_\lambda\Bigr)\Bigr(S_-+S_+\Bigr)\\
&=A-\Pi^{(q)}_{\leq\lambda}A+A^*(I-\Pi^{(q)}_{\leq\lambda})(S_-+S_+)+(S^*_-+S^*_+)N^{(q)}_\lambda(S_-+S_+).
\end{split}
\end{equation}
From \eqref{e-gue140211aII}, \eqref{e-gue140211aI} and noting that 
$(S_-+S_+)A\equiv0$ on $D$, $A^*(S_-+S_+)\equiv0$ on $D$, we conclude that
\[\Pi^{(q)}_{\leq\lambda}A+A^*(I-\Pi^{(q)}_{\leq\lambda})(S_-+S_+)+(S^*_-+S^*_+)N^{(q)}_\lambda(S_-+S_+)\equiv0
\:\:\text{on $D$}.\]
From this and \eqref{e-gue140211aVII}, we get the second formula in \eqref{e-gue140207}. The theorem follows.
\end{proof}
By using Theorem~\ref{t-gue140205}, we can repeat the proof of Theorem~\ref{t-gue140207} and conclude the following.
\begin{thm}\label{t-gue140211}
We assume that the Levi form is non-degenerate of constant signature $(n_-,n_+)$
at each point of an open set $D\Subset X$. Assume that $q\notin\set{n_-,n_+}$. Then, for any $\lambda>0$, we have
\begin{equation}\label{e-gue140211aVIII}
\Pi^{(q)}_{\leq\lambda}\equiv 0\quad\text{and}\quad
N^{(q)}_{\lambda}\equiv A\quad \text{on $D$},
\end{equation}
where $N^{(q)}_\lambda$ is given by \eqref{e-gue140206}, and $A$ is as in Theorem~\ref{t-gue140205}.
\end{thm}
From Theorem~\ref{t-gue140207} and Theorem~\ref{t-gue140211}, Theorem~\ref{t-gue140305_b} follows.
\begin{defn} \label{d-gue140211}
Let $H$ be a Hilbert space and $Q$ be a closed densely defined operator
$Q:{\rm Dom\,}Q\subset H\To{\rm Ran\,}Q\subset H$,
with closed range. By the partial inverse of $Q$, we
mean the bounded operator $M: H\To H$ such that
$Q\circ M=\pi_2$, $M\circ Q=\pi_1$ on ${\rm Dom\,}Q$,
where $\pi_1$, $\pi_2$ are the orthogonal projections in $H$ such that
${\rm Ran\,}\pi_1=({\rm Ker\,}Q)^\bot$, ${\rm Ran\,}\pi_2={\rm Ran\,}Q$.
In other words, for $u\in H$, let
$\pi_2u=Qv$, $v\in({\rm Ker\,}Q)^\bot\cap{\rm Dom\,}Q$.
Then, $Mu=v$.
\end{defn}
From Theorem~\ref{t-gue140207} and Theorem~\ref{t-gue140211}, we deduce:
\begin{cor}\label{c-gue140211}
Let $q\in\set{0,1,\ldots,n-1}$. Assume that $\Box^{(q)}_b$ has $L^2$ closed range and let
$N^{(q)}:L^2_{(0,q)}(X)\To{\rm Dom\,}\Box^{(q)}_b$
be the partial inverse of $\Box^{(q)}_b$.
We assume that the Levi form is non-degenerate of constant signature $(n_-,n_+)$
at each point of an open set $D\Subset X$. If $q\notin\set{n_-,n_+}$, then
\[
\Pi^{(q)}\equiv 0\quad\text{and}\quad
N^{(q)}\equiv A\quad\text{on $D$},
\]
where $A$ is as in Theorem~\ref{t-gue140205}. If $q\in\set{n_-,n_+}$, then
\[
\Pi^{(q)}\equiv S_-+S_+\quad\text{and}\quad
N^{(q)}\equiv A\quad\text{on $D$},
\]
where $S_-$, $S_+$ and $A$ are as in Theorem~\ref{t-gue140205I}.
\end{cor}

By using Corollary~\ref{c-gue140211I} and some standard argument in functional analysis, we can establish the spectral theory of $\Box^{(q)}_b$ when the Levi form is non-degenerate of constant signature on $X$.
%
%
\begin{proof}[Proof of Theorem \ref{t-gue140211Im}]
Let $0<\mu<\mu_1<\infty$. We claim that ${\rm Spec\,}\Box^{(q)}_b\cap[\mu,\mu_1]$ 
is a discrete subset of $\Real$. We assume that ${\rm Spec\,}\Box^{(q)}_b\cap[\mu,\mu_1]$ 
is not a discrete subset of $\Real$. Then, we can find $f_j\in E([\mu,\mu_1])$, $j=1,2,\ldots$, 
with $(\,f_j\,|\,f_\ell\,)=\delta_{j,\ell}$, for all 
$j, \ell=1,2,\ldots$\,. Take $0<\lambda_1<\mu<\mu_1<\lambda_2<\infty$. Then, we have
\begin{equation}\label{e-gue140211bIII}
f_j=\Pi^{(q)}_{(\lambda_1,\lambda_2]}f_j,\ \ j=1,2,\ldots\,.
\end{equation}
In view of Corollary~\ref{c-gue140211I}, we know that $\Pi^{(q)}_{(\lambda_1,\lambda_2]}$ 
is a smoothing operator on $X$ and hence $\Pi^{(q)}_{(\lambda_1,\lambda_2]}$ is a compact operator 
on $L^2_{(0,q)}(X)$.
By Rellich's theorem, we can find a subsequence $\set{f_{s_k}}^{\infty}_{k=1}$ of $\set{f_j}^{\infty}_{j=1}$, 
where $1<s_1<s_2<\ldots$, such that
$f_{s_k}\To f$ in $L^2_{(0,q)}(X)$ as $k\To\infty$, for some $f\in L^2_{(0,q)}(X)$. 
But $(\,f_{s_k}\,|\,f_{s_\ell}\,)=0$ if $k\neq\ell$, we get a contradiction. 
We conclude that ${\rm Spec\,}\Box^{(q)}_b\cap[\mu,\mu_1]$ 
is a discrete subset of $\Real$, for any $0<\mu<\mu_1<\infty$. 
Thus, for any $\mu>0$, ${\rm Spec\,}\Box^{(q)}_b\cap[\mu,\infty)$ is a discrete subset of $\Real$.

Let $\nu\in{\rm Spec\,}\Box^{(q)}_b$ with $\nu>0$. Since ${\rm Spec\,}\Box^{(q)}_b\cap[\mu,\mu_1]$ 
is discrete, where $0<\mu<\nu<\mu_1$, we have
\[\Box^{(q)}_b-\nu:{\rm Dom\,}\Box^{(q)}_b\subset L^2_{(0,q)}(X)\To L^2_{(0,q)}(X)\]
has $L^2$ closed range. Hence, if $\Box^{(q)}_b-\nu$ is injective, 
then ${\rm Range\,}(\Box^{(q)}_b-\nu)=L^2_{(0,q)}(X)$ and $\Box^{(q)}_b-\nu$ has a bounded inverse
$(\Box^{(q)}_b-\nu)^{-1}:L^2_{(0,q)}(X)\To L^2_{(0,q)}(X)$.
Thus, $\nu$ is a resolvent if $\Box^{(q)}_b-\nu$ is injective. We conclude that $\Box^{(q)}_b-\nu$ 
is not injective, that is, $\nu$ is an eigenvalue of $\Box^{(q)}_b$.
Take $0<\lambda_1<\nu<\lambda_2<\infty$. We have
\[
H^{(q)}_{b,\nu}(X)=\Pi^{(q)}_{(\lambda_1,\lambda_2]}H^{(q)}_{b,\nu}(X)=
\set{\Pi^{(q)}_{(\lambda_1,\lambda_2]}f;\, f\in H^q_{b,\nu}(X)}.
\]
Since $\Pi^{(q)}_{(\lambda_1,\lambda_2]}$ is a smoothing operator on $X$, we conclude that
$H^{(q)}_{b,\nu}(X)\subset\Omega^{0,q}(X)$.
Moreover, from Rellich's theorem, we see that ${\rm dim\,}H^{(q)}_{b,\nu}(X)<\infty$. The theorem follows.
\end{proof}

\section{Szeg\H{o} kernel asymptotic expansions}\label{s-gue140212}

In this section, we will apply Theorem~\ref{t-gue140205} and Theorem~\ref{t-gue140205I} 
to establish Szeg\H{o} kernel asymptotic expansions on the non-degenerate part of the Levi form under certain local conditions.

In view of Theorem~\ref{t-gue140305_a}, we see that if $\Box^{(q)}_{b}$ has $L^2$ closed range, then $\Pi^{(q)}$ admits a full
asymptotic expansion on the non-degenerate part of the Levi form. But in general, it is difficult to see that if $\Box^{(q)}_{b}$ has
$L^2$ closed range. We then impose the condition of local $L^2$ closed range, cf.\ Definition \ref{d-gue140212}.
It is clearly that if $\Box^{(q)}_b$ has $L^2$ closed range then $\Box^{(q)}_b$ has local $L^2$ closed range on every open set $D$
with respect to the identity map $I$.

We now prove the following precise version of Theorem \ref{t-gue140305VIa}.
\begin{thm}\label{t-gue140305VIab}
Let $X$ be a CR manifold of dimension $2n-1$, whose Levi form is non-degenerate of constant signature 
$(n_-,n_+)$ at
each point of an open set $D\Subset X$. Let $q\in\set{0,1,\ldots,n-1}$ and let
$Q\in L^0_{{\rm cl\,}}(X,T^{*0,q}X\boxtimes T^{*0,q}X)$
be a classical pseudodifferential operator on $X$ and let
$Q^*\in L^0_{{\rm cl\,}}(X,T^{*0,q}X\boxtimes T^{*0,q}X)$ be the $L^2$ adjoint of $Q$
with respect to $(\,\cdot\,|\,\cdot\,)$. Suppose that $\Box^{(q)}_b$ has local $L^2$ closed range
on $D$ with respect to $Q$ and $Q\Pi^{(q)}=\Pi^{(q)}Q$ on $L^2_{(0,q)}(X)$. Then,
\begin{equation}\label{e-gue140212m}
\mbox{$Q^*\Pi^{(q)}Q\equiv0$ on $D$ if $q\notin\set{n_-,n_+}$}
\end{equation}
and if $q\in\set{n_-,n_+}$, then
\begin{equation}\label{e-gue140213VIIIm}
\begin{split}
(Q^*\Pi^{(q)}Q)(x,y)
\equiv\int^\infty_0e^{i\varphi_-(x,y)t}a_-(x,y,t)dt+\int^\infty_0e^{i\varphi_+(x,y)t}a_+(x,y,t)dt\:\:\text{on $D$},
\end{split}
\end{equation}
where $\varphi_\pm(x,y)\in \cC^\infty(D\times D)$ are 
as in Theorem~\ref{t-gue140305_b} and 
$a_-(x, y, t),a_+(x, y, t)\in S^{n-1}_{{\rm cl\,}}\big(D\times D\times\mathbb{R}_+,T^{*0,q}_yX\boxtimes T^{*0,q}_xX\big)$
satisfy
\begin{equation}  \label{e-gue140213aImA}\begin{split}
&\mbox{$a_-(x,y,t)=0$ if $q\neq n_-$ or $Q\equiv0$ at $\Sigma^-\cap T^*D$},\\
&\mbox{ $a_+(x,y,t)=0$ if $q\neq n_+$ or $Q\equiv0$ at $\Sigma^+\cap T^*D$}.
\end{split}\end{equation}
(See Definition~\ref{d-gue140221} for the meaning of $Q\equiv0$ at $\Sigma^-\bigcap T^*D$.)
Moreover, assume that $q=n_-$, then the leading term $a^0_-(x,y)$ of the expansion \eqref{e-fal} 
of $a_-(x,y,t)$ satisfies
\begin{equation}  \label{e-gue140213aIIm}
\begin{split}
a^0_-(x, x)
=\frac{1}{2}\pi^{-n}\frac{v(x)}{m(x)}
\abs{{\rm det\,}\mathcal{L}_x}\tau_{x,n_-}q^*(x,-\omega_0(x))q(x,-\omega_0(x))\tau_{x,n_-},\ \ \forall x\in D,
\end{split}
\end{equation}
where $\det\mathcal{L}_x$ is the determinant of the Levi form defined in \eqref{det140530}, 
$v(x)$ is the volume form on $X$ induced by $\langle\,\cdot\,|\,\cdot\,\rangle$, 
$q(x,\eta)\in \cC^\infty(T^*D)$ is the principal symbol of $Q$, $q^*(x,\eta)$ is the adjoint of 
$q(x,\eta):T^{*0,q}_xX\To T^{*0,q}_xX$ with respect to $\langle\,\cdot\,|\,\cdot\,\rangle$ 
and $\tau_{x,n_-}$ is as in \eqref{tau140530}.
\end{thm}
To prove Theorem \ref{t-gue140305VIab} we need a series of results, starting with the following.
\begin{thm}\label{t-gue140212}
In the conditions of Theorem \ref{t-gue140305VIab} we have \eqref{e-gue140212m} and
\begin{equation}\label{e-gue140212I}
\mbox{$Q^*\Pi^{(q)}Q\equiv(S^*_-+S^*_+)Q^*Q(S_-+S_+)$ on $D$ if $q\in\set{n_-,n_+}$},
\end{equation}
where $S_-$ and $S_+$ are as in Theorem~\ref{t-gue140205I}.
\end{thm}

\begin{proof}
We first assume that $q\in\set{n_-,n_+}$. Put $S=S_-+S_+$, where $S_-$ and $S_+$ are as in Theorem~\ref{t-gue140205I} and let $S^*$ be the adjoint of $S$.
From \eqref{e-gue140205II}, we have
\begin{equation}\label{e-gue140213}
\mbox{$\Pi^{(q)}=(A^*\Box^{(q)}_b+S^*)\Pi^{(q)}=S^*\Pi^{(q)}$}
\end{equation}
and hence
\begin{equation}\label{e-gue140213I}
\mbox{$\Pi^{(q)}=\Pi^{(q)}S$ on $D$}.
\end{equation}
Fix $D'\Subset D$. Let $u\in\Omega^{0,q}_0(D')$. Since $\Box^{(q)}_b$ has local $L^2$ closed range on $D$ with respect to $Q$, we have for every $s\in\mathbb Z$,
\begin{equation}\label{e-gue140213II}
\norm{Q(I-\Pi^{(q)})Su}\leq C_{D',s}\sqrt{\norm{(\Box^{(q)}_b)^pSu}_s\norm{u}_{-s}},\ \ \forall u\in\Omega^{0,q}_0(D'),
\end{equation}
where $C_{D',s}>0$, $p\in\mathbb N$ are constants independent of $u$ and $\norm{\cdot}_s$ 
denotes the usual Sobolev norm of order $s$ on $D'$. Since $(\Box^{(q)}_b)^pS\equiv0$ on $D$, 
for every $s\in\mathbb N_0$, there is a constant $C_s>0$ such that
\begin{equation}\label{e-gue140213III}
\norm{(\Box^{(q)}_b)^pSu}_s\leq C_s\norm{u}_{-s},\ \ \forall u\in\Omega^{0,q}_0(D').
\end{equation}
From \eqref{e-gue140213III} and \eqref{e-gue140213II}, 
we can extend $Q(I-\Pi^{(q)})S=QS-Q\Pi^{(q)}$(here we used \eqref{e-gue140213I}) to 
$H^{-s}_{{\rm comp,}}(D,T^{*0,q}X)$, $\forall s\in\mathbb N_0$ and we have
\begin{equation}\label{e-gue140213IV}
\mbox{$QS-Q\Pi^{(q)}:H^{-s}_{{\rm comp\,}}(D,T^{*0,q}X)\To L^2_{(0,q)}(X)$ is continuous, $\forall s\in\mathbb N_0$}.
\end{equation}
By taking adjoint in \eqref{e-gue140213IV}, we get
\begin{equation}\label{e-gue140213V}
\mbox{$S^*Q^*-\Pi^{(q)}Q^*: L^2_{(0,q)}(X)\To H^{s}_{{\rm loc\,}}(D,T^{*0,q}X)$ is continuous, $\forall s\in\mathbb N_0$}.
\end{equation}
From \eqref{e-gue140213IV} and \eqref{e-gue140213V}, we conclude that  for any $s\in\mathbb N_0$ the map
\[
(S^*Q^*-\Pi^{(q)}Q^*)(QS-Q\Pi^{(q)}): H^{-s}_{{\rm comp\,}}(D,T^{*0,q}X)\To
H^{s}_{{\rm loc\,}}(D,T^{*0,q}X).
\]
is continuous.
Hence,
\begin{equation}\label{e-gue140213VI}
\mbox{$(S^*Q^*-\Pi^{(q)}Q^*)(QS-Q\Pi^{(q)})\equiv0$ on $D$}.
\end{equation}
Now,
\begin{equation}\label{e-gue140213VII}
\begin{split}
&(S^*Q^*-\Pi^{(q)}Q^*)(QS-Q\Pi^{(q)})\\
&=S^*Q^*QS-S^*Q^*Q\Pi^{(q)}-\Pi^{(q)}Q^*QS+\Pi^{(q)}Q^*Q\Pi^{(q)}\\
&=S^*Q^*QS-S^*\Pi^{(q)}Q^*Q-Q^*Q\Pi^{(q)}S+Q^*Q\Pi^{(q)}\\
&=S^*Q^*QS-\Pi^{(q)}Q^*Q\Pi^{(q)}.
\end{split}
\end{equation}
Here we used $Q\Pi^{(q)}=\Pi^{(q)}Q$, $Q^*\Pi^{(q)}=\Pi^{(q)}Q^*$, \eqref{e-gue140213}
and \eqref{e-gue140213I}. From \eqref{e-gue140213VII} and \eqref{e-gue140213VI}, \eqref{e-gue140212I} follows.
By using Theorem~\ref{t-gue140205}, we can repeat the procedure above and obtain \eqref{e-gue140212m}. 
We omit the details.
\end{proof}

In the rest of this section, we will study the kernel  $(S^*_-+S^*_+)Q^*Q(S_-+S_+)(x,y)$.
We will use the notations and assumptions used in Theorem~\ref{t-gue140212} and until further
notice we assume that $q=n_-$.  Let $\varphi_-(x,y)\in \cC^\infty(D\times D)$,
$\varphi_+(x,y)\in \cC^\infty(D\times D)$ be as in Theorem~\ref{t-gue140205I}.
We need the following result, which is essentially well-known and follows from the stationary
phase formula of Melin-Sj\"ostrand~\cite{MS74} (see also \cite[p.\,76-77]{Hsiao08} for more details).
\begin{lem}\label{l-gue140214}
There is a complex valued phase function $\varphi\in \cC^\infty(D\times D)$
with $\varphi(x, x)=0$, $d_x\varphi(x, y)|_{x=y}=-\omega_0(x)$,
$d_y\varphi(x, y)|_{x=y}=\omega_0(x)$ and $\varphi(x,y)$ satisfies
\eqref{e-gue140205VI} such that for any properly supported operators 
$B, C:\mathscr D'(D,T^{*0,q}X)\To\mathscr D'(D,T^{*0,q}X)$,
\[\begin{split}
&B=\int^\infty_0e^{i\varphi_-(x,y)t}b(x,y,t)dt,\ \ C=\int^\infty_0e^{i\varphi_-(x,y)t}c(x,y,t)dt,\\
\end{split}\]
with $b(x,y,t), c(x,y,t)\in S^{n-1}_{{\rm cl\,}}(D\times D\times\mathbb{R}_+,T^{*0,q}X\boxtimes T^{*0,q}X),$
we have
\[
\mbox{$B\circ C\equiv e^{i\varphi(x,y)t}d(x,y,t)dt$ on $D$},\]
where $d(x,y,t)\in S^{n-1}_{{\rm cl\,}}(D\times D\times\mathbb{R}_+,T^{*0,q}X\boxtimes T^{*0,q}X)$ 
and the leading term $d_0(x,y)$ of the expansion \eqref{e-fal} of $d(x,y,t)$ 
satisfies
\begin{equation}\label{e-gue140215a}
d_0(x,x)=2\pi^{n}\frac{m(x)}{v(x)}\abs{{\rm det\,}\mathcal{L}_x}^{-1}b_0(x,x)c_0(x,x),\ \ \forall x\in D,
\end{equation}
where $b_0(x,y)$, $c_0(x,y)$ denote the leading terms of the expansions \eqref{e-fal} of $b(x,y,t)$, $c(x,y,t)$ respectively. 
\end{lem}

We postpone the proof of the following Theorem for Section~\ref{s-gue140215}.

\begin{thm}\label{t-gue140215}
With the notations and assumptions used above, there is a $g(x,y)\in \cC^\infty(D\times D)$ with $g(x,x)=1$ such that
\begin{equation}\label{e-gue140214}
\mbox{$\varphi(x,y)-g(x,y)\varphi_-(x,y)$ vanishes to infinite order at $x=y$}.
\end{equation}
\end{thm}

From Lemma~\ref{l-gue140214} and Theorem~\ref{t-gue140215}, we deduce

\begin{cor}\label{t-gue140216}
In the conditions of Lemma \ref{l-gue140214} we have
\[
\mbox{$B\circ C\equiv e^{i\varphi_-(x,y)t}e(x,y,t)dt$ on $D$},
\]
where $e(x,y,t)\in S^{n-1}_{{\rm cl\,}}(D\times D\times\mathbb{R}_+,T^{*0,q}X\boxtimes T^{*0,q}X)$ and
the leading term $e_0(x,y)$ of the expansion \eqref{e-fal} of $e(x,y,t)$ satisfies \eqref{e-gue140215a}.
\end{cor}

Similarly, we can repeat the proof of Corollary~\ref{t-gue140216} and conclude that

\begin{thm}\label{t-gue140216I}
With the notations and assumptions used in Theorem~\ref{t-gue140212}, let $q=n_+$. 
For any properly supported operators 
$\mathcal{B}, \mathcal{C}:\mathscr D'(D,T^{*0,q}X)\To\mathscr D'(D,T^{*0,q}X)$,
\[\begin{split}
&\mathcal{B}=\int^\infty_0e^{i\varphi_+(x,y)t}b(x,y,t)dt,\ \ \mathcal{C}=\int^\infty_0e^{i\varphi_+(x,y)t}c(x,y,t)dt,\\
\end{split}\]
with $b(x,y,t), c(x,y,t)\in S^{n-1}_{{\rm cl\,}}(D\times D\times\mathbb{R}_+,T^{*0,q}X\boxtimes T^{*0,q}X)$, 
we have
\[
\mbox{$\mathcal{B}\circ \mathcal{C}\equiv e^{i\varphi_+(x,y)t}f(x,y,t)dt$ on $D$},\]
where $f(x,y,t)\in S^{n-1}_{{\rm cl\,}}(D\times D\times\mathbb{R}_+,T^{*0,q}X\boxtimes T^{*0,q}X)$ 
and the leading term $f_0(x,y)$ of the expansion \eqref{e-fal} of $f(x,y,t)$ satisfies \eqref{e-gue140215a}.
\end{thm}


We also need the following.
\begin{lem}\label{l-gue140216}
With the notations and assumptions used in Theorem~\ref{t-gue140212}, let $q=n_+=n_-$. For any
properly supported operators 
$\mathcal{B}, C:\mathscr D'(D,T^{*0,q}X)\To\mathscr D'(D,T^{*0,q}X)$,
\[\mathcal{B}=\int^\infty_0e^{i\varphi_+(x,y)t}b(x,y,t)dt,\ \ C=\int^\infty_0e^{i\varphi_-(x,y)t}c(x,y,t)dt,\]
 where $b(x,y,t), c(x,y,t)\in S^{n-1}_{{\rm cl\,}}(D\times D\times\mathbb{R}_+,T^{*0,q}X\boxtimes T^{*0,q}X)$, 
 we have
\[\mathcal{B}\circ C\equiv0\quad\text{and}\quad C\circ\mathcal{B}\equiv0\quad\text{on $D$}.\]
\end{lem}
\begin{proof}
We first notice that $\mathcal{B}\circ C$ is smoothing away $x=y$.
We also write $w=(w_1,\ldots,w_{2n-1})$ to denote local coordinates on $D$.
We have
\begin{equation}\label{e-gue140216f}
\begin{split}
\mathcal{B}\circ C(x,y)&=\int_{\sigma>0,t>0}e^{i\varphi_+(x,w)\sigma+i\varphi_-(w,y)t}b(x,w,\sigma)c(w,y,t)d\sigma m(w)dt\\
&=\int_{s>0,t>0}e^{it(\varphi_+(x,w)s+\varphi_-(w,y))}tb(x,w,st)c(w,y,t)ds m(w)dt.
\end{split}\end{equation}
Take $\chi\in \cC^\infty_0(\Real,[0,1])$ with $\chi=1$ on 
$[-\frac{1}{2},\frac{1}{2}]$, $\chi=0$ on $]-\infty,-1]\bigcup[1,\infty[$. 
From \eqref{e-gue140216f}, we have
\begin{equation}\label{e-gue140216fI}
\begin{split}
&\mathcal{B}\circ C(x,y)=I_\varepsilon+II_\varepsilon,\\
&I_\varepsilon=\int_{s>0,t>0}e^{it(\varphi_+(x,w)s+\varphi_-(w,y))}\chi\bigg(\frac{\abs{x-w}^2}{\varepsilon}\bigg)
tb(x,w,st)c(w,y,t)ds m(w)dt,\\
&II_\varepsilon=\int_{s>0,t>0}e^{it(\varphi_+(x,w)s+\varphi_-(w,y))}
\bigg(1-\chi\bigg(\frac{\abs{x-w}^2}{\varepsilon}\bigg)\bigg)\\
&\quad\quad\quad\times tb(x,w,st)c(w,y,t)ds m(w)dt,
\end{split}\end{equation}
where $\varepsilon>0$ is a small constant. Since $\varphi_+(x,w)=0$ if and only if $x=w$, 
we can integrate by parts with respect to $s$ and conclude that $II_\varepsilon$ is smoothing. 
Since $\mathcal{B}\circ C$ is smoothing away $x=y$, we may assume that $\abs{x-y}<\varepsilon$.
Since $d_w(\varphi_+(x,w)s+\varphi_-(w,y))|_{x=y=w}=-\omega_0(x)(s+1)\neq0$, if $\varepsilon>0$ 
is small, we can integrate by parts with respect to $w$ and conclude that $I_\varepsilon\equiv0$ on $D$.
We get $\mathcal{B}\circ C\equiv0$ on $D$. Similarly, we can repeat the procedure above and conclude 
that $C\circ\mathcal{B}\equiv0$ on $D$. The lemma follows.
\end{proof}
Recalling Definition~\ref{d-gue140221} we see that Theorem~\ref{t-gue140216}, 
Theorem~\ref{t-gue140216I} and Lemma~\ref{l-gue140216} yield:
\begin{thm}\label{t-gue140213}
With the notations and assumptions used in Theorem~\ref{t-gue140212}, let $q\in\set{n_-,n_+}$. Then,
\begin{equation}\label{e-gue140213VIII}
\begin{split}
&(S^*_-+S^*_+)Q^*Q(S_-+S_+)(x,y)\\
&\equiv\int^\infty_0e^{i\varphi_-(x,y)t}a_-(x,y,t)dt+\int^\infty_0e^{i\varphi_+(x,y)t}a_+(x,y,t)dt\, \:\:\text{on $D$},
\end{split}
\end{equation}
where $\varphi_\pm(x,y)\in \cC^\infty(D\times D)$ are as in Theorem~\ref{t-gue140305_b}, 
$a_\pm(x,y,t)\in S^{n-1}_{{\rm cl\,}}\big(D\times D\times\mathbb{R}_+,T^{*0,q}_yX\boxtimes T^{*0,q}_xX\big)$,
\begin{equation}  \label{e-gue140213aI}\begin{split}
&\mbox{$a_-(x,y,t)=0$ if $q\neq n_-$ or $Q\equiv0$ at $\Sigma^-\cap T^*D$},\\
&\mbox{ $a_+(x,y,t)=0$ if $q\neq n_+$ or $Q\equiv0$ at $\Sigma^+\cap T^*D$}\,.
\end{split}\end{equation}
Moreover, assume that $q=n_-$, then, for the leading term $a^0_-(x,y)$ of the expansion \eqref{e-fal} of $a_-(x,y,t)$ satisfies 
\begin{equation}  \label{e-gue140213aII}
\begin{split}
a^0_-(x, x)
=\frac{1}{2}\pi^{-n}\frac{v(x)}{m(x)}
\abs{{\rm det\,}\mathcal{L}_x}\tau_{x,n_-}q^*(x,-\omega_0(x))q(x,-\omega_0(x))\tau_{x,n_-},\ \ \forall x\in D,
\end{split}
\end{equation}
where $\det\mathcal{L}_x$ is the determinant of the Levi form defined in \eqref{det140530}, $v(x)$ 
is the volume form on $X$ induced by $\langle\,\cdot\,|\,\cdot\,\rangle$, $q(x,\eta)\in \cC^\infty(T^*D)$ 
is the principal symbol of $Q$, $q^*(x,\eta)$ is the adjoint of $q(x,\eta):T^{*0,q}_xX\To T^{*0,q}_xX$ 
with respect to $\langle\,\cdot\,|\,\cdot\,\rangle$ and $\tau_{x,n_-}$ is as in \eqref{tau140530}.
\end{thm}
\begin{proof}[Proof of Theorem \ref{t-gue140305VIa}]
From Theorem~\ref{t-gue140212} and Theorem~\ref{t-gue140213}, 
we get Theorem~\ref{t-gue140305VIab} and Theorem~\ref{t-gue140305VIa}.
\end{proof}
\begin{proof}[Proof of Theorem \ref{t-emb}] 
Fix $p\in D$, let $\{W_j\}_{j=1}^{n-1}$
be an orthonormal frame of $T^{1,0}X$ in a neighbourhood of $p$
such that the Levi form is diagonal at $p$. We take local coordinates
$x=(x_1,\ldots,x_{2n-1})$, $z_j=x_{2j-1}+ix_{2j}$, $j=1,\ldots,n-1$,
defined on some neighbourhood of $p$ such that $\omega_0(x_0)=dx_{2n-1}$, $x(p)=0$, 
and for some $c_j\in\Complex$, $j=1,\ldots,n-1$\,,
\[W_j=\frac{\pr}{\pr z_j}-i\mu_j\ol z_j\frac{\pr}{\pr x_{2n-1}}-
c_jx_{2n-1}\frac{\pr}{\pr x_{2n-1}}+O(\abs{x}^2),\ j=1,\ldots,n-1\,.\]
For $x=(x_1,x_2,\ldots,x_{2n-1})$, we write $x'=(x_1,x_2,\ldots,x_{2n-2})$. 
Take $\chi\in\cC^\infty_0(]-\varepsilon_0,\varepsilon_0[)$, $\chi=1$ near $0$, $\chi(t)=\chi(-t)$, where $\varepsilon_0>0$ 
is a small constant. Take $\varepsilon_0>0$ small enough so that $D'\times]-\varepsilon_0,\varepsilon_0[\Subset D$, 
where $D'$ is an open neighbourhood of $0\in\Real^{2n-2}$. For each $k>0$, we consider the operator 
\[E_k:u\in\cC^\infty_0(D')\To(Q^*\Pi^{(0)}Q)(e^{-iky_{2n-1}}\chi(y_{2n-1})u(y'))\in\cC^\infty(X).\]
From the stationary phase formula of Melin-Sj\"ostrand~\cite{MS74}, we can check that $E_k$ is smoothing 
and the kernel of $E_k$ satisfies 
\begin{equation}\label{e-f}
E_k(x,y')\equiv e^{ik\Phi(x,y')}g(x,y',k)\mod O(k^{-\infty}),
\end{equation}
where $g(x,y',k)\in\cC^\infty$, $g(x,y',k)\sim\sum\limits^\infty_{j=0}g_j(x,y')k^{n-1-j}$ in $S^{n-1}_{{\rm loc\,}}(1)$, 
$g_j(x,y')\in\cC^\infty$, $j=0,1,2,\ldots$, $g_0(x,x')\neq0$, $\Phi\in\cC^\infty$, ${\rm Im\,}\Phi\geq0$, $\Phi(x,x')=0$ and 
\begin{equation} \label{e-fI}
\begin{split}
&\Phi(x, y')=-x_{2n-1}+i\sum^{n-1}_{j=1}\mu_j\abs{z_j-w_j}^2 \\
&\quad+\sum^{n-1}_{j=1}\Bigr(i\mu_j(\ol z_jw_j-z_j\ol w_j)-c_jz_jx_{2n-1}-\ol c_j\ol z_jx_{2n-1}\Bigr)\\&\quad+
x_{2n-1}f(x, y') +O(\abs{(x, y')}^3),\\
&\quad f\in\cC^\infty,\ \ f(0,0)=0,\ \ w_j=y_{2j-1}+iy_{2j},\ \ j=1,\ldots,n-1.
\end{split}
\end{equation}
(See Section~\ref{s-gue140215} for the details and the precise meanings of 
$A\equiv B\mod O(k^{-\infty})$ and $S^{n-1}_{{\rm loc\,}}(1)$.) Put 
\begin{equation}\label{e-fII}
u_k(x):=E_k(\chi(ky_1)\chi(ky_2)\ldots\chi(ky_{2n-2})k^{2n-2}).
\end{equation}
Then $u_k(x)$ is a global smooth CR function on $X$. From \eqref{e-f} and \eqref{e-fI}, we can check that 
\begin{equation}\label{e-fIII}
\begin{split}
\lim_{k\To\infty}k^{-n}\frac{\pr u_k}{\pr x_{2n-1}}(0)&=\lim_{k\To\infty}k^{-n}\int e^{ik\Phi(0,y')}(-ik)g(0,y',k)\chi(ky_1)
\ldots\chi(ky_{2n-2})k^{2n-2}dy'\\
&=(-i)g_0(0,0)\int\chi(y_1)\ldots\chi(y_{2n-2})dy',\\
\lim_{k\To\infty}k^{-n}\frac{\pr u_k}{\pr x_{t}}(0)&=0,\ \ t=1,2,\ldots,2n-2.
\end{split}
\end{equation}
For any $s\in\set{1,2,\ldots,n-1}$, put 
\begin{equation}\label{e-fIV}
u^s_k(x):=E_k(k(y_{2s-1}+iy_{2s})\chi(ky_1)\chi(ky_2)\ldots\chi(ky_{2n-2})k^{2n-2}).
\end{equation}
Then $u^s_k(x)$ is a global smooth CR function on $X$, $s=1,2,\ldots,n-1$. 
From \eqref{e-f}, \eqref{e-fI} and notice that $\ddbar_bu^s_k=0$, $s=1,2,\ldots,n-1$, we can check that 
\begin{equation}\label{e-fV}
\begin{split}
&\lim_{k\To\infty}k^{-n+1}\frac{\pr u^s_k}{\pr z_s}(0)\\
&=\lim_{k\To\infty}k^{-n+1}\int e^{ik\Phi(0,y')}2k^2\mu_s\abs{y_{2s-1}+iy_{2s}}^2g(0,y',k)\chi(ky_1)
\ldots\chi(ky_{2n-2})k^{2n-2}dy'\\
&=2\mu_sg_0(0,0)\int\abs{y_{2s-1}+iy_{2s}}^2\chi(y_1)\ldots\chi(y_{2n-2})dy',\\
&\lim_{k\To\infty}k^{-n+1}\frac{\pr u^s_k}{\pr\ol z_{t}}(0)=0,\ \ t=1,2,\ldots,n-1,
\end{split}
\end{equation}
and for $t\neq s$, $t\in\set{1,2,\ldots,n-1}$, we have 
\begin{equation}\label{e-fVI}
\begin{split}
&\lim_{k\To\infty}k^{-n+1}\frac{\pr u^s_k}{\pr z_t}(0)\\
&=\lim_{k\To\infty}k^{-n+1}\int e^{ik\Phi(0,y')}2k^2\mu_t(y_{2t-1}-iy_{2t})(y_{2s-1}+iy_{2s})g(0,y',k)\chi(ky_1)
\ldots\chi(ky_{2n-2})k^{2n-2}dy'\\
&=2\mu_tg_0(0,0)\int(y_{2t-1}-iy_{2t})(y_{2s-1}+iy_{2s})\chi(y_1)\ldots\chi(y_{2n-2})dy'=0.
\end{split}
\end{equation}
From \eqref{e-fIII}, \eqref{e-fV} and \eqref{e-fVI}, it is not difficult to check that for $k$ large, the differential of the CR map
\[x\in X\To (u_k(x),u^1_k(x),\ldots,u^{n-1}_k(x))\in\Complex^n\]
is injective at $p$. Thus, near $p$, the map $x\in X\To (u_k(x),u^1_k(x),\ldots,u^{n-1}_k(x))\in\Complex^n$ 
is a CR embedding. Theorem~\ref{t-emb} follows. 
\end{proof}
\section{Szeg\H{o} projections on CR manifolds with transversal CR $S^1$ actions} \label{s-gue140211}

In this section, we will apply Theorem~\ref{t-gue140305VIa} to establish Szeg\H{o} kernel asymptotic
expansions on compact CR manifolds with transversal CR $S^1$ actions under certain Levi curvature
assumptions. As an application, we will show that if $X$ is a $3$-dimensional compact strictly
pseudoconvex CR manifold with a transversal CR $S^1$ action, then $X$ can be CR embedded
into $\Complex^N$, for some $N\in\mathbb N$. We introduce some notations first.

Let $(X, T^{1,0}X)$ be a CR manifold. Let assume that $X$
admits a $S^1$ action $S^1\times X\To X$, $(e^{i\theta},x)\mapsto e^{i\theta}x$.
Let $T\in \cC^\infty(X,TX)$ be the real vector field given by
\begin{equation}\label{e-gue140211bIII-I}
Tu=\frac{\pr}{\pr\theta}u(e^{i\theta}x)\Big|_{\theta=0}\,,\ \ u\in \cC^\infty(X).
\end{equation}
We call $T$ the global vector field induced by the $S^1$ action or the infinitesimal generator of the action.

\begin{defn}\label{d-gue140211I}
We say that the $S^1$ action $e^{i\theta}$ is CR if 
\[\big[T,\cC^\infty(X,T^{1,0}X)\big]\subset \cC^\infty(X,T^{1,0}X)\]
and is transversal if for every point $x\in X$,
\[T(x)\oplus T^{1,0}_xX\oplus T^{0,1}_xX=\Complex T_xX.\]
\end{defn}
Until further notice, we assume that $(X, T^{1,0}X)$ is a CR manifold with a transversal CR $S^1$ action
and we let $T$ be the global vector field induced by the $S^1$ action.

Fix $\theta_0\in[0,2\pi[$. Let $de^{i\theta_0}:\Complex T_xX\To\Complex T_{e^{i\theta_0}x}X$
denote the differential of the map $e^{i\theta_0}:X\To X$.

\begin{defn}\label{d-gue140211II}
Let $U\subset X$ be an open set and let $V\in \cC^\infty(U,\Complex TX)$ be a vector field on $U$. We say that $V$ is $T$-rigid if
$de^{i\theta_0}V(x)=V(x),\ \ \forall x\in e^{i\theta_0}U\cap U$,
for every $\theta_0\in[0,2\pi[$ with $e^{i\theta_0}U\cap U\neq\emptyset$.
\end{defn}

We also need

\begin{defn}\label{d-gue140211III}
Let $\langle\,\cdot\,|\,\cdot\,\rangle$ be a Hermitian metric on $\Complex TX$.
We say that $\langle\,\cdot\,|\,\cdot\,\rangle$
is $T$-rigid if for $T$-rigid vector fields $V$ and $W$ on $U$, where $U\subset X$ is any open set, we have
\[\langle\,V(x)\,|\,W(x)\,\rangle=\langle\,de^{i\theta_0}V(e^{i\theta_0}x)\,|\,
de^{i\theta_0}W(e^{i\theta_0}x)\,\rangle,\ \ \forall x\in U, \theta_0\in[0,2\pi[.\]
\end{defn}

The following result was established in~\cite[Theorem 9.2]{Hsiao14}.

\begin{thm}\label{t-gue131206I}
There is a $T$-rigid Hermitian metric $\langle\,\cdot\,|\,\cdot\,\rangle$ on $\Complex TX$ 
such that $T^{1,0}X\perp T^{0,1}X$, $T\perp (T^{1,0}X\oplus T^{0,1}X)$, $\langle\,T\,|\,T\,\rangle=1$ 
and $\langle\,u\,|v\,\rangle$ is real if $u, v$ are real tangent vectors.
\end{thm}

Until further notice, we fix a $T$-rigid Hermitian metric $\langle\,\cdot\,|\,\cdot\,\rangle$ on $\Complex TX$ 
such that $T^{1,0}X\perp T^{0,1}X$, $T\perp (T^{1,0}X\oplus T^{0,1}X)$, $\langle\,T\,|\,T\,\rangle=1$ 
and $\langle\,u\,|v\,\rangle$ is real if $u, v$ are real tangent vectors and we take $m(x)$ 
to be the volume form induced by the given $T$-rigid Hermitian metric $\langle\,\cdot\,|\,\cdot\,\rangle$. 
We will use the same notations as before.
We need the following result due to Baouendi-Rothschild-Treves~\cite[Section1]{BRT85}

\begin{thm}\label{t-gue131206}
For every point $x_0\in X$, there exists local coordinates $x=(x_1,\ldots,x_{2n-1})=(z,\theta)=(z_1,\ldots,z_{n-1},\theta)$,
$z_j=x_{2j-1}+ix_{2j}$, $j=1,\ldots,n-1$, $\theta=x_{2n-1}$, defined in some small neighbourhood $U$ of $x_0$ such that
\begin{equation}\label{e-gue131206}
\begin{split}
&T=\frac{\pr}{\pr\theta}\,,\:\: Z_j=\frac{\pr}{\pr z_j}+i\frac{\pr\phi}{\pr z_j}(z)\frac{\pr}{\pr\theta},\ \ j=1,\ldots,n-1,
\end{split}
\end{equation}
where $Z_j(x)$, $j=1,\ldots,n-1$, form a basis of $T^{1,0}_xX$, for each $x\in U$, and $\phi(z)\in \cC^\infty(U,\Real)$ 
is independent of $\theta$.
\end{thm}

Let $x=(x_1,\ldots,x_{2n-1})=(z,\theta)=(z_1,\ldots,z_{n-1},\theta)$,
$z_j=x_{2j-1}+ix_{2j}$, $j=1,\ldots,n-1$, $\theta=x_{2n-1}$, be canonical coordinates of $X$ 
defined in some open set $D\Subset X$. It is clearly that
\[\set{d\ol z_{j_1}\wedge\ldots\wedge d\ol z_{j_q};\, 1\leq j_1<j_2<\ldots<j_q\leq n-1}\]
is a basis for $T^{*0,q}_xX$, for every $x\in D$. Let $u\in\Omega^{0,q}(X)$. On $D$, we write
\[\begin{split}
&u=\sum_{1\leq j_1<j_2<\ldots<j_q\leq n-1}u_{j_1,\ldots,j_q}d\ol z_{j_1}\wedge\ldots\wedge d\ol z_{j_q}\,,
\:\:u_{j_1,\ldots,j_q}\in \cC^\infty(D)\,.
\end{split}\]
On $D$, we define
\begin{equation}\label{e-gue140218}
Tu:=\sum_{1\leq j_1<j_2<\ldots<j_q\leq n-1}(Tu_{j_1,\ldots,j_q})d\ol z_{j_1}\wedge\ldots\wedge d\ol z_{j_q}.
\end{equation}
Let
$y=(y_1,\ldots,y_{2n-1})=(w,\gamma)$, $w_j=y_{2j-1}+iy_{2j}$, $j=1,\ldots,n-1$, 
$\gamma=y_{2n-1}$, be another canonical coordinates on $D$. Then,
\begin{equation}\label{sp3-eVI}\begin{split}
T=\frac{\pr}{\pr\gamma}\,,\:\:\Td Z_j=
\frac{\pr}{\pr w_j}+i\frac{\pr\Td\phi}{\pr w_j}(w)\frac{\pr}{\pr\gamma},\ \ j=1,\ldots,n-1,
\end{split}
\end{equation}
where $\Td Z_j(y)$, $j=1,\ldots,n-1$, form a basis of $T^{1,0}_yX$, for each $y\in D$, 
and $\Td\phi(w)\in \cC^\infty(D,\Real)$ independent of $\gamma$. 
From \eqref{sp3-eVI} and \eqref{e-gue131206}, it is not difficult to see that on $D$, we have
\begin{equation}\label{sp3-eVII}
\begin{split}
&w=(w_1,\ldots,w_{n-1})=(H_1(z),\ldots,H_{n-1}(z))=H(z),\ \ H_j(z)\in \cC^\infty,\ \ \forall j,\\
&\gamma=\theta+G(z),\ \ G(z)\in \cC^\infty,
\end{split}
\end{equation}
where for each $j=1,\ldots,n-1$, $H_j(z)$ is holomorphic. From \eqref{sp3-eVII}, we can check that
\begin{equation}\label{spca}
d\ol w_j=\sum^{n-1}_{l=1}\ol{\frac{\pr H_j}{\pr z_l}}d\ol z_l,\ \ j=1,\ldots,n-1.
\end{equation}
From \eqref{spca}, it is straightforward to check that the definition \eqref{e-gue140218} 
is independent of the choice of canonical coordinates. 
We omit the details (see also \cite[Section\,5]{Hsiao12}). Thus, $Tu$ is well-defined as an element in $\Omega^{0,q}(X)$.

For $m\in\mathbb Z$, put
\begin{equation}\label{e-gue140218I}
B^{0,q}_m(X):=\set{u\in\Omega^{0,q}(X);\, Tu=-imu}
\end{equation}
and let $\mathcal{B}^{0,q}_m(X)\subset L^2_{(0,q)}(X)$ be the completion of $B^{0,q}_m(X)$ 
with respect to $(\,\cdot\,|\,\cdot\,)$. It is easy to see that for any $m, m'\in\mathbb Z$, $m\neq m'$,
\begin{equation}\label{e-gue140218II}
(\,u\,|\,v)=0,\ \ \forall u\in\mathcal{B}^{0,q}_m(X), v\in\mathcal{B}^{0,q}_{m'}(X).
\end{equation}
We have actually an orthogonal decomposition of Hilbert spaces
\[L^2_{(0,q)}(X)=\widehat{\bigoplus}_{m\in\mathbb Z}\mathcal{B}^{0,q}_m(X)\,.\]
For $m\in\mathbb Z$, let
\begin{equation}\label{e-gue140218III}
Q^{(q)}_{m}:L^2_{(0,q)}(X)\To\mathcal{B}^{0,q}_m(X)
\end{equation}
be the orthogonal projection with respect to $(\,\cdot\,|\,\cdot\,)$. 
Moreover, it is not difficult to see that for every $m\in\mathbb Z$, we have
\begin{equation}\label{e-gue140218IV}
\begin{split}
&Q^{(q)}_{m}:\Omega^{0,q}(X)\To B^{0,q}_m(X),\\
&TQ^{(q)}_{m}=-imQ^{(q)}_{m}u,\ \ \forall u\in L^2_{(0,q)}(X),\\
&\norm{TQ^{(q)}_{m}u}=\abs{m}\norm{Q^{(q)}_{m}u},\ \ \forall u\in L^2_{(0,q)}(X).
\end{split}
\end{equation}
Since the Hermitian metric $\langle\,\cdot\,|\,\cdot\,\rangle$ is $T$-rigid, 
it is straightforward to see that (see~\cite[Section 5]{Hsiao12})
\begin{equation}\label{e-gue131207VI}
\begin{split}
&\mbox{$\Box^{(q)}_{b}Q^{(q)}_{m}=Q^{(q)}_{m}\Box^{(q)}_{b}$ on 
$\Omega^{0,q}_0(X,)$, $\forall m\in\mathbb Z$},\\
&\mbox{$\ddbar_{b}Q^{(q)}_{m}=Q^{(q+1)}_{m}\ddbar_{b}$ on 
$\Omega^{0,q}_0(X)$, $\forall m\in\mathbb Z$, $q=0,1,\ldots,n-2$},\\
&\mbox{$\ddbar^*_{b}Q^{(q)}_{m}=Q^{(q-1)}_{m}\ddbar^*_{b}$ on
$\Omega^{0,q}_0(X)$, $\forall m\in\mathbb Z$, $q=1,\ldots,n-1$}.
\end{split}
\end{equation}
Now, we assume that $X$ is compact. By using elementary Fourier analysis, 
it is straightforward to see that for every $u\in\Omega^{0,q}(X)$,
\begin{equation}\label{e-gue140218V}
\begin{split}
&\mbox{$\lim\limits_{N\To\infty}\sum\limits^N_{m=-N}Q^{(q)}_{m}u= u$ in the $\cC^\infty$ topology},\\
&\sum^N_{m=-N}\norm{Q^{(q)}_{m}u}^2\leq\norm{u}^2,\ \ \forall N\in\mathbb N_0.
\end{split}
\end{equation}
Thus, for every $u\in L^2_{(0,q)}(X)$,
\begin{equation}\label{e-gue140218VI}
\begin{split}
&\mbox{$\lim\limits_{N\To\infty}\sum\limits^N_{m=-N}Q^{(q)}_{m}u=u$ in $L^2_{(0,q)}(X,L^k)$},\\
&\sum^N_{m=-N}\norm{Q^{(q)}_{m}u}^2\leq\norm{u}^2,\ \ \forall N\in\mathbb N_0.
\end{split}
\end{equation}
For $m\in\mathbb Z$, put
\begin{equation}\label{e-gue140218VII}
\begin{split}
Q^{(q)}_{\leq m}:L^2_{(0,q)}(X)\To L^2_{(0,q)}(X),\quad
u\longmapsto\lim\limits_{N\To\infty}\sum\limits^N_{j=0}Q^{(q)}_{m-j}u,
\end{split}
\end{equation}
and
\begin{equation}\label{e-gue140218VIII}
\begin{split}
Q^{(q)}_{\geq m}:L^2_{(0,q)}(X)\To L^2_{(0,q)}(X),\quad
u\longmapsto\lim\limits_{N\To\infty}\sum\limits^N_{j=0}Q^{(q)}_{m+j}u.
\end{split}
\end{equation}
In view of \eqref{e-gue140218V} and \eqref{e-gue140218VI}, we see that \eqref{e-gue140218VII} 
and \eqref{e-gue140218VIII} are well-defined.

The following is straightforward and we omit the proof.

\begin{thm} \label{t-gue140218}
Let $m\in\mathbb Z$, we have
\begin{equation}\label{e-gue140218a}
\begin{split}
&Q^{(q)}_{\geq m}, Q^{(q)}_{\leq m}:\Omega^{0,q}(X)\To\Omega^{0,q}(X),\\
&i(\,TQ^{(q)}_{\geq m}u\,|\,u)\geq m\norm{u},\ \ \forall u\in\Omega^{0,q}(X),\\
&i(\,TQ^{(q)}_{\leq m}u\,|\,u)\leq m\norm{u},\ \ \forall u\in\Omega^{0,q}(X),\\
&Q^{(q)}_{\geq m}, Q^{(q)}_{\leq m}:{\rm Dom\,}\ddbar_b\To {\rm Dom\,}\ddbar_b,\\
&\mbox{$Q^{(q)}_{\geq m}\ddbar_b=\ddbar_bQ^{(q)}_{\geq m}$ on ${\rm Dom\,}\ddbar_b$},\\
&\mbox{$Q^{(q)}_{\leq m}\ddbar_b=\ddbar_bQ^{(q)}_{\leq m}$ on ${\rm Dom\,}\ddbar_b$},\\
&Q^{(q)}_{\geq m}, Q^{(q)}_{\leq m}:{\rm Dom\,}\ol{\pr}^*_b\To {\rm Dom\,}\ol{\pr}^*_b,\\
&\mbox{$Q^{(q)}_{\geq m}\ol{\pr}^*_b=\ol{\pr}^*_bQ^{(q)}_{\geq m}$ on ${\rm Dom\,}\ol{\pr}^*_b$},\\
&\mbox{$Q^{(q)}_{\leq m}\ol{\pr}^*_b=\ol{\pr}^*_bQ^{(q)}_{\leq m}$ on ${\rm Dom\,}\ol{\pr}^*_b$},\\
&Q^{(q)}_{\geq m}, Q^{(q)}_{\leq m}:{\rm Dom\,}\Box^{(q)}_b\To {\rm Dom\,}\Box^{(q)}_b,\\
&\mbox{$Q^{(q)}_{\geq m}\Box^{(q)}_b=\Box^{(q)}_bQ^{(q)}_{\geq m}$ on ${\rm Dom\,}\Box^{(q)}_b$},\\
&\mbox{$Q^{(q)}_{\leq m}\Box^{(q)}_b=\Box^{(q)}_bQ^{(q)}_{\leq m}$ on ${\rm Dom\,}\Box^{(q)}_b$},\\
&\mbox{$Q^{(q)}_{\geq m}\Pi^{(q)}=\Pi^{(q)}Q^{(q)}_{\geq m}$ on $L^2_{(0,q)}(X)$},\\
&\mbox{$Q^{(q)}_{\leq m}\Pi^{(q)}=\Pi^{(q)}Q^{(q)}_{\leq m}$ on $L^2_{(0,q)}(X)$}.
\end{split}
\end{equation}
\end{thm}
To continue, put for $m\in\mathbb Z$,
\begin{equation}\label{e-gue140218aI}
\begin{split}
&\mathcal{B}^{0,q}_{\geq m}(X):=\set{Q^{(q)}_{\geq m}u;\, u\in L^2_{(0,q)}(X)},\\
&\mathcal{B}^{0,q}_{\leq m}(X):=\set{Q^{(q)}_{\leq m}u;\, u\in L^2_{(0,q)}(X)}.
\end{split}
\end{equation}
Note that $(Q^{(q)}_{\leq m})^2=Q^{(q)}_{\leq m}$, $(Q^{(q)}_{\geq m})^2=Q^{(q)}_{\geq m}$. 
From this observation and \eqref{e-gue140218a}, we see that
\[\begin{split}
&{\rm Dom\,}\Box^{(q)}_b\cap\mathcal{B}^{0,q}_{\leq m}=\set{Q^{(q)}_{\leq m}u;\, u\in{\rm Dom\,}\Box^{(q)}_b},\\
&{\rm Dom\,}\Box^{(q)}_b\cap\mathcal{B}^{0,q}_{\geq m}=\set{Q^{(q)}_{\geq m}u;\, u\in{\rm Dom\,}\Box^{(q)}_b},
\end{split}\]
and
\begin{equation}\label{e-gue140218aII}
\begin{split}
&\Box^{(q)}_b:{\rm Dom\,}\Box^{(q)}_b\cap\mathcal{B}^{0,q}_{\geq m}(X)\To\mathcal{B}^{0,q}_{\geq m}(X),\\
&\Box^{(q)}_b:{\rm Dom\,}\Box^{(q)}_b\cap\mathcal{B}^{0,q}_{\leq m}(X)\To\mathcal{B}^{0,q}_{\leq m}(X).
\end{split}
\end{equation}
Thus, it is quite interesting to study the behaviour of $\Box^{(q)}_b$
in the spaces $\mathcal{B}^{0,q}_{\geq m}$ and $\mathcal{B}^{0,q}_{\leq m}$.
We recall now the condition $Z(q)$ of H\"ormander.
\begin{defn}\label{d-gue140218}
Given $q\in\{0,\ldots,n-1\}$, the Levi form is said to satisfy condition $Z(q)$ at $p\in X$, if $\mathcal{L}_p$ has at least
$n-q$ positive eigenvalues or at least $q+1$ negative eigenvalues.
\end{defn}
Usually, the condition $Z(q)$ is introduced for a smooth domain $D$ with boundary $X=\partial D$
in a complex manifold $M$. Then condition $Z(q)$ implies subelliptic estimates for the
$\overline\partial$-Neumann problem on $D$, cf.\ \cite{CS01,FK:72}. 
If one wants to to obtain subelliptic estimates on $X$,
one cannot distinguish whether $X$ is the boundary of $D$ or $X$ is the boundary of the complement of $D$.
Thus, one assumes that condition $Z(q)$ holds on both $D$ and its complement $M\setminus D$.
Note that condition $Z(q)$ on $M\setminus D$ is equivalent to condition $Z(n-q-1)$ on $D$.
However, we show in the next theorem, that condition $Z(q)$ (resp.\ $Z(n-q-1)$) yields subelliptic estimates
on a $CR$ manifold with $S^1$ action, by projecting the forms with  $Q^{(q)}_{\leq0}$, (resp.\ $Q^{(q)}_{\geq0}$).
\begin{thm}\label{t-gue140218I}
With the notations and assumptions above, assume that $Z(q)$ holds at every point of $X$.
Then, for every $s\in\mathbb N_0$, there is a constant $C_s>0$ such that
\begin{equation}\label{e-gue140218aIII}
\norm{Q^{(q)}_{\leq0}u}_{s+1}\leq C_s\Bigr(\norm{\Box^{(q)}_bQ^{(q)}_{\leq0}u}_s+
\norm{Q^{(q)}_{\leq0}u}\Bigr),\ \ \forall u\in\Omega^{0,q}(X),
\end{equation}
where $\norm{\cdot}_s$ denotes the usual Sobolev norm of order $s$ on $X$.

Similarly, if $Z(n-1-q)$ holds at every point of $X$, then for every $s\in\mathbb N_0$, there is a constant $C_s>0$ such that
\begin{equation}\label{e-gue140218aIV}
\norm{Q^{(q)}_{\geq0}u}_{s+1}\leq C_s\Bigr(\norm{\Box^{(q)}_bQ^{(q)}_{\geq0}u}_s+
\norm{Q^{(q)}_{\geq0}u}\Bigr),\ \ \forall u\in\Omega^{0,q}(X).
\end{equation}
\end{thm}
\begin{proof}
If we go through Kohn's $L^2$ estimates (see~\cite[Theorem 8.4.2]{CS01}, \cite[Proposition 5.4.10]{FK:72}, \cite{Koh64}),
we see that:

\noindent
(I) If $Z(q)$ holds at every point of $X$, then, for every $s\in\mathbb N_0$,
there is a constant $C_s>0$ such that for all
$u\in\Omega^{0,q}(X)$ with $i(\,Tu\,|\,u\,)\leq0$, we have
\[\norm{u}_{s+1}\leq C_s\Bigr(\norm{\Box^{(q)}_bu}_s+\norm{u}\Bigr).\]
(II) If $Z(n-1-q)$ holds at every point of $X$, then, for every $s\in\mathbb N_0$,
there is a constant $\Td C_s>0$ such that for all
$u\in\Omega^{0,q}(X)$ with $i(\,Tu\,|\,u\,)\geq0$, we have
\[\norm{u}_{s+1}\leq\Td C_s\Bigr(\norm{\Box^{(q)}_bu}_s+\norm{u}\Bigr).\]
We notice that
\[i(\,TQ^{(q)}_{\leq0}u\,|\,Q^{(q)}_{\leq0}u\,)\leq0,\
\ i(\,TQ^{(q)}_{\geq0}u\,|\,Q^{(q)}_{\geq0}u\,)\geq0,\ \ \forall u\in\Omega^{0,q}(X).\]
From this observation and (I) and (II), the theorem follows.
\end{proof}
For every $s\in\mathbb Z$, let $H^s_-(X,T^{*0,q}X)$ and $H^s_+(X,T^{*0,q}X)$
denote the completions of $\mathcal{B}^{0,q}_{\leq0}(X)\cap\Omega^{0,q}(X)$
and $\mathcal{B}^{0,q}_{\geq0}(X)\cap\Omega^{0,q}(X)$ with respect to $\norm{\cdot}_s$ respectively.
Let $\mathscr D'_-(X,T^{*0,q}X)$ and $\mathscr D'_+(X,T^{*0,q}X)$ denote the dual spaces
of $\mathcal{B}^{0,q}_{\leq0}(X)\cap\Omega^{0,q}(X)$ and
$\mathcal{B}^{0,q}_{\geq0}(X)\cap\Omega^{0,q}(X)$ respectively.

From Theorem~\ref{t-gue140218I}, we can repeat the method of Kohn
(see~\cite[Chapter 8]{CS01}, \cite{FK:72}, \cite{Koh64}) and deduce the following.

\begin{thm}\label{t-gue140218II}
With the notations and assumptions above, assume that $Z(q)$ holds at every point of $X$. Then
$\Box^{(q)}_b:{\rm Dom\,}\Box^{(q)}_b\cap\mathcal{B}^{0,q}_{\leq0}(X)\To\mathcal{B}^{0,q}_{\leq0}(X)$
has closed range. Let
\[N^{(q)}_-:\mathcal{B}^{0,q}_{\leq0}(X)\To{\rm Dom\,}\Box^{(q)}_b\cap\mathcal{B}^{0,q}_{\leq0}(X)\]
be the associated partial inverse and let
\[\Pi^{(q)}_-:\mathcal{B}^{0,q}_{\leq0}(X)\To{\rm Ker\,}\Box^{(q)}_b\]
be the orthogonal projection. Then, we have
\begin{equation}\label{e-gue140218aV}
\begin{split}
&\mbox{$\Box^{(q)}_bN^{(q)}_-+\Pi^{(q)}_-=I$ on $\mathcal{B}^{0,q}_{\leq0}(X)$},\\
&\mbox{$N^{(q)}_-\Box^{(q)}_b+\Pi^{(q)}_-=I$ on $\mathcal{B}^{0,q}_{\leq0}(X)\cap{\rm Dom\,}\Box^{(q)}_b$},\\
&\mbox{$N^{(q)}_-:H^s_-(X,T^{*0,q}X)\To H^{s+1}_-(X,T^{*0,q}X)$, $\forall s\in\mathbb Z$},\\
&\mbox{$\Pi^{(q)}_-:H^s_-(X,T^{*0,q}X)\To H^{s+N}_-(X,T^{*0,q}X)$, $\forall s\in\mathbb Z$ and $N\in\mathbb N$}.
\end{split}
\end{equation}
Moreover, $N^{(q)}_-$ and $\Pi^{(q)}_-$ can be continuously extended to $\mathcal{D}'_-(X,T^{*0,q}X)$ and we have
\begin{equation}\label{e-gue140220III}
\begin{split}
&\Pi^{(q)}_-:\mathscr{D}'_-(X,T^{*0,q}X)\To\mathcal{B}^{0,q}_{\leq0}(X)\bigcap\Omega^{0,q}(X),\\
&N^{(q)}_-:\mathscr{D}'_-(X,T^{*0,q}X)\To\mathscr{D}'_-(X,T^{*0,q}X),\\
&\mbox{$\Box^{(q)}_bN^{(q)}_-+\Pi^{(q)}_-=I$ on $\mathscr{D}'_-(X,T^{*0,q}X)$},\\
&\mbox{$N^{(q)}_-\Box^{(q)}_b+\Pi^{(q)}_-=I$ on $\mathscr{D}'_-(X,T^{*0,q}X)$}.
\end{split}
\end{equation}
\end{thm}

\begin{thm}\label{t-gue140218III}
With the notations and assumptions above, assume that $Z(n-1-q)$ holds at every point of $X$. Then,
\[\Box^{(q)}_b:{\rm Dom\,}\Box^{(q)}_b\cap\mathcal{B}^{0,q}_{\geq0}(X)\To\mathcal{B}^{0,q}_{\geq0}(X)\]
has closed range. Let
\[N^{(q)}_+:\mathcal{B}^{0,q}_{\geq0}(X)\To{\rm Dom\,}\Box^{(q)}_b\cap\mathcal{B}^{0,q}_{\geq0}(X)\]
be the associated partial inverse and let
\[\Pi^{(q)}_+:\mathcal{B}^{0,q}_{\geq0}(X)\To{\rm Ker\,}\Box^{(q)}_b\]
be the orthogonal projection. Then, we have
\begin{equation}\label{e-gue140218aVI}
\begin{split}
&\mbox{$\Box^{(q)}_bN^{(q)}_++\Pi^{(q)}_+=I$ on $\mathcal{B}^{0,q}_{\geq0}(X)$},\\
&\mbox{$N^{(q)}_+\Box^{(q)}_b+\Pi^{(q)}_+=I$ on $\mathcal{B}^{0,q}_{\geq0}(X)\cap{\rm Dom\,}\Box^{(q)}_b$},\\
&\mbox{$N^{(q)}_+:H^s_+(X,T^{*0,q}X)\To H^{s+1}_+(X,T^{*0,q}X)$, $\forall s\in\mathbb Z$},\\
&\mbox{$\Pi^{(q)}_+:H^s_+(X,T^{*0,q}X)\To H^{s+N}_+(X,T^{*0,q}X)$, $\forall s\in\mathbb Z$ and $N\in\mathbb N$}.
\end{split}
\end{equation}
Moreover, $N^{(q)}_+$ and $\Pi^{(q)}_+$ can be continuously extended to $\mathcal{D}'_+(X,T^{*0,q}X)$ and we have
\begin{equation}\label{e-gue140220IV}
\begin{split}
&\Pi^{(q)}_+:\mathscr{D}'_+(X,T^{*0,q}X)\To\mathcal{B}^{0,q}_{\geq0}(X)\bigcap\Omega^{0,q}(X),\\
&N^{(q)}_+:\mathscr{D}'_+(X,T^{*0,q}X)\To\mathscr{D}'_+(X,T^{*0,q}X),\\
&\mbox{$\Box^{(q)}_bN^{(q)}_++\Pi^{(q)}_+=I$ on $\mathscr{D}'_+(X,T^{*0,q}X)$},\\
&\mbox{$N^{(q)}_+\Box^{(q)}_b+\Pi^{(q)}_+=I$ on $\mathscr{D}'_+(X,T^{*0,q}X)$}.
\end{split}
\end{equation}
\end{thm}

Our next goal is to prove that if $Z(q)$ fails but $Z(q-1)$ and $Z(q+1)$ hold at every point of $X$, then,
\[\Box^{(q)}_b:{\rm Dom\,}\Box^{(q)}_b\cap\mathcal{B}^{0,q}_{\leq0}(X)\To\mathcal{B}^{0,q}_{\leq0}(X)\]
has closed range. Until further notice, we assume that $Z(q)$ fails but $Z(q-1)$ and $Z(q+1)$ hold 
at every point of $X$. Let $N^{(q-1)}_-$ and $N^{(q+1)}_-$ be as in Theorem~\ref{t-gue140218II}. 
We first need the following.
\begin{lem}\label{l-gue140220I}
Let $u\in\mathcal{B}^{0,q}_{\leq0}(X)$. We have
\begin{equation}\label{e-gue140220VI}
\ol{\pr}^*_b\ddbar_bN^{(q+1)}_-\ddbar_bu=0
\end{equation}
and
\begin{equation}\label{e-gue140220VII}
\ddbar_b\ol{\pr}^*_bN^{(q-1)}_-\ol{\pr}^*_bu=0.
\end{equation}
\end{lem}

\begin{proof}
Let $u\in\mathcal{B}^{0,q}_{\leq0}(X)$. Take $u_j\in\mathcal{B}^{0,q}_{\leq0}(X)\cap\Omega^{0,q}(X)$, 
$j=1,2,\ldots$, so that
$u_j\To u$ in $L^2_{(0,q)}(X)$ as $j\To\infty$. Then, 
$\ol{\pr}^*_b\ddbar_bN^{(q+1)}_-\ddbar_bu_j\To\ol{\pr}^*_b\ddbar_bN^{(q+1)}_-\ddbar_bu$
in $\mathscr D'_-(X,T^{*0,q}X)$ as $j\To\infty$. Fix $j=1,2,\ldots$\,.
From \eqref{e-gue140218aV}, we have
\[\begin{split}
\ol{\pr}^*_b\ddbar_bN^{(q+1)}_-\ddbar_bu_j&=N^{(q+1)}_-\Box^{(q+1)}_b\ol{\pr}^*_b\ddbar_bN^{(q+1)}_-\ddbar_bu_j\\
&=N^{(q+1)}_-\ol{\pr}^*_b\ddbar_b\Box^{(q+1)}_bN^{(q+1)}_-\ddbar_bu_j\\
&=N^{(q+1)}_-\ol{\pr}^*_b\ddbar_b(I-\Pi^{(q+1)}_-)\ddbar_bu_j\\
&=N^{(q+1)}_-\ol{\pr}^*_b\ddbar^2_bu_j=0.
\end{split}\]
Hence $\ol{\pr}^*_b\ddbar_bN^{(q+1)}_-\ddbar_bu=0$.
\eqref{e-gue140220VI} follows. The proof of \eqref{e-gue140220VII} is essentially the same.
\end{proof}
\begin{lem}\label{l-gue140220}
The following operators are continuous:
\begin{equation}\label{e-gue140220V}
\begin{split}
&\ddbar_bN^{(q-1)}_-\ol{\pr}^*_b:\mathcal{B}^{0,q}_{\leq0}(X)\To\mathcal{B}^{0,q}_{\leq0}(X),\\
&\mbox{$\ol{\pr}^*_bN^{(q+1)}_-\ddbar_b:\mathcal{B}^{0,q}_{\leq0}(X)\To\mathcal{B}^{0,q}_{\leq0}(X)$}.
\end{split}
\end{equation}
Moreover, for every $u\in\mathcal{B}^{0,q}_{\leq0}(X)$,
\begin{equation}\label{e-gue140220VIII}
u-\bigr(\ddbar_bN^{(q-1)}_-\ol{\pr}^*_b+\ol{\pr}^*_bN^{(q+1)}_-\ddbar_b)u\in{\rm ker\,}
\Box^{(q)}_b\cap\mathcal{B}^{0,q}_{\leq0}(X).
\end{equation}
\end{lem}

\begin{proof}
Let $u\in\mathcal{B}^{(0,q)}_{\leq0}\cap\Omega^{0,q}(X)$. We have
\[
\begin{split}
&\norm{\ol{\pr}^*_bN^{(q+1)}_-\ddbar_bu}^2\\
&=(\,\ol{\pr}^*_bN^{(q+1)}_-\ddbar_bu\,|\,\ol{\pr}^*_bN^{(q+1)}_-\ddbar_bu\,)
=(\,\ddbar_b\ol{\pr}^*_bN^{(q+1)}_-\ddbar_bu\,|\,N^{(q+1)}_-\ddbar_bu\,)\\
&=(\,\Box^{(q+1)}_bN^{(q+1)}_-\ddbar_bu\,|\,N^{(q+1)}_-\ddbar_bu\,)\ \ (\mbox{here we used \eqref{e-gue140220VI}})\\
&=(\,\ddbar_bu\,|\,N^{(q+1)}_-\ddbar_bu\,)=(\,u\,|\,\ol{\pr}^*_bN^{(q+1)}_-\ddbar_bu\,)\\
&\leq\norm{u}\norm{\ol{\pr}^*_bN^{(q+1)}_-\ddbar_bu}.
\end{split}\]
Hence, $\norm{\ol{\pr}^*_bN^{(q+1)}_-\ddbar_bu}\leq\norm{u}$, 
$\forall u\in\mathcal{B}^{0,q}_{\leq0}(X)\cap\Omega^{0,q}(X)$. 
Thus, $\ol{\pr}^*_bN^{(q+1)}_-\ddbar_b$ can be continuously extended to 
$\ol{\pr}^*_bN^{(q+1)}_-\ddbar_b:\mathcal{B}^{0,q}_{\leq0}(X)\To\mathcal{B}^{0,q}_{\leq0}(X)$.

\noindent
Similarly, we can repeat the procedure above and conclude that
\[\mbox{$\ddbar_bN^{(q-1)}_-\ol{\pr}^*_b:\mathcal{B}^{0,q}_{\leq0}(X)\To\mathcal{B}^{0,q}_{\leq0}(X)$ is continuous}.\]
\eqref{e-gue140220V} follows.

Let $u\in\mathcal{B}^{0,q}_{\leq0}(X)\cap\Omega^{0,q}(X)$ and 
set $v=u-\bigr(\ddbar_bN^{(q-1)}_-\ol{\pr}^*_b+\ol{\pr}^*_bN^{(q+1)}_-\ddbar_b)u$. We have
\[
\begin{split}
\ddbar_bv&=\ddbar_bu-\ddbar_b\ol{\pr}^*_bN^{(q+1)}_-\ddbar_bu\\
&=\ddbar_bu-\Box^{(q+1)}_bN^{(q+1)}_-\ddbar_bu\ \ (\mbox{here we used \eqref{e-gue140220VI}})\\
&=\ddbar_bu-\ddbar_bu=0.
\end{split}
\]
Similarly, we have $\ol{\pr}^*_bv=0$. Thus,
\[u-\bigr(\ddbar_bN^{(q-1)}_-\ol{\pr}^*_b+\ol{\pr}^*_bN^{(q+1)}_-\ddbar_b)u\in{\rm Ker\,}\Box^{(q)}_b,\]
for every $u\in\mathcal{B}^{0,q}_{\leq0}(X)\cap\Omega^{0,q}(X)$. Since
\[\mbox{$I-\ddbar_bN^{(q-1)}_-\ol{\pr}^*_b-\ol{\pr}^*_bN^{(q+1)}_-\ddbar_b:
\mathcal{B}^{0,q}_{\leq0}(X)\To\mathcal{B}^{0,q}_{\leq0}(X)$}\]
is continuous, \eqref{e-gue140220VIII} follows.
\end{proof}

Let $\ol{\pr}^{*,f}_b:\Omega^{0,q+1}(X)\To\Omega^{0,q}(X)$ be the formal adjonit of 
$\ddbar_b$ with respect to $(\,\cdot\,|\,\cdot\,)$. That is, $(\,\ddbar_bf\,|\,g\,)=(\,f\,|\,\ol{\pr}^{*,f}_bg\,)$, 
for all $f\in\Omega^{0,q}(X)$, $g\in\Omega^{0,q+1}(X)$. We need

\begin{lem}\label{l-gue140220II}
Let $u\in L^2_{(0,q)}(X)$. If $\ol{\pr}^{*,f}_bu\in L^2_{(0,q-1)}(X)$, 
then $u\in{\rm Dom\,}\ol{\pr}^*_b$ and $\ol{\pr}^*_bu=\ol{\pr}^{*,f}_bu$.
\end{lem}

\begin{proof}
Let $g\in{\rm Dom\,}\ddbar_b\subset L^2_{(0,q-1)}(X)$. From Friedrichs' lemma \cite[Corollary D.2]{CS01}, 
we can find $g_j\in\Omega^{0,q-1}(X)$, $j=1,2,\ldots$, such that $g_j\To g$ in $L^2_{(0,q-1)}(X)$ 
as $j\To\infty$ and $\ddbar_bg_j\To\ddbar_bg$ in $L^2_{(0,q)}(X)$ as $j\To\infty$. We have
\[\begin{split}
(\,u\,|\,\ddbar_bg)=\lim_{j\To\infty}(\,u\,|\,\ddbar_bg_j)=\lim_{j\To\infty}(\,\ol{\pr}^{*,f}_bu\,|\,g_j)
=(\,\ol{\pr}^{*,f}_bu\,|\,g).
\end{split}\]
Thus, $u\in{\rm Dom\,}\ol{\pr}^*_b$ and $\ol{\pr}^*_bu=\ol{\pr}^{*,f}_bu$. The lemma follows.
\end{proof}

\begin{lem}\label{l-gue140221}
We have
\begin{equation}\label{e-gue140221}
\ol{\pr}^*_b(N^{(q+1)}_-)^2\ddbar_b:\mathcal{B}^{0,q}_{\leq0}(X)\To
{\rm Dom\,}\Box^{(q)}_b\cap\mathcal{B}^{0,q}_{\leq0}(X)
\end{equation}
and
\begin{equation}\label{e-gue140221I}
\ddbar_b(N^{(q-1)}_-)^2\ol{\pr}^*_b:\mathcal{B}^{0,q}_{\leq0}(X)\To
{\rm Dom\,}\Box^{(q)}_b\cap\mathcal{B}^{0,q}_{\leq0}(X).
\end{equation}
\end{lem}

\begin{proof}
In view of \eqref{e-gue140218aV}, we see that
\[\mbox{$\ol{\pr}^*_b(N^{(q+1)}_-)^2\ddbar_b:\mathcal{B}^{0,q}_{\leq0}(X)\To
\mathcal{B}^{0,q}_{\leq0}(X)$ is continuous}.\]
Let $u\in\mathcal{B}^{0,q}_{\leq0}(X)\cap\Omega^{0,q}(X)$. We have
\begin{equation}\label{e-gue140221II}
\begin{split}
\norm{\ddbar_b\ol{\pr}^*_b(N^{(q+1)}_-)^2\ddbar_bu}^2&=
(\,\ddbar_b\ol{\pr}^*_b(N^{(q+1)}_-)^2\ddbar_bu\,|\,\ddbar_b\ol{\pr}^*_b(N^{(q+1)}_-)^2\ddbar_bu\,)\\
&=(\,\ol{\pr}^*_b\ddbar_b\ol{\pr}^*_b(N^{(q+1)}_-)^2\ddbar_bu\,|\,\ol{\pr}^*_b(N^{(q+1)}_-)^2\ddbar_bu\,)\\
&=(\,\ol{\pr}^*_b\Box^{(q+1)}_b(N^{(q+1)}_-)^2\ddbar_bu\,|\,\ol{\pr}^*_b(N^{(q+1)}_-)^2\ddbar_bu\,)\\
&=(\,\ol{\pr}^*_bN^{(q+1)}_-\ddbar_bu\,|\,\ol{\pr}^*_b(N^{(q+1)}_-)^2\ddbar_bu\,)\\
&\leq\norm{\ol{\pr}^*_bN^{(q+1)}_-\ddbar_bu}\norm{\ol{\pr}^*_b(N^{(q+1)}_-)^2\ddbar_bu}.
\end{split}
\end{equation}
From \eqref{e-gue140220V} and \eqref{e-gue140221II}, we see that there is a constant $C>0$ such that
\[\norm{\ddbar_b\ol{\pr}^*_b(N^{(q+1)}_-)^2\ddbar_bu}\leq C\norm{u},\ \ 
\forall u\in\mathcal{B}^{0,q}_{\leq0}(X)\cap\Omega^{0,q}(X).\]
Thus, $\ddbar_b\ol{\pr}^*_b(N^{(q+1)}_-)^2\ddbar_b$ can be extended continuously to 
$\mathcal{B}^{0,q}_{\leq0}(X)$ and we have
\[\mbox{$\ddbar_b\ol{\pr}^*_b(N^{(q+1)}_-)^2\ddbar_b:\mathcal{B}^{0,q}_{\leq0}(X)\To
\mathcal{B}^{0,q+1}_{\leq0}(X)$ is continuous}.\]
Hence,
\begin{equation}\label{e-gue140221III}
\ol{\pr}^*_b(N^{(q+1)}_-)^2\ddbar_b:\mathcal{B}^{0,q}_{\leq0}(X)\To{\rm Dom\,}
\ddbar_b\cap\mathcal{B}^{0,q}_{\leq0}(X).
\end{equation}
Let $u\in\mathcal{B}^{0,q}_{\leq0}(X)\cap\Omega^{0,q}(X)$. We have
\begin{equation}\label{e-gue140222}
\begin{split}
\ol{\pr}^{*,f}_b\ddbar_b\ol{\pr}^*_b(N^{(q+1)}_-)^2\ddbar_bu&=
\ol{\pr}^{*}_b\ddbar_b\ol{\pr}^*_b(N^{(q+1)}_-)^2\ddbar_bu\\
&=\ol{\pr}^{*}_b\Box^{(q+1)}_b(N^{(q+1)}_-)^2\ddbar_bu\\
&=\ol{\pr}^{*}_bN^{(q+1)}_-\ddbar_bu.
\end{split}
\end{equation}
From \eqref{e-gue140220V} and \eqref{e-gue140222}, we see that there is a constant $C_1>0$ such that
\[\norm{\ol{\pr}^{*,f}_b\ddbar_b\ol{\pr}^*_b(N^{(q+1)}_-)^2\ddbar_bu}\leq 
C_1\norm{u},\ \ \forall u\in\mathcal{B}^{0,q}_{\leq0}(X)\cap\Omega^{0,q}(X).\]
Thus, $\ol{\pr}^{*,f}_b\ddbar_b\ol{\pr}^*_b(N^{(q+1)}_-)^2\ddbar_b$ 
can be extended continuously to $\mathcal{B}^{0,q}_{\leq0}(X)$ and we have
\begin{equation}\label{e-gue140222I}
\mbox{$\ol{\pr}^{*,f}_b\ddbar_b\ol{\pr}^*_b(N^{(q+1)}_-)^2\ddbar_b:
\mathcal{B}^{0,q}_{\leq0}(X)\To\mathcal{B}^{0,q}_{\leq0}(X)$ is continuous}.
\end{equation}
From \eqref{e-gue140222I} and Lemma~\ref{l-gue140220II}, we conclude that
\begin{equation}\label{e-gue140222II}
\ddbar_b\ol{\pr}^*_b(N^{(q+1)}_-)^2\ddbar_b:\mathcal{B}^{0,q}_{\leq0}(X)
\To{\rm Dom\,}\ol{\pr}^*_b\cap\mathcal{B}^{0,q+1}_{\leq0}(X).
\end{equation}

Moreover, it is easy to see that for $u\in\mathcal{B}^{0,q}_{\leq0}(X)$, 
$\ol{\pr}^*_b(N^{(q+1)}_-)^2\ddbar_bu\in{\rm Dom\,}\ol{\pr}^*_b$ and 
\[(\ol{\pr}^*_b)^2(N^{(q+1)}_-)^2\ddbar_bu=0.\]
From this observation, \eqref{e-gue140222II} and \eqref{e-gue140221III}, \eqref{e-gue140221} follows.

The proof of \eqref{e-gue140221I} is essentially the same.
\end{proof}


\begin{thm}\label{t-gue140220}
With the notations above, assume that $Z(q)$ fails but $Z(q-1)$ and $Z(q+1)$ hold at every point of $X$. Then,
\[\Box^{(q)}_b:{\rm Dom\,}\Box^{(q)}_b\cap\mathcal{B}^{0,q}_{\leq0}(X)\To\mathcal{B}^{0,q}_{\leq0}(X)\]
has closed range.
\end{thm}
\begin{proof}
Let $\Box^{(q)}_bu_j=v_j$,
$u_j\in{\rm Dom\,}\Box^{(q)}_b\cap\mathcal{B}^{0,q}_{\leq0}(X)$, 
$v_j\in\mathcal{B}^{0,q}_{\leq0}(X)$, $j=1,2,\ldots$, with $v_j\To v\in\mathcal{B}^{0,q}_{\leq0}(X)$ 
as $j\To\infty$. We are going to prove that there is a 
$g\in{\rm Dom\,}\Box^{(q)}_b\cap\mathcal{B}^{0,q}_{\leq0}(X)$ 
such that $\Box^{(q)}_bg=v$. Let $N^{(q-1)}_-$ and $N^{(q+1)}_-$ be as in 
Theorem~\ref{t-gue140218II}. Put
\[g_j=\bigr(\ol{\pr}^*_bN^{(q+1)}_-\ddbar_b+\ddbar_bN^{(q-1)}_-\ol{\pr}^*_b\bigr)u_j,\ \ j=1,2,\ldots.\]
In view of \eqref{e-gue140220V}, we see that $g_j\in\mathcal{B}^{0,q}_{\leq0}(X)$. 
Moreover, from \eqref{e-gue140220VIII}, we have 
\[u_j-\bigr(\ol{\pr}^*_bN^{(q+1)}_-\ddbar_b+\ddbar_bN^{(q-1)}_-\ol{\pr}^*_b\bigr)u_j
\in{\rm Ker\,}\Box^{(q)}_b\subset{\rm Dom\,}\Box^{(q)}_b\,,\ \ j=1,2,\ldots\,.\] 
Hence,
\begin{equation}\label{e-geu140222III}
\begin{split}
&g_j\in{\rm Dom\,}\Box^{(q)}_b\cap\mathcal{B}^{0,q}_{\leq0}(X),\ \ j=1,2,\ldots,\\
&\Box^{(q)}_bg_j=\Box^{(q)}_bu_j=v_j,\ \ j=1,2,\ldots.
\end{split}
\end{equation}
We claim that for each $j$,
\begin{equation}\label{e-gue140220}
\begin{split}
&g_j=\bigr(\ol{\pr}^*_bN^{(q+1)}_-\ddbar_b+\ddbar_bN^{(q-1)}_-\ol{\pr}^*_b\bigr)u_j\\
&=\bigr(\ol{\pr}^*_b(N^{(q+1)}_-)^2\ddbar_b+\ddbar_b(N^{(q-1)}_-)^2\ol{\pr}^*_b\bigr)\Box^{(q)}_bu_j\\
&=\bigr(\ol{\pr}^*_b(N^{(q+1)}_-)^2\ddbar_b+\ddbar_b(N^{(q-1)}_-)^2\ol{\pr}^*_b\bigr)v_j.
\end{split}
\end{equation}
Fix $j=1,2,\ldots$. Let $f_s\in\mathcal{B}^{0,q}_{\leq0}(X)\cap\Omega^{0,q}(X)$, $s=1,2,\ldots$, 
with $f_s\To u_j$ in $\mathcal{B}^{0,q}_{\leq0}(X)$ as $s\To\infty$. We have
\begin{equation}\label{e-gue140220I}
\begin{split}
&\bigr(\ol{\pr}^*_b(N^{(q+1)}_-)^2\ddbar_b+\ddbar_b(N^{(q-1)}_-)^2\ol{\pr}^*_b\bigr)\Box^{(q)}_bf_s\\
&\mbox{$\To\bigr(\ol{\pr}^*_b(N^{(q+1)}_-)^2\ddbar_b+\ddbar_b(N^{(q-1)}_-)^2\ol{\pr}^*_b\bigr)\Box^{(q)}_bu_j$ 
in $\mathscr D'_-(X,T^{*0,q}X)$ as $s\To\infty$}.
\end{split}
\end{equation}
We can check that
\begin{equation}\label{e-gue140220II}
\begin{split}
&\bigr(\ol{\pr}^*_b(N^{(q+1)}_-)^2\ddbar_b+\ddbar_b(N^{(q-1)}_-)^2\ol{\pr}^*_b\bigr)\Box^{(q)}_bf_s\\
&=\ol{\pr}^*_b(N^{(q+1)}_-)^2\Box^{(q+1)}_b\ddbar_bf_s+\ddbar_b(N^{(q-1)}_-)^2\Box^{(q-1)}_b\ol{\pr}^*_bf_s\\
&=\ol{\pr}^*_bN^{(q+1)}_-(I-\Pi^{(q+1)}_-)\ddbar_bf_s+\ddbar_bN^{(q-1)}_-(I-\Pi^{(q-1)}_-)\ol{\pr}^*_bf_s\\
&=\ol{\pr}^*_bN^{(q+1)}_-\ddbar_bf_s+\ddbar_bN^{(q-1)}_-\ol{\pr}^*_bf_s\\
&\mbox{$\To\bigr(\ol{\pr}^*_bN^{(q+1)}_-\ddbar_b+\ddbar_bN^{(q-1)}_-\ol{\pr}^*_b\bigr)u_j=g_j$ 
in $\mathscr D'_-(X,T^{*0,q}X)$ as $s\To\infty$}.
\end{split}
\end{equation}
From \eqref{e-gue140220I} and \eqref{e-gue140220II}, \eqref{e-gue140220} follows.
Since $v_j\To v\in\mathcal{B}^{0,q}_{\leq0}(X)$ and
\[\mbox{$\ol{\pr}^*_b(N^{(q+1)}_-)^2\ddbar_b+\ddbar_b(N^{(q-1)}_-)^2\ol{\pr}^*_b:
\mathcal{B}^{0,q}_{\leq0}(X)\To\mathcal{B}^{0,q}_{\leq0}(X)$ is continuous}\]
(see \eqref{e-gue140218aV}), we conclude that
\[g_j\To g:=\bigr(\ol{\pr}^*_b(N^{(q+1)}_-)^2\ddbar_b+
\ddbar_b(N^{(q-1)}_-)^2\ol{\pr}^*_b)v\in\mathcal{B}^{0,q}_{\leq0}(X)\]
and $\Box^{(q)}_bg=v$ in $\mathscr D'_-(X,T^{*0,q}X)$. 
In view of Lemma~\ref{l-gue140221}, we see that $g\in{\rm Dom\,}\Box^{(q)}_b\cap\mathcal{B}^{0,q}_{\leq0}(X)$. 
The theorem follows.
\end{proof}
We can repeat the proof of Theorem~\ref{t-gue140220} and deduce the following.
\begin{thm}\label{t-gue140222}
With the notations above, assume that $Z(n-1-q)$ fails but $Z(n-2-q)$ and $Z(n-q)$ hold at every point of $X$. Then,
\[\Box^{(q)}_b:{\rm Dom\,}\Box^{(q)}_b\cap\mathcal{B}^{0,q}_{\geq0}(X)\To\mathcal{B}^{0,q}_{\geq0}(X)\]
has closed range.
\end{thm}
Now, we can prove:
\begin{thm}\label{t-gue140222I}
With the notations above, assume that $Z(q)$ fails but $Z(q-1)$ and $Z(q+1)$ hold at every point of $X$. 
Then, $\Box^{(q)}_b$ has local $L^2$ closed range on $X$ with respect to $Q^{(q)}_{\leq0}$ 
in the sense of Definition~\ref{d-gue140212}.
\end{thm}

\begin{proof}
From Theorem~\ref{t-gue140220}, we see that there is a constant $C>0$ such that
\begin{equation}\label{e-gue140222III}
\norm{(I-\Pi^{(q)}_-)u}\leq C\norm{\Box^{(q)}_bu},\ \ \forall u\in\mathcal{B}^{0,q}_{\leq0}\cap{\rm Dom\,}\Box^{(q)}_b.
\end{equation}
Let $f\in\Omega^{0,q}(X)$. Then, $Q^{(q)}_{\leq0}f\in\mathcal{B}^{0,q}_{\leq0}\cap\Omega^{0,q}(X)$. We claim that
\begin{equation}\label{e-gue140222IV}
Q^{(q)}_{\leq0}\Pi^{(q)}f=\Pi^{(q)}_-Q^{(q)}_{\leq0}f.
\end{equation}
Note that $Q^{(q)}_{\leq0}\Pi^{(q)}f=\Pi^{(q)}Q^{(q)}_{\leq0}f$. Thus,
\[(\,Q^{(q)}_{\leq0}\Pi^{(q)}f\,|\,Q^{(q)}_{\leq0}(I-\Pi^{(q)})f\,)=
(\,\Pi^{(q)}Q^{(q)}_{\leq0}f\,|\,(I-\Pi^{(q)})Q^{(q)}_{\leq0}f\,)=0.\]
We have the orthogonal decompositions
\begin{equation}\label{e-gue140222V}
\begin{split}
&Q^{(q)}_{\leq0}f=Q^{(q)}_{\leq0}\Pi^{(q)}f+Q^{(q)}_{\leq0}(I-\Pi^{(q)})f,\\
&Q^{(q)}_{\leq0}f=\Pi^{(q)}_-Q^{(q)}_{\leq0}f+(I-\Pi^{(q)}_-)Q^{(q)}_{\leq0}f.\\
\end{split}
\end{equation}
Hence,
\begin{equation}\label{e-gue140222VI}
Q^{(q)}_{\leq0}\Pi^{(q)}f-\Pi^{(q)}_-Q^{(q)}_{\leq0}f=(I-\Pi^{(q)}_-)Q^{(q)}_{\leq0}f-Q^{(q)}_{\leq0}(I-\Pi^{(q)})f.
\end{equation}
From \eqref{e-gue140222VI}, we have
\begin{equation}\label{e-gue140222VII}
\begin{split}
&(\,Q^{(q)}_{\leq0}\Pi^{(q)}f-\Pi^{(q)}_-Q^{(q)}_{\leq0}f\,|\,Q^{(q)}_{\leq0}\Pi^{(q)}f-\Pi^{(q)}_-Q^{(q)}_{\leq0}f\,)\\
&=(\,Q^{(q)}_{\leq0}\Pi^{(q)}f-\Pi^{(q)}_-Q^{(q)}_{\leq0}f\,|\,(I-\Pi^{(q)}_-)Q^{(q)}_{\leq0}f
-Q^{(q)}_{\leq0}(I-\Pi^{(q)})f\,)\\
&=0
\end{split}
\end{equation}
since $Q^{(q)}_{\leq0}\Pi^{(q)}f-\Pi^{(q)}_-Q^{(q)}_{\leq0}f
\in{\rm Ker\,}\Box^{(q)}_b\cap\mathcal{B}^{0,q}_{\leq0}$. Hence,
\[Q^{(q)}_{\leq0}\Pi^{(q)}f=\Pi^{(q)}_-Q^{(q)}_{\leq0}f.\]
The claim \eqref{e-gue140222IV} follows.
Note that 
$Q^{(q)}_{\leq0}:{\rm Dom\,}\Box^{(q)}_b\cap L^2_{(0,q)}(X)\To
{\rm Dom\,}\Box^{(q)}_b\cap \mathcal{B}^{0,q}_{\leq0}(X)$ and $Q^{(q)}_{\leq0}:
\Omega^{0,q}(X)\To\mathcal{B}^{0,q}_{\leq0}(X)\cap\Omega^{0,q}(X)$\,.
From this observation, \eqref{e-gue140222IV} and \eqref{e-gue140222III}, we obtain
\[\begin{split}\norm{Q^{(q)}_{\leq0}(I-\Pi^{(q)})u}&=\norm{(I-\Pi^{(q)}_-)Q^{(q)}_{\leq0}u}
\leq C\norm{\Box^{(q)}_bQ^{(q)}_{\leq0}u}=C\norm{Q^{(q)}_{\leq0}\Box^{(q)}_bu}\\
&\leq C\norm{\Box^{(q)}_bu},\ \ \forall u\in\Omega^{0,q}(X),\end{split}\]
where $C>0$ is a constant. The theorem follows.
\end{proof}

Similarly, we can repeat the proof of Theorem~\ref{t-gue140222I} and deduce

\begin{thm}\label{t-gue140223}
With the notations above, assume that $Z(n-1-q)$ fails but $Z(n-2-q)$ and $Z(n-q)$ 
hold at every point of $X$. Then, $\Box^{(q)}_b$ has local $L^2$ closed range on $X$ 
with respect to $Q^{(q)}_{\geq0}$ in the sense of Definition~\ref{d-gue140212}.
\end{thm}

Let $D\subset X$ be a canonical coordinate patch and let $x=(x_1,\ldots,x_{2n-1})$ 
be canonical coordinates on $D$ as in Theorem~\ref{t-gue131206}. 
We identify $D$ with $W\times]-\pi,\pi[\subset\Real^{2n-1}$, 
where $W$ is some open set in $\Real^{2n-2}$. 
Until further notice, we work with canonical coordinates 
$x=(x_1,\ldots,x_{2n-1})$. Let $\eta=(\eta_1,\ldots,\eta_{2n-1})$ 
be the dual coordinates of $x$. Let $\alpha(x_{2n-1})\in \cC^\infty(\Real,[0,1])$ 
with $\alpha=1$ on $[\frac{1}{2},\infty[$, $\alpha=0$ on $]-\infty,\frac{1}{4}]$. 
We recall Definition~\ref{d-gue140221}. We need

\begin{lem}\label{l-gue140223}
With the notations above,
\[Q^{(q)}_{\leq0}, Q^{(q)}_{\geq0}\in L^0_{{\rm cl\,}}(X,T^{*0,q}X\boxtimes T^{*0,q}X),\]
\begin{equation}\label{e-gue140223}
\begin{split}
&\mbox{$Q^{(q)}_{\leq0}(x,y)\equiv\frac{1}{(2\pi)^{2n-1}}
\int e^{i\langle x-y,\eta\rangle }\alpha(\eta_{2n-1})d\eta_{2n-1}$ at $\Sigma^-\cap T^*D$},\\
&\mbox{$Q^{(q)}_{\leq0}(x,y)\equiv0$ at $\Sigma^+\cap T^*D$}
\end{split}
\end{equation}
and
\begin{equation}\label{e-gue140223I}
\begin{split}
&Q^{(q)}_{\geq0}(x,y)\equiv\frac{1}{(2\pi)^{2n-1}}
\int e^{i\langle x-y,\eta\rangle }\alpha(-\eta_{2n-1})d\eta_{2n-1}\:\:\text{at $\Sigma^+\cap T^*D$},\\
&\mbox{$Q^{(q)}_{\geq0}(x,y)\equiv0$ at $\Sigma^-\cap T^*D$}.
\end{split}
\end{equation}
\end{lem}

\begin{proof}
It is easy to see that on $D$,
\begin{equation}\label{e-gue140223III}
Q^{(q)}_{\leq0}u(y)=\frac{1}{2\pi}\sum_{m\in\mathbb Z,m\geq0}
e^{imy_{2n-1}}\int^{\pi}_{-\pi}e^{-imt}u(y',t)dt,\ \ \forall u\in\Omega^{0,q}_0(D),
\end{equation}
where $y'=(y_1,\ldots,y_{2n-2})$. Fix $D'\Subset D$ and let $\chi(y_{2n-1})\in \cC^\infty_0(]-\pi,\pi[)$ 
such that $\chi(y_{2n-1})=1$ for every $(y',y_{2n-1})\in D'$.  Let $\beta(x_{2n-1})\in \cC^\infty(\Real,[0,1])$ 
with $\beta=1$ on $[-\frac{1}{4},\infty[$, $\beta=0$ on $]-\infty,-\frac{1}{2}]$. 
Let $R:\Omega^{0,q}_0(D')\To\Omega^{0,q}(D')$ be the continuous operator given by
\begin{equation}\label{e-gue140223IV}
\begin{split}
u\mapsto\frac{1}{(2\pi)^2}\sum_{m\in\mathbb Z}
\int_{\abs{t}\leq\pi}e^{i\langle x_{2n-1}-y_{2n-1},\eta_{2n-1}\rangle }
\beta(\eta_{2n-1})(1-\chi(y_{2n-1}))\\
e^{imy_{2n-1}}e^{-imt}u(x',t)dtd\eta_{2n-1}dy_{2n-1},
\end{split}
\end{equation}
where $x'=(x_1,\ldots,x_{2n-2})$. Moreover, we can integrate by parts with respect to 
$\eta_{2n-1}$ and conclude that
\begin{equation}\label{e-gue140223V}
R\equiv0\ \ \text{at $\Sigma^-\cap T^*D'$},\quad R\equiv0\ \ \mbox{at $\Sigma^+\cap T^*D'$}.
\end{equation}
Now, we claim that
\begin{equation}\label{e-gue140223II}
Q^{(q)}_{\leq0}(x,y)=R(x,y)+\frac{1}{(2\pi)^{2n-1}}\int e^{i\langle x-y,\eta\rangle }
\beta(\eta_{2n-1})d\eta_{2n-1}\:\: \text{on $D'$}.
\end{equation}
Let $u\in\Omega^{0,q}_0(D')$. From Fourier inversion formula, it is straightforward to see that
\begin{equation}\label{e-gue140223VI}
\begin{split}
&\frac{1}{(2\pi)^{2n-1}}\int e^{i\langle x-y,\eta\rangle}\beta(\eta_{2n-1})u(y)d\eta_{2n-1}\\
&=\frac{1}{(2\pi)^2}\sum_{m\in\mathbb Z}
\int_{\abs{t}\leq\pi}e^{i\langle x_{2n-1}-y_{2n-1},\eta_{2n-1}\rangle }\beta(\eta_{2n-1})\\
&\quad\times\chi(y_{2n-1})e^{imy_{2n-1}}e^{-imt}u(x',t)dtd\eta_{2n-1}dy_{2n-1}.
\end{split}
\end{equation}
From \eqref{e-gue140223IV} and \eqref{e-gue140223VI}, we have
\begin{equation}\label{e-gue140223VII}
\begin{split}
&\frac{1}{(2\pi)^{2n-1}}\int e^{i\langle x-y,\eta\rangle }\beta(\eta_{2n-1})u(y)d\eta_{2n-1}+Ru(x)\\
&=\frac{1}{(2\pi)^2}\sum_{m\in\mathbb Z}
\int_{\abs{t}\leq\pi}e^{i\langle x_{2n-1}-y_{2n-1},\eta_{2n-1}\rangle }\beta(\eta_{2n-1})\\
&\quad\quad\times e^{imy_{2n-1}}e^{-imt}u(x',t)dtd\eta_{2n-1}dy_{2n-1}.
\end{split}\end{equation}
From Fourier inversion formula and notice that for every $m\in\mathbb Z$,
\[\int e^{imy_{2n-1}}e^{-iy_{2n-1}\eta_{2n-1}}dy_{2n-1}=2\pi\delta_m(\eta_{2n-1}),\]
where the integral above is defined as an oscillatory integral and $\delta_m$ is the Dirac measure 
at $m$ (see Chapter 7.2 in H\"ormander~\cite{Hor03}), \eqref{e-gue140223VII} becomes
\begin{equation}\label{e-gue140223VIII}
\begin{split}
&\frac{1}{(2\pi)^{2n-1}}\int e^{i\langle x-y,\eta\rangle }\beta(\eta_{2n-1})u(y)d\eta_{2n-1}+Ru(x)\\
&=\frac{1}{2\pi}\sum_{m\in\mathbb Z}\beta(m)e^{ix_{2n-1}m}
\int_{\abs{t}\leq\pi}e^{-imt}u(x',t)dt\\
&=\frac{1}{2\pi}\sum_{m\in\mathbb Z,m\geq0}e^{ix_{2n-1}m}
\int_{\abs{t}\leq\pi}e^{-imt}u(x',t)dt\\
&=Q^{(q)}_{\leq0}u(x).
\end{split}\end{equation}
Here we used \eqref{e-gue140223III}. The claim \eqref{e-gue140223II} follows.
From \eqref{e-gue140223II} and \eqref{e-gue140223V}, we conclude that
\[\begin{split}
&\mbox{$Q^{(q)}_{\leq0}(x,y)-\frac{1}{(2\pi)^{2n-1}}\int e^{i\langle x-y,\eta\rangle }
\beta(\eta_{2n-1})d\eta_{2n-1}\equiv0$ at $\Sigma^-\cap T^*D'$},\\
&\mbox{$Q^{(q)}_{\leq0}(x,y)\equiv0$ at $\Sigma^+\cap T^*D'$}.\end{split}\]
Moreover, it is straightforward to see that
\[
\frac{1}{(2\pi)^{2n-1}}\int e^{i\langle x-y,\eta\rangle}\big(\alpha(\eta_{2n-1})-\beta(\eta_{2n-1})\big)d\eta_{2n-1}
\equiv0\:\: \text{at $\Sigma^-\cap T^*D$}.\]
From this observation, \eqref{e-gue140223} follows.
The proof of \eqref{e-gue140223I} is essentially the same as the proof of \eqref{e-gue140223}.
\end{proof}

From Theorem~\ref{t-gue140222I}, Theorem~\ref{t-gue140223}, Lemma~\ref{l-gue140223} 
and Theorem~\ref{t-gue140305VIab}, we get the following two results.

\begin{thm}\label{t-gue140223I}
Let $(X,T^{1,0}X)$ be a compact CR manifold of dimension $2n-1$, $n\geq2$, 
with a transversal CR $S^1$ action and let $T\in \cC^\infty(X,TX)$ be the real 
vector field induced by this $S^1$ action. We fix a $T$-rigid Hermitian metric 
$\langle\,\cdot\,|\,\cdot\,\rangle$ on $\Complex TX$ such that 
$T^{1,0}X\perp T^{0,1}X$, $T\perp (T^{1,0}X\oplus T^{0,1}X)$, 
$\langle\,T\,|\,T\,\rangle=1$ and $\langle\,u\,|v\,\rangle$ is real if $u, v$ 
are real tangent vectors and we take $m(x)$ to be the volume form 
induced by the given $T$-rigid Hermitian metric $\langle\,\cdot\,|\,\cdot\,\rangle$. 
Assume that $Z(q)$ fails but $Z(q-1)$ and $Z(q+1)$ hold at every point of $X$. 
Suppose that the Levi form is non-degenerate of constant signature $(n_-,n_+)$ 
on an open canonical coordinate patch $D\Subset X$.  Let $Q^{(q)}_{\leq0}:L^2_{(0,q)}(X)\To L^2_{(0,q)}(X)$
be as in \eqref{e-gue140218VII}. Then,
\begin{equation}\label{e-gue140223f}
\mbox{$Q^{(q)}_{\leq0}\Pi^{(q)}Q^{(q)}_{\leq0}\equiv0$ on $D$ if $q\neq n_-$}
\end{equation}
and
\begin{equation}\label{e-gue140223fI}
Q^{(q)}_{\leq0}\Pi^{(q)}Q^{(q)}_{\leq0}(x,y)\equiv\int^\infty_0e^{i\varphi_-(x,y)t}a(x,y,t)dt\,\:\: \text{on $D$ if $q=n_-$}\,,
\end{equation}
where $\varphi_-\in \cC^\infty(D\times D)$ is as in Theorem~\ref{t-gue140305_b} 
and $a(x,y,t)\in S^{n-1}_{{\rm cl\,}}\big(D\times D\times\mathbb{R}_+,T^{*0,q}_yX\boxtimes T^{*0,q}_xX\big)$
where with notations as in \eqref{det140530}, \eqref{tau140530}, the leading term $a_0(x,y)$ 
of the expansion \eqref{e-fal} of $a(x,y,t)$ satisfies
\[\begin{split}
&a_0(x,x)=\frac{1}{2}\pi^{-n}\abs{{\rm det\,}\mathcal{L}_x}\tau_{x,n_-},\ \ \forall x\in D.
\end{split}\]
\end{thm}
Similarly, we obtain the following.
\begin{thm}\label{t-gue140223II}
Under the hypotheses of Theorem \ref{t-gue140223I}
assume that $Z(n-1-q)$ fails but $Z(n-2-q)$ and $Z(n-q)$ hold at every point of $X$. Suppose that the Levi form
is non-degenerate of constant signature $(n_-,n_+)$ on an open canonical coordinate patch $D\Subset X$.  Let
\[Q^{(q)}_{\geq0}:L^2_{(0,q)}(X)\To L^2_{(0,q)}(X)\]
be as in \eqref{e-gue140218VIII}. Then,
\begin{equation}\label{e-gue140223fII}
\mbox{$Q^{(q)}_{\geq0}\Pi^{(q)}Q^{(q)}_{\geq0}\equiv0$ on $D$ if $q\neq n_+$}
\end{equation}
and
\begin{equation}\label{e-gue140223fIII}
Q^{(q)}_{\geq0}\Pi^{(q)}Q^{(q)}_{\geq0}(x,y)\equiv\int^\infty_0e^{i\varphi_+(x,y)t}b(x,y,t)dt\,\:\: 
\text{on $D$ if $q=n_+$}\,,
\end{equation}
where $\varphi_+\in \cC^\infty(D\times D)$ is as in Theorem~\ref{t-gue140305_b} 
and $b(x, y, t)\in S^{n-1}_{{\rm cl\,}}\big(D\times D\times\mathbb{R}_+,T^{*0,q}_yX\boxtimes T^{*0,q}_xX\big)$, 
where with notations as in \eqref{det140530}, \eqref{tau140530}, the leading term $b_0(x,y)$ 
of the expansion \eqref{e-fal} of $b_0(x,y,t)$ satisfies
\[\begin{split}
&b_0(x,x)=\frac{1}{2}\pi^{-n}\abs{{\rm det\,}\mathcal{L}_x}\tau_{x,n_+},\ \ \forall x\in D.
\end{split}\]
\end{thm}

Kohn proved that if $X$ is any compact CR manifold and $Y(q)$ fails but $Y(q-1)$ and $Y(q+1)$ hold on $X$ then $\Box^{(q)}_b$
has $L^2$ closed range (see~\cite{CS01}). By using Theorem~\ref{t-gue140218II}, Theorem~\ref{t-gue140218III},
Theorem~\ref{t-gue140220} and Theorem~\ref{t-gue140222},
we can improve Kohn's result if $X$ admits a transversal CR $S^1$ action.
\begin{defn}\label{d-gue140223}
Given $q\in\{0,\ldots,n-1\}$, the Levi form is said to satisfy condition $W(q)$ at $p\in X$,
if one of the following condition holds: (I)\,$Y(q)$ holds at $p$.
(II)\, $Z(q)$,  $Z(n-2-q)$ and $Z(n-q)$ hold at $p$.
(III)\,$Z(q-1)$, $Z(q+1)$ and $Z(n-1-q)$ hold at $p$. (IV)\, $Y(q-1)$ and $Y(q+1)$ hold.
\end{defn}
It is straightforward to see that if the Levi form is non-degenerate of constant signature on $X$ then for every
$q\in\set{0,1,\ldots,n-1}$, $W(q)$ holds at every point of $X$.
It is clear that if $Y(q-1)$ and $Y(q+1)$ hold at $p\in X$, or $Y(q)$ holds at $p$, then $W(q)$ holds at $p$.
But it can happen that $W(q)$ holds at $p$ but $Y(q)$ fails at $p$ and $Y(q-1)$ or $Y(q+1)$ fail at $p$.
For example, if the Levi form is non-degenerate of constant signature $(n_-,n_+)$ at $p$
and $n_+=n_-+1$, then for $q=n_-$, $Z(q-1)$, $Z(q+1)$ and $Z(n-1-q)$ hold at $p$.
Thus, $W(q)$ holds at $p$ but $Y(q)$ and $Y(q+1)$ fail at $p$.
\begin{thm}\label{t-gue140223III}
Let $(X,T^{1,0}X)$ be a compact CR manifold of dimension $2n-1$, $n\geq2$,
with a transversal CR $S^1$ action and let $T\in \cC^\infty(X,TX)$ be the real vector field
induced by this $S^1$ action. We fix a $T$-rigid Hermitian metric
$\langle\,\cdot\,|\,\cdot\,\rangle$ on $\Complex TX$ such that $T^{1,0}X\perp T^{0,1}X$,
$T\perp (T^{1,0}X\oplus T^{0,1}X)$, $\langle\,T\,|\,T\,\rangle=1$ and $\langle\,u\,|v\,\rangle$
is real if $u, v$ are real tangent vectors and we take $m(x)$ to be the volume form induced by the given
$T$-rigid Hermitian metric $\langle\,\cdot\,|\,\cdot\,\rangle$. Assume that $W(q)$ holds at every point of $X$.
Then, $\Box^{(q)}_b:{\rm Dom\,}\Box^{(q)}_b\To L^2_{(0,q)}(X)$ has $L^2$ closed range.
In particular, if the Levi form is non-degenerate of constant signature on $X$,
then $\Box^{(q)}_b:{\rm Dom\,}\Box^{(q)}_b\To L^2_{(0,q)}(X)$ has $L^2$ closed range.
\end{thm}
\begin{proof}
Since $W(q)$ holds at every point of $X$, from
Theorem~\ref{t-gue140218II}, Theorem~\ref{t-gue140218III}, Theorem~\ref{t-gue140220}
and Theorem~\ref{t-gue140222}, we see that the operators
\[\begin{split}
\Box^{(q)}_b:{\rm Dom\,}\Box^{(q)}_b\cap\mathcal{B}^{0,q}_{\leq0}\To\mathcal{B}^{0,q}_{\leq0}\,,
\quad\Box^{(q)}_b:{\rm Dom\,}\Box^{(q)}_b\cap\mathcal{B}^{0,q}_{\geq0}\To\mathcal{B}^{0,q}_{\geq0}\end{split}\]
have closed range. It is not difficult to see that this implies that
$\Box^{(q)}_b:{\rm Dom\,}\Box^{(q)}_b\To L^2_{(0,q)}(X)$ has $L^2$ closed range. We leave the details to the reader.
\end{proof}
\begin{cor}\label{c-gue140306}
Under the same notations and assumptions used in Theorem~\ref{t-gue140223III}
, let
$N^{(q)}:L^2_{(0,q)}(X)\To{\rm Dom\,}\Box^{(q)}_b$ be the partial inverse of $\Box^{(q)}_b$.
We assume that the Levi form is non-degenerate of constant signature $(n_-,n_+)$
at each point of an open set $D\Subset X$.
If $q\notin\set{n_-,n_+}$, then
\[
\Pi^{(q)}\equiv 0\quad\text{and}\quad N^{(q)}\equiv A\quad\text{on $D$},
\]
where $A\in L^{-1}_{\frac{1}{2},\frac{1}{2}}(D,T^{*0,q}X\boxtimes T^{*0,q}X)$
is as in Theorem~\ref{t-gue140305_b}. If $q\in\set{n_-,n_+}$, then
\[
\Pi^{(q)}\equiv S_-+S_+\quad\text{and}\quad
N^{(q)}\equiv G\quad\text{on $D$},
\]
where $S_-, S_+\in L^{0}_{\frac{1}{2},\frac{1}{2}}(D,T^{*0,q}X
\boxtimes T^{*0,q}X)$ and $G\in L^{-1}_{\frac{1}{2},\frac{1}{2}}(D,T^{*0,q}X\boxtimes T^{*0,q}X)$
are as in Theorem~\ref{t-gue140305_b}.

In particular, for any CR submanifold in $\Complex\mathbb P^N$ of the form \eqref{e-gue140303},
the associated Szeg\H{o} kernel admits a full asymptotic expansion.
\end{cor}
For hypersurfaces of type \eqref{e-gue140303} of signature $(1,N-3)$, Biquard  \cite{Biq05} studied the filling problem 
for small deformations of the CR structure. 

From Corollary~\ref{c-gue140306} and Theorem~\ref{t-gue140223III}, we establish the
global embeddablity for three dimensional compact strictly pseudoconvex CR manifolds 
with transversal CR $S^1$ actions (Theorem \ref{t-gue140306f}).
\begin{thm}\label{t-gue140223III-I}
Let $(X,T^{1,0}X)$ be a compact strictly pseudoconvex CR manifold of dimension three with a transversal 
CR $S^1$ action. Then $X$ can be CR embedded into $\Complex^N$, for some $N\in\mathbb N$.
\end{thm}
\begin{proof}
Let $T\in \cC^\infty(X,TX)$ be the real vector field induced by the given transversal CR $S^1$ action on 
$X$ and we fix a $T$-rigid Hermitian metric $\langle\,\cdot\,|\,\cdot\,\rangle$ on $\Complex TX$ such that 
$T^{1,0}X\perp T^{0,1}X$, $T\perp (T^{1,0}X\oplus T^{0,1}X)$, $\langle\,T\,|\,T\,\rangle=1$ and 
$\langle\,u\,|v\,\rangle$ is real if $u, v$ are real tangent vectors and we take $m(x)$ to be the volume 
form induced by the given $T$-rigid Hermitian metric $\langle\,\cdot\,|\,\cdot\,\rangle$. 
We will use the same notations as before. From  Theorem~\ref{t-gue140223III}, 
we know that $\Box^{(0)}_b:{\rm Dom\,}\Box^{(q)}_b\To L^2(X)$ has closed range. 
Let $N^{(0)}:L^2(X)\To{\rm Dom\,}\Box^{(q)}_b$ be the partial inverse of $\Box^{(q)}_b$. 
From Corollary~\ref{c-gue140306}, we have
\begin{equation}\label{e-gue140223fa}
\begin{split}
&\mbox{$\Box^{(0)}_bN^{(0)}+\Pi^{(0)}=I$ on $L^2(X)$},\\
&\mbox{$N^{(0)}\Box^{(0)}_b+\Pi^{(0)}=I$ on ${\rm Dom\,}\Box^{(0)}_b$},\\
&N^{(0)}\in L^{-1}_{\frac{1}{2},\frac{1}{2}}(X),\ \ \Pi^{(0)}\in L^{0}_{\frac{1}{2},\frac{1}{2}}(X).
\end{split}
\end{equation}

From Kohn's result \cite{Koh86}, in order to prove that $X$ can be CR embedded 
into $\Complex^N$, for some $N\in\mathbb N$, we only need to prove that
$\ddbar_b:{\rm Dom\,}\ddbar_b\subset L^2(X)\To L^2_{(0,1)}(X)$ has closed range. Let $\ddbar_bu_j=v_j$,
$u_j\in{\rm Dom\,}\ddbar_b$, $v_j\in L^2_{(0,1)}(X)$, $j=1,2,\ldots$\,, 
with $v_j\To v\in L^2_{(0,1)}(X)$ as $j\To\infty$. 
We are going to prove that there is a $g\in{\rm Dom\,}\ddbar_b$ such that $\ddbar_bg=v$. 
We claim that for every $j=1,2,\ldots$\,,
\begin{equation}\label{e-gue140223faI}
N^{(0)}\ol{\pr}^{*,f}_bv_j\in L^2_{(0,1)}(X)\:\:\text{and}\:\:
(I-\Pi^{(0)})u_j=N^{(0)}\ol{\pr}^{*,f}_bv_j,
\end{equation}
where $\ol{\pr}^{*,f}_b$ is the formal adjoint of $\ddbar_b$ (acting on distributions). 
Since $N^{(0)}\in L^{-1}_{\frac{1}{2},\frac{1}{2}}(X)$, it is clearly that $N^{(0)}\ol{\pr}^{*,f}_bv_j\in L^2(X)$, $\forall j$. 
Fix $j\in\mathbb{N}$. Let $f_s\in \cC^\infty(X)$, $s\in\mathbb{N}$, with $f_s\To u_j$ in $\mathscr D'(X)$ as $s\To\infty$.
From \eqref{e-gue140223fa}, we have
\begin{equation}\label{e-gue140223faII}
\begin{split}
N^{(0)}\ol{\pr}^{*,f}_b\ddbar_bf_s=N^{(0)}\Box^{(0)}_bf_s=(I-\Pi^{(0)})f_s\rightarrow(I-\Pi^{(0)})u_j\,\:\: 
\text{in $\mathscr D'(X)$ as $j\To\infty$}.
\end{split}
\end{equation}
Note that $N^{(0)}\ol{\pr}^{*,f}_b\ddbar_bf_s\To N^{(0)}\ol{\pr}^{*,f}_b\ddbar_bu_j=N^{(0)}\ol{\pr}^{*,f}_bv_j$ 
in $\mathscr D'(X)$ as $j\To\infty$. From this observation and \eqref{e-gue140223faII}, 
the claim \eqref{e-gue140223faI} follows.
Since $N^{(0)}\in L^{-1}_{\frac{1}{2},\frac{1}{2}}(X)$,
\begin{equation}\label{e-gue140223faIII}
\mbox{$N^{(0)}\ol{\pr}^{*,f}_b:L^2_{(0,1)}(X)\To L^2(X)$ is continuous}.
\end{equation}
From \eqref{e-gue140223faI} and \eqref{e-gue140223faIII}, we conclude that
\[\mbox{$(I-\Pi^{(0)})u_j=N^{(0)}\ol{\pr}^{*,f}_bv_j\To N^{(0)}\ol{\pr}^{*,f}_bv=:u$ in $L^2(X)$}.\]
Thus, $\ddbar_bu=v$ in the sense of distribution. Since $v\in L^2_{(0,1)}(X)$, $u\in{\rm Dom\,}\ddbar_b$. 
We have proved that $\ddbar_b:{\rm Dom\,}\ddbar_b\subset L^2(X)\To L^2_{(0,1)}(X)$ has closed range. 
The theorem follows.
\end{proof}
\begin{ex}[Grauert tube]
Let $M$ be a compact complex manifold endowed with a Hermitian metric $\Theta$ and associated 
Riemannian metric $g^{TM}$. 
We consider a Hermitian holomorphic line bundle $(L,h^L)$ on $M$. The Grauert tube 
associated to $(L,h^L)$
is the disc bundle $G= \{u\in L^*, |u|_{h^{L^*}}<1\}$, with defining function $\varrho:L^*\to\mathbb{R}$, 
$\varrho=|u|^2_{h^{L^*}}-1$. The boundary  $X= \partial G=\{u\in L^*, |u|_{h^{L^*}}=1\}$ 
is the unit circle bundle in $L^*$. 
The Grauert tube was introduced by Grauert \cite{Gra:62}, one important application being the Kodaira
embedding theorem for singular spaces.

Let $\nabla ^L$ be the Chern
connection on $(L,h^L)$ and let $R^L=(\nabla ^L)^2$ be the
Chern curvature. 
The Levi form of $\varrho$
restricted to the complex tangent plane of $X$ coincides with the pull-back
of $\omega=\sqrt{-1}R^L$ through the canonical projection $\rho:X\to M$.
Therefore, the signature of the Levi form of $\varrho$ coincides with the signature of the curvature
form $\sqrt{-1}R^L$.

Note that $\rho : X\to M$ is a $S^1$-principal bundle and there exists a canonical $S^1$ action on $X$.
The connection $\nabla^L$ on  $L$ induces a connection 
on this $S^1$-principal bundle. 
Let $T^H X \subset TX$ be the corresponding horizontal bundle.
Let us introduce the Riemannian metric $g^{TX}=\rho^*(g^{TM})\oplus 
d\vartheta^2$ on $TX=T^HX\oplus TS^1$. 
We will denote by $\ddbar_b^*$ the formal adjoint of $\ddbar_b$ with 
respect to this metric and form the Kohn-Laplacian $\Box_b$.
The operators $\ddbar_b$, $\ddbar_b^*$ and $\Box^{(q)}_b$ commute with the action of $S^1$ 
on $X$.

Consider the space $\cC^\infty(X)_p$ of smooth functions $f$ on $Y$ which 
transform under the action $(y,\vartheta)\mapsto e^{i\vartheta}y$ of $S^1$ 
according to the law
\begin{equation}\label{equi}
f(e^{i\vartheta}y)=e^{ip\vartheta}f(y).
\end{equation}
Then 
$\cC^\infty(X)_p=B^{0,0}_{-p}(X)$, where $B^{0,q}_{m}(X)$ were defined in \eqref{e-gue140218I}.
Let us endow $\Omega^{\bullet,\bullet}(M,L^p)$ with the $L^2$ inner product induced by $\Theta$ and $h^L$.
There exists a natural isometry 
\[B^{0,0}_{-p}(X)=\cC^\infty(X)_p\cong\Omega^{0,0}(M,L^p).\]
More generally, consider the space of sections $\Omega^{0,k}(X)_p$ 
which transform 
under the action of $S^1$ according to \eqref{equi}. Then $\Omega^{0,k}(X)_p=B^{0,k}_{-p}(X)$ is naturally 
isometric to the space $\Omega^{0,k}(X,L^p)$.
In this way we obtain an interpretation of the spaces $B^{0,q}_{-p}(X)$ , $\mathcal{B}^{0,q}_{-p}(X)$ 
and of the projectors $Q_{\leq0}$, $Q_{\geq0}$ in terms
of the sections of $L^{p}$.  For more details on the relation of the Szeg\H{o} projection and the Bergman kernel of $L^p$
one can consult \cite[\S1.5]{MM06}, \cite[\S3.2]{MM08a}.

\end{ex}

\section{Szeg\H{o} kernel asymptotic expansion on weakly pseudoconvex CR manifolds} \label{s-gue140224}

By using Theorem~\ref{t-gue140305VIa}, we establish Szeg\H{o} kernel asymptotic expansions on
some weakly pseudoconvex CR manifolds.
We will consider in Section  \ref{s_gue140812} the case of boundaries of weakly pseudoconvex
domains (corresponding to Corollary \ref{c-140621} (i)). We also give an application to
the asymptotics of the Bergman kernel of a semi positive line bundle. 
In Section \ref{s_gue_2_140812} we study some non-compact weakly pseudoconvex 
domains.

\subsection{Compact pseudoconvex domains}\label{s_gue140812}
Let $G$ be a relatively compact, weakly pseudoconvex domain, with
smooth boundary $X$, in a complex manifold $G'$ of dimension $n$. Then $X$ is a compact weakly
pseudoconvex CR manifold of dimension $2n-1$ with CR structure
$T^{1,0}X:=T^{1,0}G'\cap\Complex TX$.
\begin{thm}\label{t-gue140305VIb}
Let $G$ be a relatively compact domain in a complex manifold $G'$ of dimension $n$, such that 
$G$ has smooth boundary $X=\partial G$, which is everywhere weakly pseudoconvex and 
strictly pseudoconvex on an open subset $D\subset X$. Fix $D_0\Subset D$. 
Assume that there exist a smooth strictly plurisubharmonic function defined in a neighborhood of $X$.
%
%
%
Let $\phi\in\cC^\infty(G')$ be a defining function of $G$, let $\langle\,\cdot\,|\,\cdot\,\rangle$ be 
a Hermitian metric  on $G'$ and let $v(x)$ be the induced volume 
form on $X$. Let $m(x)$ be a volume form on $X$ and consider the corresponding space $L^2(X)$. Then, the
kernel of the Szeg\H{o} projector $\Pi^{(0)}:L^2(X)\to\ker\overline\partial_b$ has the form
\begin{equation}\label{e-falI}
\Pi^{(0)}(x,y)\equiv\int^\infty_0e^{i\varphi(x,y)t}s(x,y,t)dt\,\;\;\text{on $D_0$},\end{equation}
where $\varphi(x,y)\in \cC^\infty(U\times U)$ is an almost analytic extension of $\phi$ as in \eqref{e:aae} to 
some neighbourhood $U$ be  of $D_0$ in $G'$,
and  $s(x, y, t)\in S^{n-1}_{{\rm cl\,}}\big(D\times D\times\mathbb{R}_+\big)$. Moreover the leading term $s_0(x,y)$ of the expansion \eqref{e-fal} of $s_0(x,y,t)$ satisfies 
\[\begin{split}
&s_0(x,x)=\frac{1}{2}\pi^{-n}\frac{v(x)}{m(x)}\abs{{\rm det\,}\mathcal{L}_x},\ \ \forall x\in D_0,
\end{split}\]
where $\mathcal{L}_x$ is the restriction of $\mathcal{L}_x(\phi)$ to the tangent space 
$T^{(1,0)}X$, $|\!\det\mathcal{L}_x|=\abs{\mu_1(x)}\ldots\abs{\mu_{n-1}(x)}$, with
$\mu_1(x),\ldots,\mu_{n-1}(x)$ the eigenvalues of $\mathcal{L}_x$ with respect to 
$\langle\,\cdot\,|\,\cdot\,\rangle$.
\end{thm}
\begin{proof}
By a theorem of Kohn \cite[p.\,543]{Koh86} we know that if $G$ meets the conditions in statement above,
then Kohn's Laplacian $\Box^{(0)}_b$ has $L^2$ closed range. For boundaries of pseudoconvex domains in 
$\mathbb{C}^n$ the closed range property was shown in \cite{BoSh:86,Shaw85}. By Theorem \ref{t-gue140305VIc}
we deduce that $\Pi^{(0)}$ is a complex Fourier integral operator on $D_0$ and $\Pi^{(0)}(x,y)$ has the form
\eqref{e-falI} with a phase function as in \eqref{e-gue140205IV}. 

Fix $p\in D_0$ and take local coordinates $x=(x_1,x_2,\ldots,x_{2n-1})$ of $X$ defined in a small neighbourhood 
of $p$ in $D_0$ such that $x(p)=0$ and $\omega_0(p)=dx_{2n-1}$. 
It is easy to see that $\frac{\pr\varphi}{\pr y_{2n-1}}(0,0)=1=\frac{\pr\varphi_-}{\pr y_{2n-1}}(0,0)$, 
where $\varphi_-$ is as in Theorem~\ref{t-gue140305_b}.  From the 
Malgrange preparation theorem \cite[Theorem 7.57]{Hor03}, 
we conclude that in some small neighbourhood of $(p,p)$ in $D_0\times D_0$, we can find $f(x,y), f_1(x,y)\in \cC^\infty$ such that
\[
\begin{split}
&\varphi_-(x,y)=f(x,y)(y_{2n-1}+h(x,y')),\\
&\varphi(x,y)=f_1(x,y)(y_{2n-1}+h_1(x,y'))
\end{split}
\]
in some small neighbourhood of $(p,p)$ in $D_0\times D_0$, where 
$y'=(y_1,\ldots,y_{2n-2})$, $h, h_1\in\cC^\infty$. It is not difficult to see that 
$y_{2n-1}+h(x,y')$, $y_{2n-1}+h_1(x,y')$ satisfy \eqref{e-gue140205VI}, 
\eqref{e-gue140205VII} and 
\[
\ddbar_b(y_{2n-1}+h(x,y'))\,,\:\: \ddbar_b(y_{2n-1}+h_1(x,y'))
\] 
vanish to infinite order on $x=y$. From this observation, it is straightforward 
to check that $h(x,y')-h_1(x,y')$ vanishes to infinite order on $x=y$. 
We conclude that $\varphi$ and $\varphi_-$ are equivalent. The theorem follows. 
\end{proof}
\begin{thm}\label{t-gue140305VId}
Let $M$ be a projective manifold and let $L\to M$ be an ample line bundle. Let $h^L$ be a smooth
Hermitian metric on $L$ such that $\sqrt{-1}R^L$ is semipositive. Consider the Grauert tube
$G=\{v\in L^*:|v|_{h^{L^*}}<1\}$, $X=\partial G$ and $\rho:X\to M$ the projection.
Then the Szeg\H{o} projector $\Pi^{(0)}:L^2_{(0,0)}(X)\to\ker\overline{\partial}_b$ is a Fourier integral operator 
with complex phase on the set $\rho^{-1}(M(0))$, where $M(0)\subset M$ is the set where $\sqrt{-1}R^L$
is positive.
\end{thm}
\begin{proof}
Since $L$ is ample, there exists a Hermitian metric $h^L_0$ on $L$ with positive curvature. The Levi form of the 
function $\varrho_0:L^*\to\mathbb{R}$, $\varrho=|u|^2_{h^{L^*}_0}$ is positive definite on the complex tangent
space of any level set $\varrho_0=c>0$. It is easy to see that given any compact set $K\subset L^*\setminus 0$ 
we can modify $\varrho$ to construct a strictly plurisubharmonic on $K$. Therefore the Grauert tube $G$ 
fulfills the hypothesis of Theorem \ref{t-gue140305VIb}.
\end{proof}
Theorems \ref{t-gue140305VIb} and  \ref{t-gue140305VId} are based on closed range property
for $\ddbar_b$. Note that Donnelly \cite{Don03} gave an example of a semipositive line bundle $L\to M$ 
which is positive at some point (i.\,e.\ $M(0)\neq\emptyset$), whose
Grauert tube doesn't have the closed range property
for $\ddbar_b$. 

An important application of the asymptotics of the Szeg\H{o} kernel of the Grauert tube is the 
asymptotics of the Bergman kernel of the tensor powers of the bundle $L$. This was first achieved by 
Catlin~\cite{Ca99} and Zelditch~\cite{Zelditch98} for a positively curved metric $h^L$.
We exemplify here such an application of  Theorem \ref{t-gue140305VId}.

Consider a Hermitian metric $\Theta$ on $M$ and introduce the $L^2$ inner product on $\cC^\infty(M,L^p)$
induced by the volume element $\Theta^n/n!$ and the metric $h^{L^p}$ and denote by $L^2(M,L^p)$
the corresponding $L^2$ space. Let $P_p:L^2(M,L^p)\to H^0(M,L^p)$ be
the orthogonal projection, called Bergman projection. Its kernel $P_{p}(\,\cdot\,,\cdot)$ is called the Bergman 
kernel. The restriction to the diagonal of $P_{p}(\,\cdot\,,\cdot)$ is denoted $P_{p}(\cdot)$ and is 
called the Bergman kernel function (or density). 
We refer the reader to the book \cite{MM07} and to the survey \cite{Ma10}
for a comprehensive study of the Bergman kernel and its applications.
\begin{cor}\label{t-gue140305VIe}
Let $M$ be a projective manifold of dimension $n$ and let $L\to M$ be an ample line bundle. Let $h^L$ be a smooth
Hermitian metric on $L$ such that $\sqrt{-1}R^L$ is semipositive. 
Then the Bergman kernel function $P_{p}(\,\cdot\,)$ has the asymptotic expansion
\begin{equation}\label{e-falII}
P_{p}(x)\sim\sum^\infty_{j=0}p^{n-j}b^{(0)}_j(x)\ \ \mbox{locally uniformly on $M(0)$}, \end{equation}
where $b^{(0)}_j\in\cC^\infty(M(0))$, $j=0,1,2,\ldots$\,. 
\end{cor}
\begin{proof} The Bergman kernel $P_p$ and the Szeg\H{o} kernel $\Pi^{(0)}$  are linked by the formula
\begin{equation}\label{Fourier}
P_p(x)=\frac{1}{2\pi}\int_{S^1}\Pi^{(0)}(e^{i\vartheta}y,y)
e^{-ip\vartheta}\,d\vartheta\,,
\end{equation}
where $x\in M$ and $y\in X$ satisfy $\rho(y)=x$, that is, $P_p(x)$ represent the Fourier coefficients of 
the distribution $\Pi^{(0)}(y,y)$.
Since $\Pi^{(0)}$ is a Fourier integral operator on $\rho^{-1}(M(0))$ by Theorem \ref{t-gue140305VId},
we deduce the asymptotics  \eqref{e-falII} exactly as in \cite{Ca99,Zelditch98} 
by applying the stationary phase method.
\end{proof}
Corollary \ref{t-gue140305VIe} was obtained by different methods by 
Berman~\cite{Be07} in the case of a projective manifold $M$
and in \cite[Theorem 1.10]{HM12} for a general Hermitian manifold $M$.

\subsection{Non-compact pseudoconvex domains}\label{s_gue_2_140812}
Now, we consider non-compact cases. By using Theorem~\ref{t-gue140305VIa}, 
we will establish Szeg\H{o} kernel asymptotic expansions on some non-compact CR manifolds.
Let $\Gamma$ be a strictly pseudoconvex domain in $\Complex^{n-1}$, $n\geq 2$. 
Consider $X:=\Gamma\times\Real$. Let $(z,t)$ be the coordinates of $X$, where $z=(z_1,\ldots,z_{n-1})$ denote the
coordinates of $\Complex^{n-1}$ and $t$ is the coordinate of $\Real$. We write $z_j=x_{2j-1}+ix_{2j}$, $j=1,\ldots,n-1$. 
We also write $(z,t)=x=(x_1,\ldots,x_{2n-1})$ and let $\eta=(\eta_1,\ldots,\eta_{2n-1})$ be the dual variables of $x$. 
Let $\mu(z)\in \cC^\infty(\Gamma,\Real)$. We define
$T^{1,0}X$ to be the space spanned by
\[
\Big\{\frac{\pr}{\pr z_j}+i\frac{\pr\mu}{\pr z_j}\frac{\pr}{\pr t},\ \ j=1,\ldots,n-1\Big\}.
\]
Then $(X,T^{1,0}X)$ is a non-compact CR manifold of dimension $2n-1$. We take a Hermitian metric 
$\langle\,\cdot\,|\,\cdot\,\rangle$ on the complexified tangent bundle $\Complex TX$ such that
\[
\Big\lbrace
\dfrac{\pr}{\pr z_j}+i\frac{\pr\mu}{\pr z_j}\frac{\pr}{\pr t}\,, 
\frac{\pr}{\pr\ol z_j}-i\frac{\pr\mu}{\pr\ol z_j}\frac{\pr}{\pr t}\,, T:=\frac{\pr}{\pr t}\,;\, j=1,\ldots,n-1\Big\rbrace
\]
 is an orthonormal basis. The dual basis of the complexified cotangent bundle $\Complex T^*X$ is
\[
\Big\lbrace
dz_j\,,\, d\ol z_j\,,\, -\omega_0:=
dt+\textstyle\sum^{n-1}_{j=1}(-i\frac{\pr\mu}{\pr z_j}dz_j+i\frac{\pr\mu}{\pr\ol z_j}d\ol z_j); j=1,\ldots,n-1
\Big\rbrace\,.
\]
The Levi form $\mathcal{L}_p$ of $X$ at $p\in X$ is given by
\[\mathcal{L}_p=\sum^{n-1}_{j,\ell=1}\frac{\pr^2\mu}{\pr z_j\pr\ol z_{\ell}}(p)dz_j\wedge d\ol z_{\ell}.\]
Now, we assume that
\begin{equation}\label{e-gue140226m}
\left(\frac{\pr^2\mu}{\pr z_j\pr\ol z_\ell}(z)\right)^{n-1}_{j,\ell=1}\geq0,\ \ \forall z\in\Gamma,
\end{equation}
and take
\begin{equation}\label{e-gue140306}
m(x):=e^{-2\abs{z}^2}dx_1dx_2\ldots dx_{2n-1}
\end{equation}
to be the volume form on $X$. Thus, $X$ is a weakly pseudoconvex CR manifold.

Take $\tau\in \cC^\infty(\Real,[0,1])$ with $\tau=0$ on $]-\infty,\frac{1}{4}]$, $\tau=1$ 
on $[\frac{1}{2},\infty[$. We also write $\theta$ to denote the $t$ variable.
Let $Q^{(0)}:\cC^\infty_0(X)\To \cC^\infty(X)$ be the operator given by
\begin{equation}\label{e-gue131222tm}
Q^{(0)}u(z,t):=\frac{1}{2\pi}\int e^{i\langle t-\theta,\eta\rangle }u(z,\theta)\tau(\eta)d\eta d\theta\in 
\cC^\infty(X),\ \ u(z,t)\in \cC^\infty_0(X).
\end{equation}
We can extend $Q^{(0)}$ to $L^2(X)$ such that
\begin{equation}\label{e-gue131222tIVm}
\begin{split}
&\mbox{$Q^{(0)}:L^2(X)\To L^2(X)$ is continuous},\\
&\norm{Q^{(0)}u}\leq\norm{u},\ \ \forall u\in L^2(X),\\
&Q^{(0)}\in L^0_{{\rm cl\,}}(X),\\
&\mbox{$Q^{(0)}\equiv0$ at $\Sigma^+\cap T^*D$, $\forall D\Subset X$}.
\end{split}
\end{equation}
We will prove that $\Box^{(0)}_b$ has local $L^2$ closed 
range property on $X$ with respect to $Q^{(0)}$ under certain assumptions. More precisely, we have the following.

\begin{thm}\label{t-gue140306}
Let $\Gamma=\Complex^{n-1}$ or $\Gamma$ be a bounded strictly pseudoconvex domain
in $\Complex^{n-1}$.
Let $\mu\in \cC^\infty(\Gamma')$, where $\Gamma'$ is an open neighbourhood of 
$\ol\Gamma$ {\rm(}if $\Gamma=\Complex^{n-1}$ this means just that $\mu\in \cC^\infty(\Complex^{n-1})${\rm)}. When $\Gamma=\Complex^{n-1}$, we assume that $\mu\geq0$.
Then
\begin{equation}\label{e-gue131226II}
\norm{Q^{(0)}(I-\Pi^{(0)})u}^2\leq C_0\norm{\ddbar_bu}^2,\ \ \forall u\in \cC^\infty_0(X),
\end{equation}
where $C_0>0$ is a constant independent of $u$.
In particular, $\Box^{(0)}_{b}$ has local $L^2$ closed range on $X$ with respect to $Q^{(0)}$.
\end{thm}
From Theorem~\ref{t-gue140305VIa}, \eqref{e-gue131222tIVm} and Theorem~\ref{t-gue140306}, we deduce
\begin{thm}\label{t-gue140306I}
With the notations and assumptions of Theorem \ref{t-gue140306},
suppose that the matrix $\left(\frac{\pr^2\mu}{\pr z_j\pr\ol z_{\ell}}(x)\right)^{n-1}_{j,\ell=1}$ 
is positive  definite on an open set $D\Subset X$. Then,
\begin{equation}\label{e-gue140227fb}
Q^{(0)}\Pi^{(0)}Q^{(0)}(x,y)\equiv\int^\infty_0e^{i\varphi_-(x,y)t}a(x,y,t)dt\,,\:\:\text{on $D$},
\end{equation}
where $\varphi_-(x,y)\in \cC^\infty(D\times D)$  is as in Theorem~\ref{t-gue140305_b} and
\[\begin{split}
&a(x, y, t)\in S^{n-1}_{{\rm cl\,}}\big(D\times D\times\mathbb{R}_+\big), \\
&a(x, y, t)\sim\sum^\infty_{j=0}a_j(x, y)t^{n-1-j}\quad\text{ in }S^{n-1}_{1, 0}
\big(D\times D\times\mathbb{R}_+\big)\,,\\
&a_0(x,x)=\frac{1}{2}\pi^{-n}{\rm det\,}\left(\frac{\pr^2\mu}{\pr z_j\pr\ol z_{\ell}}(x)\right)^{n-1}_{j,\ell=1},\ \ \forall x\in D.
\end{split}\]
\end{thm}

We first introduce the partial Fourier transform $\mathcal{F}$ and the operator $Q^{(q)}$.
Let $u\in\Omega^{0,q}_0(X)$. Put
\begin{equation} \label{e-gue131222tI}
(\mathcal{F}u)(z, \eta)=\int_{\Real}e^{-i\eta t}u(z,t)dt.
\end{equation}
From Parseval's formula, we have
\begin{equation}\label{e-gue131225}
\norm{\mathcal{F}u}^2=\int_{X}\abs{(\mathcal{F}u)(z,\eta)}^2d\eta dv(z)
=2\pi\int_{X}\abs{u(z,t)}^2dtdv(z)=2\pi\norm{u}^2,
\end{equation}
where $dv(z)=e^{-2\abs{z}^2}dx_1dx_2\ldots dx_{2n-2}$.
Thus, we can extend the operator $\mathcal{F}$ to $L^2_{(0,q)}(X)$ and
\begin{equation}\label{e-gue131225I-I}
\begin{split}
&\mbox{$\mathcal{F}:L^2_{(0,q)}(X)\To L^2_{(0,q)}(X)$ is continuous},\\
&\norm{\mathcal{F}u}=\sqrt{2\pi}\norm{u},\ \ \forall u\in L^2_{(0,q)}(X).
\end{split}
\end{equation}
For $u\in L^2_{(0,q)}(X)$, we call $\mathcal{F}u$ the partial Fourier transform of $u$ with respect to $t$.

Take $\tau\in \cC^\infty(\Real,[0,1])$ with $\tau=0$ on $]-\infty,\frac{1}{4}]$, $\tau=1$ on $[\frac{1}{2},\infty[$. 
We also write $\theta$ to denote the $t$ variable.
Let $Q^{(q)}:\Omega^{0,q}_0(X)\To\Omega^{0,q}(X)$ be the operator given by
\begin{equation}\label{e-gue131222t}
Q^{(q)}u(z,t):=\frac{1}{2\pi}\int e^{i\langle t-\theta,\eta \rangle }u(z,\theta)\tau(\eta)d\eta d\theta\in\Omega^{0,q}(X),\ 
\ u(z,t)\in\Omega^{0,q}_0(X).
\end{equation}
From Parseval's formula and \eqref{e-gue131225}, we have
\begin{equation}\label{e-gue131222tIII}
\begin{split}
\norm{Q^{(q)}u}^2&=\frac{1}{4\pi^2}\int_{X}\abs{\int e^{i\langle t-\theta,\eta \rangle }
u(z,\theta)\tau(\eta)d\eta d\theta}^2dv(z)dt\\
&=\frac{1}{4\pi^2}\int_{X}\abs{\int e^{i\langle t,\eta \rangle }(\mathcal{F}u)(z,\eta)\tau(\eta)d\eta}^2dv(z)dt\\
&=\frac{1}{2\pi}\int\abs{(\mathcal{F}u)(z,\eta)}^2\abs{\tau(\eta)}^2d\eta dv(z)\\
&\leq\frac{1}{2\pi}\int\abs{(\mathcal{F}u)(z,\eta)}^2d\eta dv(z)=\norm{u}^2,
\end{split}
\end{equation}
where $u\in\Omega^{0,q}_0(X)$. Thus, we can extend $Q^{(q)}$ to $L^2_{(0,q)}(X)$ and
\begin{equation}\label{e-gue131222tIV}
\begin{split}
&\mbox{$Q^{(q)}:L^2_{(0,q)}(X)\To L^2_{(0,q)}(X)$ is continuous},\\
&\norm{Q^{(q)}u}\leq\norm{u},\ \ \forall u\in L^2_{(0,q)}(X).
\end{split}
\end{equation}
We need

\begin{lem}\label{l-gue131225}
Let $u\in L^2_{(0,q)}(X)$. Then,
\begin{equation}\label{e-gue131225I}
(\mathcal{F}Q^{(q)}u)(z,\eta)=(\mathcal{F}u)(z,\eta)\tau(\eta).
\end{equation}
\end{lem}

\begin{proof}
Let $u_j\in\Omega^{0,q}_0(X)$, $j=1,2,\ldots$, with
$\lim_{j\To\infty}\norm{u_j-u}=0$. From \eqref{e-gue131222tIV} and \eqref{e-gue131225I-I}, we see that
\begin{equation}\label{e-gue131225II}
\mbox{$\mathcal{F}Q^{(q)}u_j\To\mathcal{F}Q^{(q)}u$ in $L^2_{(0,q)}(X)$ as $j\To\infty$}.
\end{equation}
From Fourier inversion formula, we have
\begin{equation}\label{e-gue131225III}
(\mathcal{F}Q^{(q)}u_j)(z,\eta)=(\mathcal{F}u_j)(z,\eta)\tau(\eta),\ \ j=1,\ldots.
\end{equation}
Note that  $(\mathcal{F}u_j)(z,\eta)\tau(\eta)\To(\mathcal{F}u)(z,\eta)\tau(\eta)$ in $L^2_{(0,q)}(X)$ 
as $j\To\infty$. From this observation, \eqref{e-gue131225III} and \eqref{e-gue131225II}, we obtain \eqref{e-gue131225I}.
\end{proof}

The following is straightforward. We omit the proofs.

\begin{lem}\label{l-gue131225I}
We have
\begin{equation}\label{e-gue131225IV}
\begin{split}
&Q^{(q)}:{\rm Dom\,}\ddbar_{b}\To{\rm Dom\,}\ddbar_{b},\ \ q=0,1,\ldots,n-1,\\
&\mbox{$Q^{(q+1)}\ddbar_{b}=\ddbar_{b}Q^{(q)}$ on ${\rm Dom\,}\ddbar_{b}$},\ \ q=0,1,\ldots,n-2,
\end{split}
\end{equation}
and
\begin{equation}\label{e-gue131225V}
\mbox{$Q^{(q)}\Pi^{(q)}=\Pi^{(q)}Q^{(q)}$ on $L^2_{(0,q)}(X)$}.
\end{equation}
Moreover, for $u\in\Omega^{0,q}_0(X)$, we have
\begin{equation}\label{e-gue131225VI}
\ddbar_z\bigr((\mathcal{F}u)(z,\eta)e^{\eta\mu(z)}\bigr)e^{-\eta\mu(z)}=(\mathcal{F}\ddbar_{b}u)(z,\eta),\ \ \forall (z,\eta)\in X,
\end{equation}
where $\mu\in \cC^\infty(\Gamma,\Real)$ is as in the beginning of Section~\ref{s-gue140224}.
\end{lem}

We will study now the local $L^2$ closed range property for $\Box^{(0)}_{b}$ with respect to $Q^{(0)}$. 
We pause and introduce some notations. Let $\Omega^{0,q}(\Gamma)$ be the space of all smooth $(0,q)$ forms on 
$\Gamma$ and let $\Omega^{0,q}_0(\Gamma)$ be the subspace of $\Omega^{0,q}(\Gamma)$ 
whose elements have compact support in $\Gamma$. We take the Hermitian metric 
$\langle\,\cdot\,|\,\cdot\,\rangle$ on $T^{*0,q}\Gamma$ 
the bundle of $(0, q)$ forms of $\Gamma$ so that
\[\{d\ol z_{j_1}\wedge d\ol z_{j_2}\wedge\ldots\wedge d\ol z_{j_q}; 1\leq j_1< j_2\ldots<j_q\leq n-1\}\]
is an orthonormal basis. Let $\Upsilon\in \cC^\infty(\Gamma,\Real)$ and let $(\,\cdot\,|\,\cdot)_{\Upsilon}$ 
be the $L^2$ inner product on $\Omega^{0,q}_0(\Gamma)$ given by
\[(f\,|\,g\,)_{\Upsilon}=\int\langle\,f\,|\,g\,\rangle e^{-2\Upsilon(z)}d\lambda(z),\ \ f, g\in\Omega^{0,q}_0(\Gamma),\]
where $d\lambda(z)=dx_1dx_2\ldots dx_{2n-2}$.
Let $L^2_{(0,q)}(\Gamma,\Upsilon)$ denote the completion of $\Omega^{0,q}_0(\Gamma)$ 
with respect to the inner product $(\cdot\,|\,\cdot\,)_{\Upsilon}$. We write 
$L^2(\Gamma,\Upsilon):=L^2_{(0,0)}(\Gamma,\Upsilon)$. Put
\[H^0(\Gamma,\Upsilon):=\set{f\in L^2(\Gamma,\Upsilon);\, \ddbar f=0}.\]
From now on, we assume that
\begin{equation}\label{e-gue140226}
\left(\frac{\pr^2\mu}{\pr z_j\pr\ol z_\ell}(z)\right)^{n-1}_{j,\ell=1}\geq0,\ \ \forall z\in\Gamma,
\end{equation}
and take
\[m(x):=e^{-2\abs{z}^2}dx_1dx_2\ldots dx_{2n-2}dt=e^{-2\abs{z}^2}d\lambda(z)dt\]
be the volume form on $X$.

Now, suppose $\Gamma=\Complex^{n-1}$ or $\Gamma$ is a bounded strictly pseudoconvex 
domain in $\Complex^{n-1}$.


\begin{proof}[Proof of Theorem {\ref{t-gue140306}} for $\Gamma=\Complex^{n-1}$]
Let $u\in \cC^\infty_0(X)$. We consider $Q^{(0)}(I-\Pi^{(0)})u$.
In view of \eqref{e-gue131225V}, we see that $Q^{(0)}(I-\Pi^{(0)})u=(I-\Pi^{(0)})Q^{(0)}u$. Put
\begin{equation}\label{e-gue131226aI}
v(z,\eta)=\mathcal{F}Q^{(0)}(I-\Pi^{(0)})u(z,\eta)e^{\eta\mu(z)}.
\end{equation}
From \eqref{e-gue131222tIV}, \eqref{e-gue131225I-I} and \eqref{e-gue131225I},
we see that $\int\abs{v(z,\eta)}^2e^{-2\eta\mu(z)-2\abs{z}^2}d\lambda(z)d\eta<\infty$
and $v(z,\eta)=0$ if $\eta\notin{\rm Supp\,}\tau(\eta)$. From Fubini's Theorem and
some elementary real analysis, we know that for every $\eta\in\Real$, $v(z,\eta)$
is a measurable function of $z$ and for almost every
$\eta\in\Real$, $v(z,\eta)\in L^2(\Gamma,\eta\mu(z)+\abs{z}^2)$
and for every $z\in\Gamma$, $v(z,\eta)$ is a measurable function of $\eta$
and for almost every $z\in\Gamma$, $\int\abs{v(z,\eta)}^2d\eta<\infty$.
Moreover, let $\beta\in L^2(\Gamma,\abs{z}^2)$, then the function
\[f(\eta):=\eta\mapsto\int v(z,\eta)\ol\beta(z)e^{-2\eta\mu(z)-2\abs{z}^2}d\lambda(z)\]
is measurable and  $f(\eta)$ is finite for almost every $\eta\in\Real$,
$f(\eta)=0$ if $\eta\notin{\rm Supp\,}\tau(\eta)$ and $f(\eta)\in L^2(\Real)$. We claim that
\begin{equation}\label{e-gue131229}
\begin{split}
&\mbox{For almost every $\eta\in\ol\Real_+$, $v(z,\eta)\in L^2(\Gamma,\eta\mu(z)+\abs{z}^2)$ and}\\
&(\,v(z,\eta)\,|\,\beta\,)_{\eta\mu+\abs{z}^2}=0,\ \ \forall\beta\in  H^0(\Gamma,\eta\mu(z)+\abs{z}^2).
\end{split}
\end{equation}
From the discussion after \eqref{e-gue131226aI}, we know that there is a measurable set $A_0$
in $\ol\Real_+$ with $\abs{A_0}=0$ such that for every $\eta\notin A_0$,
$v(z,\eta)\in L^2(\Gamma,\eta\mu(z)+\abs{z}^2)$, where $\abs{A_0}$ denote the Lebesgue measure of $A_0$.
Since $\mu\geq0$, $\set{z^\alpha;\, \alpha\in\mathbb N^{n-1}_0}$
is a basis for $H^0(\Complex^{n-1},\eta\mu(z)+\abs{z}^2)$, for every $\eta\geq0$.
Fix $\alpha\in\mathbb N^{n-1}_0$. We consider
\[f_\alpha(\eta)=\int v(z,\eta)\ol z^\alpha e^{-2\eta\mu(z)-2\abs{z}^2}d\lambda(z).\]
From the discussion after \eqref{e-gue131226aI}, we know that $f_\alpha(\eta)\in L^2(\Real)$.
Fix $n\in\mathbb N$, put
$h_n(\eta):=\ol f_\alpha(\eta)1_{[0,n]}(\eta)$.
Then,
\begin{equation}\label{e-gue140227}
\int^n_0\abs{f_\alpha(\eta)}^2d\eta=\int f_\alpha(\eta)h_n(\eta)d\eta=
\int v(z,\eta)\ol z^\alpha h_n(\eta) e^{-2\eta\mu(z)-2\abs{z}^2}d\lambda(z)d\eta.
\end{equation}
Let $\beta_\ell\in \cC^\infty_0(X)$, $\ell=1,2,\ldots$, such that
$\beta_\ell\To(I-\Pi^{(0)})u$ in $L^2(X)$ as $\ell\To\infty$.
From \eqref{e-gue131225I-I}, \eqref{e-gue131222tIV}, \eqref{e-gue131225I} and \eqref{e-gue140227}, we see that
\begin{equation}\label{e-gue131226b}
\lim_{\ell\To\infty}\int\mathcal{F}Q^{(0)}\beta_\ell(z,\eta)\ol z^\alpha h_n(\eta) e^{-\eta\mu(z)-2\abs{z}^2}d\lambda(z)d\eta\To\int^n_0\abs{f_\alpha(\eta)}^2d\eta.
\end{equation}
From \eqref{e-gue131225I} and Parseval's formula, we can check that
\begin{equation}\label{e-gue131226bI}
\begin{split}
&\int\mathcal{F}Q^{(0)}\beta_\ell(z,\eta)\ol z^\alpha h_n(\eta) e^{-\eta\mu(z)-2\abs{z}^2}d\lambda(z)d\eta\\
&=\int\mathcal{F}\beta_\ell(z,\eta)\tau(\eta)\ol z^\alpha h_n(\eta)e^{-\eta\mu(z)-2\abs{z}^2}d\lambda(z)d\eta\\
&=\int\beta_\ell(z,t)(\int\ol z^\alpha h_n(\eta)\tau(\eta)e^{-\eta\mu(z)-i\eta t}d\eta)e^{-2\abs{z}^2}d\lambda(z)dt\\
&\To\int(I-\Pi^{(0)})u(z,t)(\int \ol z^\alpha h_n(\eta)\tau(\eta)e^{-\eta\mu(z)-i\eta t}d\eta)e^{-2\abs{z}^2}d\lambda(z)dt,\\ 
&\mbox{as $\ell\To\infty$}.
\end{split}
\end{equation}
It is straightforward to check that the function
\[\int z^\alpha(z)\ol h_n(\eta)\tau(\eta)e^{-\eta\mu(z)+i\eta t}d\eta\in{\rm Ker\,}\ddbar_{b}\cap L^2(X).\]
Thus,
\begin{equation}\label{e-gue131226bII}
\int(I-\Pi^{(0)})u(z,t)(\int \ol z^\alpha h_n(\eta)\tau(\eta)e^{-\eta\mu(z)-i\eta t}d\eta)e^{-2\abs{z}^2}d\lambda(z)dt=0.
\end{equation}
From \eqref{e-gue131226bII}, \eqref{e-gue131226bI} and \eqref{e-gue131226b}, 
we conclude that $f_\alpha(\eta)=0$ almost everywhere.
Thus, there is a measurable set $A_\alpha\supset A_0$ in $\ol\Real_+$ with $\abs{A_\alpha}=0$ 
such that for every $\eta\notin A_\alpha$ we have $(\,v(z,\eta)\,|\,z^\alpha\,)_{\eta\mu+\abs{z}^2}=0$.
Put $A=\bigcup_{\alpha\in\mathbb N^{n-1}_0}A_\alpha$. Then, $\abs{A}=0$. 
We conclude that for every $\eta\notin A$, $\eta\geq0$,
\[(\,v(z,\eta)\,|\,\beta\,)_{\eta\mu+\abs{z}^2}=0,\ \ \forall\beta\in  H^0(\Gamma,\eta\mu+\abs{z}^2).\]
The claim \eqref{e-gue131229} follows.

Now, we can prove \eqref{e-gue131226II}. Let $u\in \cC^\infty_0(X)$. From \eqref{e-gue131225IV} 
and \eqref{e-gue131225VI}, we have
\begin{equation}\label{e-gue131226III}
\begin{split}
&\ddbar_{b}Q^{(0)}(I-\Pi^{(0)})u=Q^{(1)}\ddbar_{b}u,\\
&(\mathcal{F}Q^{(1)}\ddbar_{b}u)(z,\eta)=\ddbar_z(\mathcal{F}Q^{(0)}u(z,\eta)e^{\eta\mu(z)})e^{-\eta\mu(z)}.
\end{split}
\end{equation}
As before, we put $v(z,\eta)=\mathcal{F}Q^{(0)}(I-\Pi^{(0)})u(z,\eta)e^{\eta\mu(z)}$ and set
\[\ddbar_z\Bigr(\mathcal{F}Q^{(0)}(I-\Pi^{(0)})u(z,\eta)e^{\eta\mu(z)}\Bigr)=\ddbar_zv(z,\eta)=:g(z,\eta).\]
It is easy to see that
\begin{equation}\label{e-gue131226IV}
\begin{split}
&\ddbar_zg(z,\eta)=0,\\
&\mbox{$g(z,\eta)=0$ if $\eta\notin{\rm Supp\,}\tau(\eta)$},\\
&\int\abs{g(z,\eta)}^2e^{-2\eta\mu(z)-2\abs{z}^2}d\lambda(z)<\infty,\ \ \forall \eta\in{\rm Supp\,}\tau(\eta).
\end{split}
\end{equation}
From \eqref{e-gue140226}, we see that there is a $C>0$ independent of $\eta\in{\rm Supp\,}\tau(\eta)$ such that
\begin{equation}\label{e-gue131226V}
\begin{split}
&\sum^{n-1}_{j,\ell=1}\frac{\pr^2(\abs{z}^2+
\eta\mu(z))}{\pr z_j\pr\ol z_\ell}(z)w_j\ol w_\ell\geq C\sum^{n-1}_{j=1}\abs{w_j}^2,\forall (w_1,\ldots,w_{n-1})
\in\Complex^{n-1}, z\in\Gamma, \eta\in
{\rm Supp\,}\tau(\eta).
\end{split}
\end{equation}

From \eqref{e-gue131226V} and H\"ormander's $L^2$ estimates \cite[Lemma 4.4.1]{Hor90}, 
we conclude that for every $\eta\in{\rm Supp\,}\tau(\eta)$, we can find a 
$\beta_\eta(z)\in L^2_{(0,1)}(\Gamma,\eta\mu(z)+\abs{z}^2)$ such that
\begin{equation}\label{e-gue131226VI}
\ddbar_z\beta_\eta(z)=g(z,\eta)
\end{equation}
and
\begin{equation}\label{e-gue131226VII}
\int\abs{\beta_\eta(z)}^2e^{-2\eta\mu(z)-2\abs{z}^2}d\lambda(z)\leq 
C\int\abs{g(z,\eta)}^2e^{-2\eta\mu(z)-2\abs{z}^2}d\lambda(z).
\end{equation}
In view of \eqref{e-gue131229}, we see that there is a measurable set $A$ in $\ol\Real_+$
with Lebesgue measure zero in $\Real$ such that for every $\eta\notin A$, $\eta\geq0$, 
$v(z,\eta)\perp H^0(\Gamma,\eta\mu(z)+\abs{z}^2)$. Thus, for every $\eta\notin A$,
$\eta\geq0$, $v(z,\eta)$ has the minimum $L^2$ norm with respect to 
$(\,\cdot\,|\,\cdot\,)_{\eta\mu+\abs{z}^2}$ of the solutions $\ddbar\alpha=\ddbar_zv(z,\eta)=g(z,\eta)$. 
From this observation and \eqref{e-gue131226VII}, we conclude that $\forall\eta\notin A$,
\begin{equation}\label{e-gue131227}
\int\abs{v(z,\eta)}^2e^{-2\eta\mu(z)-2\abs{z}^2}d\lambda(z)\leq 
C\int\abs{\ddbar_zv(z,\eta)}^2e^{-2\eta\mu(z)-2\abs{z}^2}d\lambda(z).
\end{equation}
Thus,
\begin{equation}\label{e-gue131227I}
\int\abs{v(z,\eta)}^2e^{-2\eta\mu(z)-2\abs{z}^2}d\lambda(z)d\eta\leq 
C\int\abs{\ddbar_zv(z,\eta)}^2e^{-2\eta\mu(z)-2\abs{z}^2}d\lambda(z)d\eta.
\end{equation}
From the definition of $v(z,\eta)$, \eqref{e-gue131225I-I}, \eqref{e-gue131226III} 
and \eqref{e-gue131222tIV}, it is straightforward to see that
\begin{equation}\label{e-gue131227II}
\begin{split}
&\int\abs{v(z,\eta)}^2e^{-2\eta\mu(z)-2\abs{z}^2}d\lambda(z)d\eta\\
&=(2\pi)\int\abs{Q^{(0)}(I-\Pi^{(0)})u(z,t)}^2e^{-2\abs{z}^2}d\lambda(z)dt
\end{split}
\end{equation}
and
\begin{equation}\label{e-gue131227III}
\begin{split}
&\int\abs{\ddbar_zv(z,\eta)}^2e^{-2\eta\mu(z)-2\abs{z}^2}d\lambda(z)d\eta\\
&=(2\pi)\int\abs{Q^{(1)}\ddbar_{b}u(z,t)}^2e^{-2\abs{z}^2}d\lambda(z)dt\\
&\leq(2\pi)\int\abs{\ddbar_{b}u(z,t)}^2e^{-2\abs{z}^2}d\lambda(z)dt.
\end{split}
\end{equation}
From \eqref{e-gue131227I}, \eqref{e-gue131227II} and \eqref{e-gue131227III}, we conclude that
\[\begin{split}
\norm{Q^{(0)}(I-\Pi^{(0)})u}^2&=
\int\abs{Q^{(0)}(I-\Pi^{(0)})u(z,t)}^2e^{-2\abs{z}^2}d\lambda(z)dt\\
&\leq C\int\abs{\ddbar_{b}u(z,t)}^2e^{-2\abs{z}^2}d\lambda(z)dt=C\norm{\ddbar_{b}u}^2.\end{split}\]
Theorem~\ref{t-gue140306} for $\Gamma=\mathbb{C}^{n-1}$ follows.
\end{proof}

Now, we consider the case when $\Gamma$ is a bounded strictly pseudoconvex domain in $\Complex^{n-1}$.


\begin{proof}[Proof of Theorem~\ref{t-gue140306} for $\Gamma$ a bounded strictly pseudoconvex 
domain in $\Complex^{n-1}$]
Let $u\in \cC^\infty_0(X)$. We have
\begin{equation}\label{e-gue131229II}
\begin{split}
&\ddbar_{b}Q^{(0)}(I-\Pi^{(0)})u=Q^{(1)}\ddbar_{b}u,\\
&(\mathcal{F}Q^{(1)}\ddbar_{b}u)(z,\eta)=\ddbar_z(\mathcal{F}Q^{(0)}u(z,\eta)e^{\eta\mu(z)})e^{-\eta\mu(z)}.
\end{split}
\end{equation}
As before, we put $v(z,\eta)=\mathcal{F}Q^{(0)}(I-\Pi^{(0)})u(z,\eta)e^{\eta\mu(z)}$ and set
\[\ddbar_z(\mathcal{F}Q^{(0)}(I-\Pi^{(0)})u(z,\eta)e^{\eta\mu(z)})=\ddbar_zv(z,\eta)=:g(z,\eta).\]
Then,
\begin{equation}\label{e-gue131229III}
\begin{split}
&\ddbar_zg(z,\eta)=0,\\
&\mbox{$g(z,\eta)=0$ if $\eta\notin{\rm Supp\,}\tau(\eta)$},\\
&\int\abs{g(z,\eta)}^2e^{-2\eta\mu(z)-2\abs{z}^2}d\lambda(z)<\infty,\ \ \forall\eta\in{\rm Supp\,}\tau(\eta).
\end{split}
\end{equation}

From \eqref{e-gue131226V} and H\"ormander's $L^2$ estimates \cite[Lemma 4.4.1]{Hor90},
we conclude that for every $\eta\in{\rm Supp\,}\tau(\eta)$, we can find a
$\beta_\eta(z)\in L^2_{(0,1)}(\Gamma,\eta\mu(z)+\abs{z}^2)$ such that
\begin{equation}\label{e-gue131229V}
\ddbar_z\beta_\eta(z)=g(z,\eta)
\end{equation}
and
\begin{equation}\label{e-gue131229VI}
\int\abs{\beta_\eta(z)}^2e^{-2\eta\mu(z)-2\abs{z}^2}d\lambda(z)\leq
C\int\abs{g(z,\eta)}^2e^{-2\eta\mu(z)-2\abs{z}^2}d\lambda(z),
\end{equation}
where $C>0$ is a constant independent of $\eta$, $g(z,\eta)$ and $\beta_\eta(z)$.
Moreover, since $g(z,\eta)$ is smooth, it is well-known that $\beta_\eta(z)$ can be taken to be
dependent smoothly on $\eta$ and $z$ (see the proof of~\cite[Lemma 2.1]{Bo06}).
Take $\chi(\eta)\in \cC^\infty(\Real,[0,1])$ with $\chi=0$ if $\abs{\eta}\geq1$ and
$\chi=1$ if $\abs{\eta}\leq\frac{1}{2}$. For $j=1,2,\ldots$, set $\chi_j(\eta)=\chi(\frac{\eta}{j})$.  Put
\[\alpha_j(z,t)=\frac{1}{2\pi}\int\beta_\eta(z)\chi_j(\eta)e^{-\eta\mu(z)}e^{i\eta t}d\eta\in \cC^\infty(X).\]
From \eqref{e-gue131229VI}, we have
\begin{equation}\label{e-gue140313f}\begin{split}
&\norm{\alpha_j-\alpha_k}^2\\
&=\frac{1}{4\pi^2}\int\abs{\int\beta_\eta(z)\bigr(\chi_j(\eta)-\chi_k(\eta)\bigr)
e^{-\eta\mu(z)}e^{i\eta t}(\eta)d\eta}^2e^{-2\abs{z}^2}d\lambda(z)dt\\
&\leq\frac{1}{4\pi^2}\int\int\abs{\beta_\eta(z)}^2\abs{\chi_j(\eta)-\chi_k(\eta)}^2e^{-2\eta\mu(z)-2\abs{z}^2}d\lambda(z)d\eta\\
&\leq C_0\int\int\abs{g(z,\eta)}^2\abs{\chi_j(\eta)-\chi_k(\eta)}^2e^{-2\eta\mu(z)-2\abs{z}^2}d\lambda(z)d\eta\\
&\mbox{$\To 0$ as $j,k\To\infty$},
\end{split}\end{equation}
where $C_0>0$ is a constant independent of $j$, $k$, $\beta_\eta(z)$ and $g(z,\eta)$.
Thus, $\alpha_j\To\alpha$ in $L^2(X)$, for some $\alpha\in L^2(X)$. Moreover,
we can repeat the procedure above with minor change and deduce that
\begin{equation}\label{e-gue140313fI}
\begin{split}
\norm{\alpha}^2&\leq C_0\int\int\abs{g(z,\eta)}^2e^{-2\eta\mu(z)-2\abs{z}^2}d\lambda(z)d\eta\\
&\leq C_0(2\pi)\norm{Q^{(1)}\ddbar_bu}^2\leq C_1\norm{\ddbar_bu}^2,
\end{split}
\end{equation}
where $C_0>0$ is the constant as in \eqref{e-gue140313f} and $C_1=C_0(2\pi)$. Furthermore, it is straightforward to see that
\begin{equation}\label{e-gue140313fII}
\begin{split}
\ddbar_{b}\alpha(z,t)=Q^{(1)}\ddbar_{b}u(z,t).
\end{split}
\end{equation}
From \eqref{e-gue140313fI} and \eqref{e-gue140313fII}, we conclude that
$\ddbar_{b}\alpha(z,t)=Q^{(1)}\ddbar_{b}u(z,t)$ and $\norm{\alpha}^2\leq C_1\norm{\ddbar_{b}u}^2$. Since $(I-\Pi^{(0)})Q^{(0)}u$ has the minimum $L^2$ norm of the solutions of $\ddbar_{b}f=Q^{(1)}\ddbar_{b}u(z,t)$, we conclude that
\[\norm{(I-\Pi^{(0)})Q^{(0)}u}^2=\norm{Q^{(0)}(I-\Pi^{(0)})u}^2\leq\norm{\alpha}^2\leq C_1\norm{\ddbar_{b}u}^2.\]
Theorem {\ref{t-gue140306}} follows.
\end{proof}

\begin{proof}[Proof of Theorem \ref{t-gue140306I}]
Recall that $\omega_0=-dt+\textstyle\sum^{n-1}_{j=1}(i\frac{\pr\mu}{\pr z_j}dz_j-i\frac{\pr\mu}{\pr\ol z_j}d\ol z_j)$. 
Thus,
\begin{equation}\label{e-gue140227f}\begin{split}
\Sigma^+&=\Big\{(x,\eta)\in T^*X;\, \eta=-\lambda dx_{2n-1}+
\lambda\sum^{n-1}_{j=1}\Big(\frac{\pr\mu}{\pr x_{2j}}(z)dx_{2j-1}-\frac{\pr\mu}{\pr x_{2j-1}}(z)dx_{2j}\Big), 
\lambda>0\Big\},\\
\Sigma^-&=\Big\{(x,\eta)\in T^*X;\, \eta=-\lambda dx_{2n-1}+
\lambda\sum^{n-1}_{j=1}\Big(\frac{\pr\mu}{\pr x_{2j}}(z)dx_{2j-1}-
\frac{\pr\mu}{\pr x_{2j-1}}(z)dx_{2j}\Big), \lambda<0\Big\}.
\end{split}\end{equation}
Note that
 \[Q^{(0)}(x,y)=\int e^{i\langle x-y,\eta\rangle}\tau(\eta_{2n-1})d\eta,\]
where $\tau\in \cC^\infty(\Real,[0,1])$ with $\tau=0$ on $]-\infty,\frac{1}{4}]$, 
$\tau=1$ on $[\frac{1}{2},\infty[$. From this observation and \eqref{e-gue140227f}, we conclude that
\begin{equation}\label{e-gue140227fa}
\mbox{$Q^{(0)}\equiv0$ at $\Sigma^+\cap T^*D$, $\forall D\Subset X$}.
\end{equation}
From Theorem~\ref{t-gue140305VIa}, Theorem \ref{t-gue140306}
and \eqref{e-gue140227fa}, we get Theorem~\ref{t-gue140306I}.
\end{proof}
\section{Proof of Theorem~\ref{t-gue140215}}\label{s-gue140215}

We introduce some notations from semi-classical analysis.

\begin{defn} \label{d-gue13628}
Let $W$ be an open set in $\Real^N$. We define the space of symbols
\[
S(1)=S(1;W)=\Big\{a\in \cC^\infty(W);\forall\,\alpha\in\mathbb N^N_0\:\exists\, C_\alpha>0:
\abs{\pr^\alpha_xa(x)}\leq C_\alpha\:\:\text{on $W$}\Big\}\,.
\]
If $a=a(x,k)$ depends on $k\in]1,\infty[$, we say that
$a(x,k)\in S_{{\rm loc\,}}(1;W)=S_{{\rm loc\,}}(1)$ if $\chi(x)a(x,k)$ uniformly bounded
in $S(1)$ when $k$ varies in $]1,\infty[$, for any $\chi\in
\cC^\infty_0(W)$. For $m\in\Real$, we put $S^m_{{\rm
loc}}(1;W)=S^m_{{\rm loc}}(1)=k^mS_{{\rm loc\,}}(1)$. If $a_j\in S^{m_j}_{{\rm
loc\,}}(1)$, $m_j\searrow-\infty$, we say that $a\sim
\sum^\infty_{j=0}a_j$ in $S^{m_0}_{{\rm loc\,}}(1)$ if
$a-\sum^{N_0}_{j=0}a_j\in S^{m_{N_0+1}}_{{\rm loc\,}}(1)$ for every
$N_0$.  For a given sequence $a_j$ as above, we can always find such an asymptotic sum $a$ 
and $a$ is unique up to an element in $S^{-\infty}_{{\rm loc\,}}(1)=
S^{-\infty}_{{\rm loc\,}}(1;W):=\cap_mS^m_{{\rm loc\,}}(1)$.
We say that $a(x,k)\in S^{m_0}_{{\rm loc\,}}(1)$ is a classical symbol on $W$ of order $m_0$ if
\begin{equation} \label{e-gue13628I}
\mbox{$a(x,k)\sim\sum^\infty_{j=0}k^{m_0-j}a_j(x)$ in 
$S^{m_0}_{{\rm loc\,}}(1)$},\ \ a_j(x)\in S_{{\rm loc\,}}(1),\ j=0,1\ldots.
\end{equation}
The set of all classical symbols on $W$ of order $m_0$ is denoted by 
$S^{m_0}_{{\rm loc\,},{\rm cl\,}}(1)=S^{m_0}_{{\rm loc\,},{\rm cl\,}}(1;W)$.

Let $E$ be a vector bundle over a smooth paracompact manifold $Y$. 
We extend the definitions above to the space of smooth sections of $E$ over $Y$ in the natural way 
and we write $S^m_{{\rm loc\,}}(1;Y,E)$ and $S^m_{{\rm loc\,},{\rm cl\,}}(1;Y,E)$ to denote the corresponding spaces.
\end{defn}

A $k$-dependent continuous operator
$A_k:\cC^\infty_0(W,E)\To\mathscr D'(W,F)$ is called $k$-negligible (on $W$)
if $A_k$ is smoothing and the kernel $A_k(x, y)$ of $A_k$ satisfies
$\abs{\pr^\alpha_x\pr^\beta_yA_k(x, y)}=O(k^{-N})$ locally uniformly
on every compact set in $W\times W$, for all multi-indices $\alpha$,
$\beta$ and all $N\in\mathbb N$. $A_k$ is $k$-negligible if and only if
\[A_k=O(k^{-N'}): H^s_{\rm comp\,}(W,E)\To H^{s+N}_{\rm loc\,}(W,F)\,,\]
for all $N, N'\geq0$ and $s\in\mathbb Z$. Let $C_k:\cC^\infty_0(W,E)\To\mathscr D'(W,F)$
be another $k$-dependent continuous operator. We write $A_k\equiv C_k\mod O(k^{-\infty})$ (on $W$) or
$A_k(x,y)\equiv C_k(x,y)\mod O(k^{-\infty})$ (on $W$)
if $A_k-C_k$ is $k$-negligible on $W$.

Now, we prove Theorem~\ref{t-gue140215}. We will use the same notations and assumptions in 
Theorem~\ref{t-gue140215}. Fix $p\in D$. Take local coordinates $x=(x_1,\ldots,x_{2n-1})$
defined in some small neighbourhood of $p$ such that $x(p)=0$ and $\omega_0(p)=dx_{2n-1}$. 
Since $d_y\varphi(x, y)|_{x=y}=d_y\varphi_-(x, y)|_{x=y}=\omega_0(x)$, we have 
$\frac{\pr\varphi}{\pr y_{2n-1}}(p,p)=\frac{\pr\varphi_-}{\pr y_{2n-1}}(p,p)=1$. 
From this observation and the Malgrange preparation theorem \cite[Theorem 7.57]{Hor03}, 
we conclude that in some small neighbourhood of $(p,p)$, we can find $f(x,y), f_1(x,y)\in \cC^\infty$ such that
\begin{equation}\label{e-gue140214I}
\begin{split}
&\varphi_-(x,y)=f(x,y)(y_{2n-1}+h(x,y')),\\
&\varphi(x,y)=f_1(x,y)(y_{2n-1}+h_1(x,y'))
\end{split}
\end{equation}
in some small neighbourhood of $(p,p)$, where $y'=(y_1,\ldots,y_{2n-2})$. 
For simplicity, we assume that \eqref{e-gue140214I} hold on $D\times D$. 
It is clearly that $\varphi_-(x,y)$ and $y_{2n-1}+h(x,y')$ are equivalent 
in the sense of Melin-Sj\"ostrand~\cite{MS74}, $\varphi(x,y)$ and $y_{2n-1}+h_1(x,y')$ 
are equivalent in the sense of Melin-Sj\"ostrand~\cite{MS74}, we may assume that 
$\varphi_-(x,y)=y_{2n-1}+h(x,y')$ and $\varphi(x,y)=y_{2n-1}+h_1(x,y')$. Fix $x_0\in D$. 
We are going to prove that $h(x,y')-h_1(x,y')$ vanishes to infinite order at $(x_0,x_0)$. 
Note that $S_-\circ S_-\equiv S_-$. From this observation and Lemma~\ref{l-gue140214}, 
it is straightforward to see that
\begin{equation}\label{e-gue140215f}
\int^\infty_0e^{i(y_{2n-1}+h(x,y'))t}s_-(x,y,t)dt\equiv\int^\infty_0e^{i(y_{2n-1}+h_1(x',y))t}a(x,y,t)dt\:\:\text{on $D$},
\end{equation}
where $s_-(x,y,t),\,a(x,y,t)\in S^{n-1}_{{\rm cl\,}}(D\times D\times\mathbb{R}_+,T^{*0,q}X\boxtimes T^{*0,q}X)$ 
are as in \eqref{e-gue140205III}. 
Put
\[x_0=(x^1_0,x^2_0,\ldots,x^{2n-1}_0),\ \ x'_0=(x^1_0,\ldots,x^{2n-2}_0).\]
Take $\tau\in \cC^\infty_0(\Real^{2n-1})$, $\tau_1\in \cC^\infty_0(\Real^{2n-2})$, $\chi\in \cC^\infty_0(\Real)$ 
so that $\tau=1$ near $x_0$, $\tau_1=1$ near $x'_0$,  $\chi=1$ near $x^{2n-1}_0$ 
and ${\rm Supp\,}\tau\Subset D$, ${\rm Supp\,}\tau_1\times{\rm Supp\,}\chi\Subset D'\times{\rm Supp\,}\chi\Subset D$, 
where $D'$ is an open neighbourhood of  $x'_0$ in $\Real^{2n-2}$. For each $k>0$, we consider the distributions
\begin{equation}\label{e-gue140215fII}
\begin{split}
&A_k:u\mapsto\int^\infty_0e^{i(y_{2n-1}+h(x,y'))t-iky_{2n-1}}\tau(x)s_-(x,y,t)\tau_1(y')\chi(y_{2n-1})u(y')dydt,\\
&B_k:u\mapsto\int^\infty_0e^{i(y_{2n-1}+h_1(x,y'))t-iky_{2n-1}}\tau(x)a(x,y,t)\tau_1(y')\chi(y_{2n-1})u(y')dydt,
\end{split}
\end{equation}
for $u\in \cC^\infty_0(D',T^{*0,q}X)$.
By using the stationary phase formula of Melin-Sj\"ostrand~\cite{MS74}, we can show 
that (cf.\ the proof of~\cite[Theorem 3.12]{HM12}) $A_k$ and $B_k$ are smoothing operators and
\begin{equation}\label{e-gue140215fIV}
\begin{split}
&A_k(x,y')\equiv e^{ikh(x,y')}g(x,y',k)\mod O(k^{-\infty}),\\
&B_k(x,y')\equiv e^{ikh_1(x,y')}b(x,y',k)\mod O(k^{-\infty}),\\
&g(x,y',k), b(x,y',k)\in S^{n-1}_{{\rm loc\,},{\rm cl\,}}(1;D'\times D,T^{*0,q}X\boxtimes T^{*0,q}X),\\
&\mbox{$g(x,y',k)\sim\sum^\infty_{j=0}g_j(x,y')k^{n-1-j}$ in $S^{n-1}_{{\rm loc\,}}(1;D'\times D,T^{*0,q}X\boxtimes T^{*0,q}X)$},\\
&\mbox{$b(x,y',k)\sim\sum^\infty_{j=0}b_j(x,y')k^{n-1-j}$ in $S^{n-1}_{{\rm loc\,}}(1;D'\times D,T^{*0,q}X\boxtimes T^{*0,q}X)$},\\
&g_j(x,y'), b_j(x,y')\in \cC^\infty(D\times D',T^{*0,q}_{y'}X\boxtimes T^{*0,q}_xX),\ \ j=0,1,\ldots,\\
&g_0(x_0,x'_0)\neq0,\ \ b_0(x_0,x'_0)\neq0.
\end{split}
\end{equation}
Since
\[\int^\infty_0e^{i(y_{2n-1}+h(x,y'))t}s_-(x,y,t)dt-\int^\infty_0e^{i(y_{2n-1}+h_1(x,y'))t}a(x,y,t)dt\]
is smoothing, by using integration by parts with respect to $y_{2n-1}$, it is easy to see that
$A_k-B_k\equiv0\mod O(k^{-\infty})$ (see~\cite[Section 3]{HM12}). Thus,
\begin{equation}\label{e-gue140215fV}
\begin{split}
&e^{ikh(x,y')}g(x,y',k)=e^{ikh_1(x,y')}b(x,y',k)+F_k(x,y'),\\
&F_k(x,y')\equiv0\mod O(k^{-\infty}).
\end{split}
\end{equation}
Now, we are ready to prove that $h(x,y')-h_1(x,y')$ vanishes to infinite order at $(x_0,x'_0)$. We assume that there exist $\alpha_0\in\mathbb N^{2n-1}_0$, $\beta_0\in\mathbb N^{2n-2}_0$, $\abs{\alpha_0}+\abs{\beta_0}\geq1$ such that
\[\rabs{\pr^{\alpha_0}_x\pr^{\beta_0}_{y'}(ih(x,y')-ih_1(x,y'))}_{(x_0,x'_0)}=C_{\alpha_0,\beta_0}\neq0\]
and
\[\rabs{\pr^{\alpha}_x\pr^{\beta}_{y'}(ih(x,y')-ih_1(x,y'))}_{(x_0,x'_0)}=0\ \ \mbox{if $\abs{\alpha}+\abs{\beta}<\abs{\alpha_0}+\abs{\beta_0}$}.\]
From \eqref{e-gue140215fV}, we have
\begin{equation}\label{e-gue140215fVI}
\begin{split}
&\rabs{\pr^{\alpha_0}_x\pr^{\beta_0}_y\Bigr(e^{ikh(x,y')-ikh_1(x,y')}g(x,y',k)-b(x,y,k)\Bigr)}_{(x_0,x'_0)}\\
&=-\rabs{\pr^{\alpha_0}_x\pr^{\beta_0}_y\Bigr(e^{-ikh_1(x,y')}F_k(x,y)\Bigr)}_{(x_0,x'_0)}.
\end{split}
\end{equation}
Since $h_1(x_0,x'_0)=-x^{2n-1}_0$ and $F_k(x,y')\equiv0\mod O(k^{-\infty})$, we have
\begin{equation} \label{e-gue140215fVII}
\lim_{k\To\infty}k^{-n}\rabs{\pr^{\alpha_0}_x\pr^{\beta_0}_y\Bigr(e^{-ikh_1(x,y')}F_k(x,y')\Bigr)}_{(x_0,x_0)}=0.
\end{equation}
On the other hand, we can check that
\begin{equation} \label{e-gue140215fVIII}
\begin{split}
&\lim_{k\To\infty}k^{-n}
\rabs{\pr^{\alpha_0}_x\pr^{\beta_0}_y\Bigr(e^{ikh(x,y')-ikh_1(x,y')}g(x,y',k)- b(x,y',k)\Bigr)}_{(x_0,x'_0)}\\
&=C_{\alpha_0,\beta_0}g_0(x_0,x'_0)\neq0
\end{split}
\end{equation}
since $g_0(x_0,x'_0)\neq0$. From \eqref{e-gue140215fVI}, \eqref{e-gue140215fVII} and 
\eqref{e-gue140215fVIII}, we get a contradiction. Thus, $h(x,y')-h_1(x,y')$ vanishes to infinite order at $(x_0,x'_0)$. 
Since $x_0$ is arbitrary, the theorem follows.

\subsection*{Acknowledgements} We are 
most grateful to Louis Boutet de Monvel and Johannes Sj\"ostrand for their beautiful 
theory \cite{BouSj76} and many discussions about the Szeg\H{o} kernel over the years.

\end{document}